\documentclass[journal]{IEEEtran}
\usepackage{amsmath,amssymb,amsthm,mathrsfs,amsfonts}
\usepackage{algorithmicx}
\usepackage{algorithm}
\usepackage{algpseudocode}
\usepackage{array}
\usepackage{textcomp}
\usepackage{stfloats}
\usepackage{url}
\usepackage{verbatim}
\usepackage{graphicx}
\usepackage{cite}
\usepackage{varwidth}
\usepackage{subfigure}
\newtheorem{thm}{Theorem}
\newtheorem{lem}{Lemma}
\newtheorem{prop}{Proposition}

\newtheorem{rem}{Remark}
\newtheorem{assum}{Assumption}
\newtheorem{defn}{Definition}

\newcommand{\rfig}[1]{Fig.\,\ref{#1}} 
\newcommand{\req}[1]{\eqref{#1}} 
\newcommand{\bfmath}[1]{\mbox{\boldmath $#1$}} 
\newcommand{\rtab}[1]{Table\,\ref{#1}}
\newcommand{\rlem}[1]{Lemma\,\ref{#1}}

\newcommand{\rrem}[1]{Remark\,\ref{#1}}
\newcommand{\rsec}[1]{Section\,\ref{#1}}
\newcommand{\rdef}[1]{Definition\,\ref{#1}}
\newcommand{\ras}[1]{Assumption\,\ref{#1}}
\newcommand{\ralg}[1]{Algorithm\,\ref{#1}}
\newcommand{\rapp}[1]{Appendix\,\ref{#1}}

\newcommand{\rline}[1]{line\,\ref{#1}}
\newcommand{\rprop}[1]{Proposition\,\ref{#1}}
\newcommand{\rthm}[1]{Theorem\,\ref{#1}}

\newcommand{\qedwhite}{\hfill \ensuremath{\Box}}

\newcommand{\argmax}{\mathop{\rm arg~max}\limits}

\newcommand{\diag}[1]{\mathrm{diag}({#1})}
 
\begin{document}
\title{Learning-based Event-triggered MPC with Gaussian processes under terminal constraints}
\author{Yuga Onoue, Kazumune Hashimoto, Akifumi Wachi
\thanks{Yuga Onoue is with the Graduate School of Engineering Science, Osaka University, Osaka, Japan (e-mail: yugaro.yugaro@gmail.com)}
\thanks{Kazumune Hashimoto is with the Graduate School of Engineering, Osaka University, Osaka, Japan (e-mail: hashimoto@eei.eng.osaka-u.ac.jp)}
\thanks{Akifumi Wachi is with LINE Corporation, Japan (e-mail: akifumi.wachi@linecorp.com)}
\thanks{This work is supported by JST CREST JPMJCR201, Japan and by JSPS KAKENHI Grant 21K14184.}
}
\maketitle

\begin{abstract}
    The event-triggered control strategy is capable of significantly reducing the number of control task executions without sacrificing control performance. In this paper, we propose a novel learning-based approach towards an event-triggered model predictive control (MPC) for nonlinear control systems whose dynamics is unknown apriori. In particular, the optimal control problems (OCPs) are formulated based on predictive states learned by Gaussian process (GP) regression under a terminal constraint constructed by a symbolic model. The event-triggered condition proposed in this paper is derived from the recursive feasibility so that the OCPs are solved only when an error between the predictive and the actual states exceeds a certain threshold. Based on the event-triggered condition, we analyze the convergence of the closed-loop system and show that the state of the system enters the terminal set in finite time if the uncertainty of the GP model becomes small enough. Moreover, in order to reduce the uncertainty of the GP model and increase efficiency to find the optimal solution, we provide an overall learning-based event-triggered MPC algorithm based on an iterative task. Finally, we demonstrate the proposed approach through a tracking control problem.
\end{abstract}


\subsection{Introduction}\label{introsec}
Model predictive control (MPC) \cite{camacho2013model} is an advanced framework for optimizing the performance of control systems under various constraints on states and control inputs online. In terms of performance, MPC is often superior to other control approaches since it takes predictions of the process into account, not only current states, which can compute current control actions based on future responses. Due to such an advantage, MPC is one of the most common approaches applied to various practical applications such as automated control systems \cite{li2015trajectory}, smart grids \cite{shi2018model}, chemical processes \cite{eaton1992model}, resource allocations \cite{kohler2018dynamic}, and networked control systems~\cite{grune2009networked}.

Generally, MPC requires heavy computation, especially for nonlinear systems, for solving an optimal control problem (OCP) at each sampling time step. In order to reduce such computational burden, an \textit{event-triggered} MPC, which is an extension of the MPC by improving computational efficiency without compromising control performance, has been proposed in \cite{li2014event, li2015periodic, brunner2015robust, brunner2017robust, luo2019robust, hashimoto2017event, zhan2019adaptive, wang2015event, zou2016multirate, yang2018event, li2019event,eventMPC1,eventMPC2,eventMPC3,eventMPC4,eventMPC5,eventMPC6}. In the event-triggered MPC, OCPs are solved only when event-triggered conditions designed to guarantee a certain control performance are violated. The event-triggered MPC strategy allows us to reduce the frequencies of solving the OCPs and thus it has the potential to save the computational burden. For example, in \cite{li2014event, li2015periodic}, the event-triggered conditions were derived by evaluating an error between the actual and the predictive state. Moreover, the recursive feasibility of the OCP and stability of the closed-loop system were shown based on input-to-state stability (ISS). 
In \cite{brunner2015robust, brunner2017robust, luo2019robust}, event-triggered strategies were introduced by updating the parameters of a tube-based MPC, and an average sampling rate was analyzed. Adaptive event-triggered MPCs were proposed in \cite{hashimoto2017event, zhan2019adaptive}, in which the event-triggered conditions were evaluated intermittently with the intervals selected adaptively by incorporating periodic event-triggered control frameworks. Event-triggered MPCs were studied for networked control systems with packet losses, transmission delays, or quantization errors in \cite{wang2015event, zou2016multirate, yang2018event, li2019event}. 

In the aforecited event-triggered MPC frameworks, it is typically assumed that the models for predicting the states and solving the OCPs are known \textit{apriori}. However, even though such models could often be obtained by, e.g., first-principle modeling, it is sometimes difficult or even impossible in practice to obtain them due to high nonlinearities and complexities of the dynamics. This fact implies that the control performance may be degraded when applying event-triggered MPC to actual systems because the models that play an essential role in the process prediction cannot be accurately obtained. 
In order to deal with such problems, 
 \textit{learning-based} approaches, in which the underlying dynamics is learned based on the training data, are proposed for the MPC design (see, e.g., \cite{rasmussen2003gaussian, deisenroth2013gaussian,kocijan2016modelling}). 
In the learning-based approach, {Gaussian process} (GP) regression \cite{rasmussen2003gaussian} has been employed to learn the underlying dynamics of the control system. 
Furthermore, various theoretical analysis on an error (or regret) bound between the underlying function and the GP model have been provided in recent years (see, e.g. \cite{srinivas2012information, hashimoto2020learning}). 

Motivated by the aforementioned background, in this paper we propose a novel learning-based approach towards an event-triggered MPC for nonlinear control systems whose dynamics is unknown \textit{apriori}. 
We in particular adopt the GP regression to learn the underlying dynamics and investigate the convergence of the closed-loop system under the proposed event-triggered MPC algorithm. 
First, the OCPs are formulated based on state predictions by the GP regression under a {terminal constraint}, in which the predictive state at the terminal time is constrained to belong to a prescribed terminal region. The OCP with the terminal constraint is one of the common schemes to guarantee stability (see, e.g., \cite{Chen1998a}). 
Then, we provide an event-triggered condition so that the OCPs are solved only when an error between the predictive and the actual state exceeds a certain threshold (for details, see \rsec{evtrigmpcsec}). 
Moreover, in order to reduce the uncertainty of the GP model and increase efficiency to find the optimal solution, we provide an overall algorithm of learning-based event-triggered MPC based on an {iterative task} 
which consists of an update phase to learn the unknown system and an execution phase to implement the event-triggered MPC and collect training data (for details, see \rsec{etmpcalgsec}). 
The event-triggered condition proposed in this paper is derived based on recursive feasibility (for details, see \rsec{feasibilitysec}), which signifies that, under an assumption that the OCP at the current time step is feasible, the OCP for the next time step is shown to be also feasible. 
Moreover, convergence is analyzed using the concept of decreasing horizon strategy \cite{hashimoto2017event}, in which the prediction horizon decreases as the OCPs are solved online (for details, see \rsec{stabilitysec}). 

Summarizing, the main contribution of this paper is given as follows:
\begin{enumerate}
    \item We propose to formulate a learning-based event-triggered MPC with GP regression under a terminal constraint. 
    The event-triggered condition is given such that the OCPs are solved only when an error between the predictive and actual states exceeds a given threshold. 
    \item Based on the event-triggered condition, we analyze the feasibility and convergence of the closed-loop system. In particular, we show that, as the uncertainty of the GP model becomes smaller, recursive feasibility and the finite-time convergence to the terminal region are achieved.
\end{enumerate}

\smallskip
\textit{Related works}: 
Our approach is related to several applications of machine learning techniques to design optimal control policies. In the following, we describe how our proposed approach differs from the related works to highlight our main contributions.

The results of designing optimal control policies based on machine learning techniques can be broadly divided into two categories depending on whether models are constructed; 
see model-free approaches \cite{zhong2014event, baumann2018deep, vamvoudakis2018model, funk2021learning} and model-based approaches \cite{deisenroth2013gaussian, rosolia2017learning, kocijan2004gaussian, aswani2013provably, yoo2017event, yoo2019event, manzano2020robust, maddalena2021kpc, hashimoto2020learnings, maiworm2018stability, hewing2018cautious, hewing2019cautious, maiworm2021online, koller2018learning}. 
Our approach is particularly related to the second category since controllers are designed from the (GP) model estimated by the training data. Among the model-based approaches, applications of machine learning to the event-triggered MPC for compensating model uncertainties were proposed in \cite{yoo2017event, yoo2019event}. 
In contrast to these methods, we employ the GP regression as a machine learning technique to learn dynamical systems. Furthermore, we provide a detailed algorithm for constructing the terminal set for unknown systems. 
In particular, we provide a method to compute a general terminal set (e.g., the terminal set is not necessary to be an ellipsoidal set and can be non-convex) by applying a novel technique of symbolic abstraction \cite{hashimoto2020learning} and formulating an augmentation of the notation.
The authors in \cite{maddalena2021kpc} proposed a learning-based predictive control based on a kernel ridge regression (KRR) with deterministic safety guarantees. 
The approach presented in this paper is different and advantageous from \cite{maddalena2021kpc} in the following sense. 
First, we provide detailed, analytical error bounds for multi-step predictions under the same assumption (i.e., the unknown system lies in the RKHS).  
In particular, we show that the deterministic error bounds for multi-step ahead predictions can be derived recursively using the SE kernel and the corresponding kernel metric (although \cite{maddalena2021kpc} provided error bounds on the multi-step ahead predictions, they did not analytically investigate them in a detailed way). 
Furthermore, we show that the recursive feasibility of the OCPs is guaranteed if the error bounds on the unknown function become small enough. 
In addition, we provide a way to construct terminal (contractive) set, and show that the state trajectory enters the terminal set in finite time and stays therein for all future times. 


Optimal control frameworks that incorporate the GP regression were investigated in \cite{hewing2018cautious, hewing2019cautious,deisenroth2013gaussian, kocijan2004gaussian, hashimoto2020learnings, maiworm2018stability, maiworm2021online}. For instance, the authors in \cite{hewing2018cautious, hewing2019cautious} formulated a chance-constrained MPC in terms of probabilistic reachable sets, and a sparse GP was introduced as an approximation technique to reduce the computational burden of solving the OCP. 
In contrast the existing works of \cite{hewing2018cautious,hewing2019cautious,deisenroth2013gaussian, hashimoto2020learnings}, the approach presented in this paper allows us to analyze the recursive feasibility of the OCPs and the convergence of the closed-loop system. 
In \cite{maiworm2018stability, maiworm2021online}, output feedback MPC was formulated based without using a terminal constraint. The robust stability was established in the form of ISS with respect to the modeling error between the GP model and the underlying system. Moreover, robust stability was derived by assuming the optimal cost function as a candidate Lyapunov function. Our approach differs from \cite{maiworm2018stability, maiworm2021online} in the following sense. First, we provide the optimal control policy based on an event-triggered MPC framework. Our approach solves the OCPs only when needed according to a well-designed event-triggered condition to reduce the frequency of solving the OCPs. Second, we focus on the learning-based MPC formulation \textit{under a terminal constraint} and investigate the recursive feasibility and the convergence of the closed-loop system. 
Since the MPC strategy under the terminal constraint is one of the other common techniques, we argue that it is worth investigating its theoretical analysis, in particular on the recursive feasibility and convergence. 
Third, while \cite{maiworm2018stability, maiworm2021online} analyzed the stability in a probabilistic sense, the approach presented in this paper allows us to analyze convergence (and feasibility) in a deterministic sense based on the error bound presented in \cite{hashimoto2020learnings}. Lastly, the convergence is established via decreasing horizon strategy (instead of regarding the optimal cost as the Lyapunov function candidate), and it is shown via a dual-mode strategy that the state trajectory enters the terminal set in finite time and stays therein for all future times.

\smallskip
\textit{Notation}: We denote by $\mathbb{N}, \mathbb{N}_{\ge a}, \mathbb{N}_{>a}$, and $\mathbb{N}_{a:b}$ the sets of integers, intergers larger than or equal to $a$, integers lager than $a$, and integers from $a$ to $b$, respectively. 
Let $\mathbb{R}, \mathbb{R}_{\ge a}$, and $\mathbb{R}_{>a}$ be the sets of reals, reals larger than or equal to $a$, and reals larger than $a$, respectively. 
Given $a, b \in \mathbb{R}$ with $a \le b$, let $[a, b]$ be the interval set from $a$ to $b$. 
Given $a \in \mathbb{R}$ and $b \in \mathbb{R}_{\ge 0}$, we let $[a \pm b] = [a - b, a + b]$. 
Given $\mathcal{A}, \mathcal{B} \subseteq \mathbb{R}^n$ with $\mathcal{A} \subseteq \mathcal{B}$, let $[\mathcal{A}, \mathcal{B}]$ be a set between $\mathcal{A}$ and $\mathcal{B}$. 
Given two vectors $a, b \in \mathbb{R}^n$, the {Hadamard product} of $a$ and $b$ (i.e., the pointwise product of $a$ and $b$) is denoted by $a \odot b$. 
For two given sets $\mathcal{A}, \mathcal{B} \subseteq \mathbb{R}^{n}$, denote by $\mathcal{A} \oplus \mathcal{B}$ the {Minkowski sum} $\mathcal{A} \oplus \mathcal{B} = \{a + b \in \mathbb{R}^n| a \in \mathcal{A}, b\in \mathcal{B}\}$ and $\mathcal{A} \times \mathcal{B}$ the {Cartesian product} $\mathcal{A} \times \mathcal{B} = \{(a, b) | a \in \mathcal{A}, b \in \mathcal{B}\}$.
Let $0_n, 1_n \in \mathbb{R}^n$ be the all-zeros vector and all-ones vector, respectively. 
For a given ${x} \in \mathbb{R}^n$ and a positive definite matrix $P \in \mathbb{R}^{n \times n}$, we denote by $\|{x}\|$ the Euclidean norm of ${x}$ and $\|{x}\|_{P} \in \mathbb{R}^{n \times n}$ the weighted norm of ${x}$ for the matrix $P$, i.e., $\|{x}\|_{P} = \sqrt{{x}^\top P {x}}$. We denote a metric by a function ${\rm d} : \mathbb{R}^n \times \mathbb{R}^n \to \mathbb{R}_{\ge 0}$ satisfying ${\rm d}({x}, {x}') = 0$ iff ${x} = {x}'$, ${\rm d}({x}, {x}') + {\rm d}({x}', {x}'') \ge {\rm d}({x}, {x}'')$, and ${\rm d}({x}, {x}') = {\rm d}({x}', {x})$ with ${x}, {x}', {x}'' \in \mathbb{R}^n$. 
Given $\mathcal{X} \subseteq \mathbb{R}^n$ and ${ \eta} = [\eta_1, \ldots, \eta_n]^\top \in \mathbb{R}_{> 0}^n$, we denote by $[\mathcal{X}]_{ \eta} \subseteq \mathbb{R}^n$ the lattice in $\mathcal{X}$ with the quantization prameter ${ \eta}$, i.e., $[\mathcal{X}]_{ \eta} = \{{x} \in \mathcal{X} : x_i = \mathsf{z}_i \frac{2}{\sqrt{n}} \eta_i, \mathsf{z}_i \in \mathbb{N}, i = 1, 2, \ldots, n\}$, where $x_i \in \mathbb{R}$ is the $i$-th element of ${x}$. 
For given $\gamma \in \mathbb{R}_{\ge 0}$, $x \in \mathbb{R}^n$, and a metric ${\rm d}(\cdot, \cdot)$, we denote by $\mathcal{B}_{{\rm d}} ({x}; \gamma) \subseteq \mathbb{R}^n$ and $\widetilde{\mathcal{B}}_{{\rm d}} ({x}; \gamma) \subseteq [\mathbb{R}^n]_{\eta_x}$ the ball set of radius $\gamma$ centered at $x$ in the continuous and discrete space such that $\mathcal{B}_{{\rm d}}({x}; \gamma) = \{{x}' \in \mathbb{R}^n: {\rm d}({x}, {x}') \le \gamma \}$ and $\widetilde{\mathcal{B}}_{{\rm d}}({x}; \gamma) = \{{x}' \in [\mathbb{R}^n]_{\eta_x}: {\rm d}({x}, {x}') \le \gamma \}$, respectively.
Given $\mathcal{X} \subseteq \mathbb{R}^n$, $\mathcal{X}_\mathsf{q} \subseteq [\mathbb{R}^n]_{\eta_x}$, $\gamma \in \mathbb{R}_{\ge 0}$, and a metric ${\rm d}(\cdot, \cdot)$,  we denote by $\mathsf{Int}_{\rm d}(\mathcal{X}; \gamma) \subseteq \mathbb{R}^n $ and $\widetilde{\mathsf{Int}}_{\rm d}(\mathcal{X}_\mathsf{q}; \gamma) \subseteq [\mathbb{R}^n]_{\eta_x}$ the set of all the states that are $\gamma$ or more internal away from the boundary of $\mathcal{X}$ and $\mathcal{X}_\mathsf{q}$ w.r.t the metric ${\rm d}(\cdot, \cdot)$, such that $\mathsf{Int}_{\rm d}(\mathcal{X}; \gamma) = \{{x} \in \mathcal{X}: \mathcal{B}_{\rm d}({x}; \gamma) \subseteq \mathcal{X}\}$ and $\widetilde{\mathsf{Int}}_{\rm d}(\mathcal{X}_\mathsf{q}; \gamma) = \{{x}_\mathsf{q} \in \mathcal{X}_\mathsf{q}: \widetilde{\mathcal{B}}_{\rm d}({x}_\mathsf{q}; \gamma) \subseteq \mathcal{X}_\mathsf{q}\}$, respectively.
Given $\mathcal{X} \subseteq \mathbb{R}^n$, $\mathcal{X}_\mathsf{q} \subseteq [\mathbb{R}^n]_{\eta_x}$, $\gamma \in \mathbb{R}_{\ge 0}$, and a metric ${\rm d}(\cdot, \cdot)$, we denote by $\mathsf{Out}_{\rm d}(\mathcal{X}; \gamma) \subseteq \mathbb{R}^n$ and $\widetilde{\mathsf{Out}}_{\rm d}(\mathcal{X}_\mathsf{q}; \gamma) \subseteq [\mathbb{R}^n]_{\eta_x}$ the set of all the states in $\mathcal{X}$ and $\mathcal{X}_\mathsf{q}$ or all the states that are $\gamma$ or less external away from the boundary of  $\mathcal{X}$ and $\mathcal{X}_\mathsf{q}$ w.r.t the metric ${\rm d}(\cdot, \cdot)$, such that $\mathsf{Out}_{\rm d}(\mathcal{X}; \gamma) = \{{x}' \in \mathbb{R}^{n} : \exists x \in \mathcal{X},  x' \in \mathcal{B}_{\rm d}({x}; \gamma)\}$ and $\widetilde{\mathsf{Out}}_{\rm d}(\mathcal{X}_\mathsf{q}; \gamma) = \{{x}_\mathsf{q}' \in [\mathbb{R}^{n}]_{\eta_x} : \exists {x}_\mathsf{q} \in \mathcal{X}_\mathsf{q}, x_\mathsf{q}' \in \widetilde{\mathcal{B}}_{\rm d}({x}_\mathsf{q}; \gamma) \}$, respectively.

\section{Preliminaries}\label{prelisec}
In this section, we recall some concepts of Gaussian process (GP) regression \cite{rasmussen2003gaussian} and define a notation of a transition system.

\subsection{Gaussian process regression}\label{gprsec}
This section reviews some basic concepts of \textit{Gaussian process} (GP) regression \cite{rasmussen2003gaussian}, a supervised machine learning framework that is used to solve the regression problem. 
Let $f: \mathbb{R}^n \to \mathbb{R}$ be an unknown function and given a set of $N \in \mathbb{N}_{>0}$ input-output training data $\mathcal{D}_N = \{X_N, Y_N\}$ with $X_N = [x_1, \ldots, x_N]^\top$ and $Y_N = [y_1, \ldots, y_N]$ by $y_t = f(x_t) + w_t$ for $t \in \mathbb{N}_{1: N}$, where $x_t \in \mathbb{R}^n$ is the training input, $y_t \in \mathbb{R}$ is the training target, and $w_t \sim \mathcal{N}(0, \sigma_w^2)$ is the Gaussian distributed white noise. In the GP regression, the prediction is fully specified by the mean function $\mu : \mathbb{R}^n \to \mathbb{R}$ and the covariance function $\mathsf{k}: \mathbb{R}^n \times \mathbb{R}^n \to \mathbb{R}$ that represents the covariance between the data points. In this paper, we choose a prior mean function as $\mu \equiv 0$ and the covariance function as a squared exponential (SE) kernel, which is defined by
\begin{align}\label{sekernel}
\mathsf{k} (x_{t}, x_{t'}) = \alpha^2 \exp \left(-\frac{1}{2}\|{x}_t - {x}_{t'}\|_{ \Lambda^{-1}}^2\right),  
\end{align}
where $\alpha^2 \in \mathbb{R}_{>0}$ is the variance of the function $f$ and $\Lambda = \diag{\lambda_1, \ldots, \lambda_n}$ characterizes the length-scale $\lambda_i$ ($i \in \mathbb{N}_{1: n}$) of the input data. The posterior distribution of the unknown function $f$ at the new data point $x \in \mathbb{R}^n$ is then provided as a Gaussian distribution $\mathcal{N}(\mu_{f}(x; \mathcal{D}_N), \sigma_f(x; \mathcal{D}_N))$ with the mean function $\mu_{f}(x; \mathcal{D}_N)$ and the variance function $\sigma^2_f(x; \mathcal{D}_N)$ by
\begin{align}
\mu_{f}({x}; \mathcal{D}_{N}) & = \mathsf{k}_{N}^{\star \top}({x})(K + \sigma_{w}^2 I)^{-1}Y_{N}, \\ 
\sigma_{f}^2({x}; \mathcal{D}_{N}) & = \mathsf{k}({x}, {x}) - \mathsf{k}_{N}^{\star\top}({x})(K + \sigma_{w}^2 I)^{-1}\mathsf{k}_{N}^\star({x}),
\end{align}
where $\mathsf{k}^\star _N(z) = [\mathsf{k} (x, x_{1}), \ldots, \mathsf{k} (x, x_{N})]^\top$, $K \in \mathbb{R}^{n \times n}$ is the kernel matrix with entries $K_{t, t'} = \mathsf{k}(x_t, x_{t'})$ for $t, t' \in \mathbb{N}_{1: N}$, and $I$ is the identity matrix of appropriate dimension. In the case of multiple outputs, the GP regression is individually applied to each element of the outputs with the covariance between the different outputs set to zero.

\subsection{Transition system}\label{transitionsystemsec}
The notion of a transition system is formalized as follows:
\begin{defn}\label{transdef}
\normalfont
A transition system is a triple $\Sigma = (\mathcal{X}, \mathcal{U}, G)$ consisting of:
\begin{itemize}
    \item $\mathcal{X}$ is a set of states;
    \item $\mathcal{U}$ is a set of inputs; 
    \item $G: \mathcal{X} \times \mathcal{U} \to 2^{\mathcal{X}}$ is a transition map.\qedwhite
\end{itemize}
\end{defn}
A transition of the system from a state ${x} \in \mathcal{X}$ to ${x}' \in \mathcal{X}$ by applying a control input ${u} \in \mathcal{U}$ (i.e., $x' \in G(x, u)$) is also denoted by ${x} \xrightarrow{{u}} {x}'$ and the state ${x}'$ is called an ${u}$-successor of the state ${x}$. That is, we denote by the transition map $G({x}, {u})$ the set of ${u}$-successor of a state ${x}$. Moreover, we denote by $\mathcal{U}({x})$ the set of inputs ${u} \in \mathcal{U}$, for which $G({x}, {u}) \neq \varnothing$.

\section{Problem formulation}\label{problemsec}
In this section, we provide a system description and the goal of this paper.

\subsection{System description}
We consider the following nonlinear discrete-time system:
\begin{align}\label{realdynamics}
&{x}({t + 1}) = {f}({x}(t), {u}(t)) + {w}(t), \notag \\
&\ {u}(t) \in \mathcal{U}, \ {w}(t) \in \mathcal{W}, \ x(0) \in \mathcal{X}_{\rm init}
\end{align}
for all $t \in \mathbb{N}_{\ge 0}$, where ${x}(t) \in \mathbb{R}^{n_x}$ is the state, ${u}(t) \in \mathbb{R}^{n_u}$ is the control input, ${w}(t) \in \mathbb{R}^{n_x}$ is the additive noise, and ${f} : \mathbb{R}^{n_x} \times \mathbb{R}^{n_u} \to \mathbb{R}^{n_x}$ is the \textit{unknown} transition function. As shown in \req{realdynamics}, control input, additive noise, and the initial state are restricted to lie in the given sets $\mathcal{X} \subset \mathbb{R}^{n_x}$, $\mathcal{U} \subset \mathbb{R}^{n_u}$, $\mathcal{W} \subset \mathbb{R}^{n_x}$, and $\mathcal{X}_{\rm init} \subset \mathbb{R}^{n_x}$, respectively, where $\mathcal{X}_{\rm init}$ and $\mathcal{U}$ are compact and $\mathcal{W}$ is given by $\mathcal{W} = \mathcal{W}_1 \times \ldots \times \mathcal{W}_{n_x}$ with $\mathcal{W}_i = \{{w}_i \in \mathbb{R} : |w_i| \le \sigma_{w_i}\}$ for all $i \in \mathbb{N}_{1:n_x}$ for given $\sigma_{w_i} > 0$, $i \in \mathbb{N}_{1:n_x}$.

Let $f_i$, $i\in \mathbb{N}_{1: n_x}$ denote the $i$-th element of $f$, i.e., $f = [f_1, f_2, \ldots, f_{n_x}]^\mathsf{T}$. In addition, for each $i\in \mathbb{N}_{1: n_x}$, let $\mathsf{k}_i : \mathbb{R}^{n_x + n_u}\times \mathbb{R}^{n_x + n_u} \rightarrow \mathbb{R}_{\geq 0}$ denote a given squared exponential (SE) kernel function (see  \req{sekernel}) with the given hyperparameters denoted as $\alpha_{i} \in \mathbb{R}_{> 0}$ and $\Lambda_{i} = \diag{\{\lambda_{i,\ell}\}^{n_x+n_u} _{\ell=1}}\in \mathbb{R}_{>0}^{(n_x + n_u) \times (n_x + n_u)}$. In this paper, we assume that each $f_i$ lies in a reproducing kernel Hilbert space (RKHS) corresponding to $\mathsf{k}_i$:

\begin{assum}\label{rkhsas}
\normalfont
For each $i\in \mathbb{N}_{1: n_x}$, let $\mathcal{H}_{\mathsf{k}_i}$ be the RKHS corresponding to the kernel $\mathsf{k}_i$ with the induced norm denoted as $\|\cdot \|_{\mathsf{k}_i}$. 
Then, for all $i\in \mathbb{N}_{1: n_x}$, $f_i \in \mathcal{H}_{\mathsf{k}_i}$. 
Moreover, $\|f_i\|_{\mathsf{k}_i} \le b_i$ for all $i\in \mathbb{N}_{1: n_x}$. 
\qedwhite
\end{assum}
\ras{rkhsas} implies that each $f_i$, $i \in \mathbb{N}_{1: n_x}$ is characterized by $f_i({z}) = \Sigma_{n = 1}^\infty \alpha_n \mathsf{k}_i({z}, {z}_n)$ with ${z} = [{x}^\top, {u}^\top]^\top$, where ${z}_n = [{x}_n^\top, {u}_n^\top]^\top \in \mathbb{R}^{n_x + n_u}$, $n \in \mathbb{N}_{>0}$ are the representer points, $\alpha_n \in \mathbb{R}$, $n \in \mathbb{N}_{>0}$ are the weight parameters such that the induced norm $\|f_i\|^2 _{\mathsf{k}_i} = \sum_{n=1}^\infty \sum_{n'=1}^\infty \alpha_n \alpha_{n'} \mathsf{k}_i ({ z}_n, { z}_{n'})$ is bounded. 
Note that \ras{rkhsas} is one of the common assumptions in learning-based control with the GP regression (see, e.g., \cite{berkenkamp2016safe}), and several ways to compute $b_i$ have recently proposed (see, e.g., \cite{maddalena2020deterministic,hashimoto2020learning,jackson2021}). 
Although we must restrict here the class of the kernel function by the SE kernel, the SE kernel fulfills some nice properties that are important and worth to be employed. 
For instance, the SE kernel is known to satisfy a \textit{{universal approximation property}}, which means that the RKHS corresponding to the SE kernel is dense in the space of all continuous functions over an arbitrary compact set, i.e., {any} continuous function within any compact set can be estimated arbitrarily well by a function that lies in the RKHS; for details, see \cite{universalkernel}.


\subsection{Goal of this paper}\label{goalsec}
The goal of this paper is to formulate an event-triggered model predictive control (MPC) strategy, where optimal control problems (OCPs) are solved \textit{only when} they are needed according to a well-designed event-triggered condition. 
Let $t_k$, $k \in \mathbb{N}_{\ge 0}$ with $t_0 = 0$ be the time instants when the OCPs (which will be formalized in details in \rsec{ocpsec}) are solved. In the event-triggered control, these time instants are computed as follows: $t_0 = 0$, and for all $k \in \mathbb{N}_{>0}$, 
\begin{align}
    t_{k+1} = \min \{ t_k + j: L(x(t_k + j)) >0 \}, 
\end{align}
where $L(\cdot)$ is a given function that characterizes the event-triggered condition. 
The control inputs are given by 
\begin{align}\label{appliedinput}
    u(t_k + j) = u^* ( j | t_k)
\end{align}
for all $ j \in \mathbb{N}_{0 : t_{k+1} -t_k-1}$, where $u^* (j | t_k)$ is a portion of the optimal control inputs obtained by solving the OCP at $t_k$. \req{appliedinput} implies that the optimal control inputs obtained by the OCP at $t_k$ are applied until the next update time $t_{k+1}$. Detailed formalization of the OCP as well as the event-triggered condition are given in the next section. 

\section{Event-triggered MPC strategy}\label{etmpcstrategysec}
In this section, we formulate an event-triggered MPC. In \rsec{learndsec}, we provide how to learn the unknown function $f$ and then define an optimal control problem in \rsec{ocpsec}. In \rsec{evtrigmpcsec}, we provide an event-triggered condition and finally propose an overall event-triggered MPC algorithm in \rsec{etmpcalgsec}.

\subsection{Learning dynamics with Gaussian processes}\label{learndsec}
In this paper, we learn each element of ${f}$, i.e., $f_i$, $i \in \mathbb{N}_{1: n_x}$ by the Gaussian Process (GP) regression. 
Let $\mathcal{D}_{N, i} = \{Z_N, Y_{N, i}\}$ for $i \in \mathbb{N}_{1: n_x}$ be a set of $N \in \mathbb{N}_{> 0}$ measured input-output training data, where $Z_N = [{z}_1, \ldots, {z}_N]^\top \in \mathbb{R}^{N \times (n_x + n_u)}$ is the input dataset with ${z}_t = [{x}^\top(t), {u}^\top(t) ]^\top \in \mathbb{R}^{n_x + n_u}$ for $t \in \mathbb{N}_{1: N}$ and $Y_{N, i} = [y_{1, i}, \ldots, y_{N, i}]^\top \in \mathbb{R}^N$ is the output dataset with $y_{t, i} = x_i(t + 1)\in \mathbb{R}$ for $t \in \mathbb{N}_{1: N}$ ($x_i(t)$ denotes the $i$-th element of ${x}(t)$). 
The GP model of $f_i$ for the new point ${z} = [{x}^\top, {u}^\top]^\top$ are given by  $\mathcal{N}(\mu_{f_i}(x, u; \mathcal{D}_{N, i}), \sigma_{f_i}^2(x, u; \mathcal{D}_{N, i}))$, 
where $\mu_{f_i}(x, u; \mathcal{D}_{N, i})$ and $\sigma_{f_i}^2(x, u; \mathcal{D}_{N, i})$ are the posterior mean and the variance given by
\begin{align}
\mu_{f_i}(x, u; \mathcal{D}_{N, i}) & = \mathsf{k}_{N, i}^{\star \top}({z})(K_{i} + \sigma_{w_i}^2 I)^{-1}Y_{N, i}\label{gpmean} \\ 
\sigma_{f_i}^2(x, u; \mathcal{D}_{N, i}) & = \mathsf{k}_{i}({z}, {z}) - \mathsf{k}_{N, i}^{\star\top}({z})(K_{i} + \sigma_{w_i}^2 I)^{-1}\mathsf{k}_{N, i}^\star({z})\label{gpvar}
\end{align}
for all $i \in \mathbb{N}_{1: n_x}$ with $\mathsf{k}_{N, i}^{\star}({z}) = [\mathsf{k}_{i}({z}, {z}_1), \ldots, \mathsf{k}_{i}({z}, {z}_N)]^\top \in \mathbb{R}_{\ge 0}^N$. 

Using \req{gpmean} and \req{gpvar}, together with the assumption that $f_i \in \mathcal{H}_{\mathsf{k}_i}$ (\ras{rkhsas}), we can compute a deterministic error bound on $f_i$, $i \in \mathbb{N}_{1: n_x}$ as follows (see, e.g., \cite{hashimoto2020learning}): 
\begin{lem}\label{errorbound}
\normalfont
Given the training dataset $\mathcal{D}_{N, i} = \{Z_N, Y_{N, i}\}$ for each $f_i$, $i \in \mathbb{N}_{1: n_x}$, and let \ras{rkhsas} hold. Then, for all ${x} \in \mathbb{R}^{n_x}$, ${u} \in \mathbb{R}^{n_u}$, and $i \in \mathbb{N}_{1: n_x}$, it follows that $f_i(x, u)\in \mathcal{F}_i(x, u; \mathcal{D}_{N, i})$, where 
\begin{align}\label{intervalset}
&\!\!\!\mathcal{F}_i(x, u; \mathcal{D}_{N, i}) = [\mu_{f_i}(x, u; \mathcal{D}_{N, i}) \pm \beta_{N, i} \cdot \sigma_{f_i}(x, u; \mathcal{D}_{N, i})], 
\end{align}
with $\beta_{N, i} = \sqrt{b_i^2 - Y_{N, i}^\top (K_i + \sigma_{w_i}^2 I)^{-1}Y_{N, i} + N}$.
\qedwhite
\end{lem}
For the proof of \rlem{errorbound}, see Appendix B in \cite{hashimoto2020learning}. 
Based on \rlem{errorbound}, let $\delta_{N, i}(x, u) \in \mathbb{R}_{> 0}$ be an uncertainty for an arbitrary point $x \in \mathbb{R}^{n_x}$ and $u \in \mathbb{R}^{n_u}$ denoted by
\begin{align}\label{ndeltadef}
\delta_{N, i}(x, u) = \beta_{N, i}  \sigma_{f_i}(x, u; \mathcal{D}_{N, i}) + \sigma_{w_i},
\end{align}
for all $i \in \mathbb{N}_{1: n_x}$. 
Moreover, let $d_i: \mathbb{R}^{n_x} \times \mathbb{R}^{n_x} \rightarrow \mathbb{R}_{\ge 0}$ for all $i \in \mathbb{N}_{1: n_x}$ denote a \textit{kernel metric} \cite{hearst1998support} (with respect to the state-space $\mathbb{R}^{n_x}$) given by 
\begin{align}\label{kernelmetric}
     d_i (&{x}_1, x_2) = \sqrt{\mathsf{k}_{x, i}({x}_1, {x}_1) - 2 \mathsf{k}_{x, i}({x}_1, {x}_2) + \mathsf{k}_{x, i}({x}_2, {x}_2)},
\end{align}
for all $x_1, x_2 \in \mathbb{R}^{n_x}$, where
\begin{align}\label{kx}
    \mathsf{k}_{x, i}(x_1, x_2) = \alpha_{i}^2 \exp\left(-\frac{1}{2}\|x_1 -x_2\|_{\Lambda_{x, i}^{-1}}\right). 
\end{align}
In \req{kx}, $\Lambda_{x, i}$ is a diagonal matrix that collects the first $n_x$ diagonal entries of $\Lambda_{i}$, i.e., $\Lambda_{x, i} = \diag{\lambda_{i, 1}, \ldots, \lambda_{i, n_x}}$ (in other words, $\mathsf{k}_{x, i}$ denotes the SE kernel in the state space $\mathbb{R}^{n_x}$). 
 Using the above notations, we can show that $f_i$ has a continuity property: 
\begin{lem}\label{continuity}
\normalfont
Let \ras{rkhsas} hold. 
Then, $|f_i({x}_1, {u}) - f_i({x}_2, {u}) | \le b_i d_i({x}_1, {x}_2)$ for all ${x}_1, {x}_2 \in \mathbb{R}^{n_x}$ and $u \in \mathbb{R}^{n_u}$.
\qedwhite
\end{lem}
For the proof of \rlem{continuity}, see \rapp{continuityproof}. 
Moreover, the following inclusion relation based on the kernel metric will also be useful in the theoretical analysis provided later in this paper. 
\begin{lem}\label{inoutlem}
\normalfont
Given $r_i \in [0, \sqrt{2}\alpha_i]$, $i \in \mathbb{N}_{1: n_x}$, let $\mathcal{X}_1, \mathcal{X}_2 \subset \mathbb{R}^{n_x}$ be any two sets satisfying 
\begin{align}\label{inoutlemeq}
    \mathcal{X}_1 \subseteq \bigcap_{i = 1}^{n_x}\mathsf{Int}_{d_i}(\mathcal{X}_2; r_i). 
\end{align}
Then, for all $x_1, x_2 \in \mathbb{R}^{n_x}$ with $d_i(x_1, x_2) \le r_i$ for all $i \in \mathbb{N}_{1: n_x}$, it follows that $x_1 \in \mathcal{X}_1 \implies x_2 \in \mathcal{X}_2$. 
\qedwhite
\end{lem}
For the proof of \rlem{inoutlem}, see \rapp{inoutlemproof}. Roughly speaking, \req{inoutlemeq} means that $\mathcal{X}_1$ is contained in the set of all the states that are $r_i$ (or more) interior from the boundary of $\mathcal{X}_2$ with respect to the kernel metrics $d_i$ for all $i \in \mathbb{N}_{1: n_x}$. 

\subsection{Optimal control problem}\label{ocpsec}

Let us now formulate an optimal control problem. Let $\mathcal{D}_{N, i} = \{Z_N, Y_{N, i}\}$ for all $i \in \mathbb{N}_{1: n_x}$ be given training data set as defined in \rsec{learndsec}. Moreover, as described in \rsec{goalsec}, let $t_k$, $k \in \mathbb{N}_{\ge 0}$ with $t_0 = 0$ be the time instants when the OCPs are solved. Given $x(t_k)$ and a sequence of the control inputs $u(j | t_k) \in \mathcal{U}$, $j \in \mathbb{N}_{\geq 0}$, let $\hat{{x}}(j | t_k) \in \mathbb{R}^{n_x}$ for all $j \in \mathbb{N}_{\ge 0}$ with $\hat{{x}}(0 | t_k) = x(t_k)$ denote the \textit{predictive states} starting from $x(t_k)$ by applying $u(j | t_k)$, $j \in \mathbb{N}_{\geq 0}$ based on the GP mean \req{gpmean}: 
\begin{align}
\hat{x}_i(j + 1 | t_k) &= \hat{{f}}_i(\hat{{x}}(j | t_k), {{u}}(j | t); \mathcal{D}_{N, i}) \notag \\
&:= \mu_{f_i}(\hat{x} (j | t_k), {u}(j | t_k); \mathcal{D}_{N, i}),\  i \in \mathbb{N}_{1:n_x}
\end{align}
for all $j \in \mathbb{N}_{\ge 0}$, where $\hat{x}_i(j | t_k)$ is the $i$-th element of $\hat{x}(j | t_k)$. The overall prediction model of $f$ is then defined by 
\begin{align}\label{gpmodel}
\hat{{x}}(j + 1 | t_k) &= \hat{{f}}(\hat{{x}}(j | t_k), {{u}}(j | t_k); \mathcal{D}_{N}) \notag \\
&:= {\mu}_{f}(\hat{x} (j | t_k), {u}(j | t_k); \mathcal{D}_{N}),
\end{align}
where $\hat{{f}} = [\hat{f}_1, \ldots, \hat{f}_{n_x}]^\top$, $\mathcal{D}_N = \{\mathcal{D}_{N, i}\}_{i = 1}^{n_x}$, and ${ \mu}_{f} = [\mu_{f_1}, \ldots, \mu_{f_{n_x}}]^\top$. 

Based on the above notation, we define a cost function to be minimized as follows: 
\begin{align}\label{costfunc}
J_{H_k} ({{x}}(t_k), \bfmath{u}(t_k)) = &\sum_{j = 0}^{H_k - 1} h_s(\hat{{x}}(j | t_k), {{u}}(j | t_k)),
\end{align}
with $\hat{{x}}(0 | t_k) = x(t_k)$, where $H_k \in \mathbb{N}_{>0}$ is the prediction horizon that is selected at $t_k$, $\bfmath{u}(t_k) = \{{{u}}(0 | t_k), \ldots, {{u}}(H_k - 1 | t_k)\}$ is the set of control inputs to be optimized, and $h_s: \mathbb{R}^{n_x + n_u} \to \mathbb{R}_{\geq 0}$ is a given stage cost function. 
Here, the prediction horizon $H_k$ is not given constant for all $k \in \mathbb{N}_{\ge 0}$ but is selected adaptively such that it is decreasing with $k \in \mathbb{N}_{\ge 0}$ from the initial prediction horizon $H_0 = \bar{H} \in \mathbb{N}_{>0}$. More characterization of $H_k$ is provided in \rsec{etmpcalgsec}. 

The overall OCP solved at time $t_k$ is then formulated as follows:
\begin{align}\label{ocp}
&\min_{\bfmath{u}(t_k)}  \ J_{H_k} ({{x}}(t_k), \bfmath{u}(t_k))\notag \\ 
{\rm s.t.} \ \ \ \ \ \ & \hat{{x}}(j + 1 | t_k) = \hat{{f}}(\hat{{x}}(j | t_k), {{u}}(j | t_k); \mathcal{D}_N) \\
&\ u(j | t_k ) \in \mathcal{U}, \ \hat{{x}}(H_k | t_k ) \in \mathcal{X}_f, \notag \\
&\ \forall j = 0, \ldots, H_k - 1, \notag 
\end{align}
where the optimal control and the predictive state sequence are denoted by $\bfmath{u}^*(t_k) = \{{{u}}^*(0 | t_k), \ldots, {{u}}^*(H_k - 1 | t_k)\}$ and $\hat{\bfmath{x}}^*(t_k) = \{\hat{{x}}^*(0 | t_k), \ldots, \hat{{x}}^*(H_k| t_k)\}$, respectively. Moreover, the terminal set is given by ${\mathcal{X}}_f \subset \mathbb{R}^{n_x}$. In this paper, this terminal set $\mathcal{X}_f$ is constructed such that the following assumption is satisfied: 
\begin{assum}\label{terminalsetas}
\normalfont
Given the training dataset $\mathcal{D}_{N} = \{\mathcal{D}_{N, i}\}_{i = 1}^{n_x}$ and a set of scalars $\bar{\gamma}_i \in \mathbb{R}_{> 0}$ given by $\bar{\gamma}_i = \max_{x \in \mathbb{R}^{n_x}, u \in \mathcal{U}} d_i(0_{n_x}, \delta_N(x, u))$ for all $i \in \mathbb{N}_{1: n_x}$ with $\delta_N = [\delta_{N, 1}, \ldots, \delta_{N, n_x}]^\top$,  there exist a set $\mathcal{X}_S \subset \mathbb{R}^{n_x}$, a controller $C: \mathbb{R}^{n_x} \to 2^\mathcal{U}$, and a set of scalars $\gamma_i \in [\bar{\gamma}_i, \sqrt{2}\alpha_i)$ for all $i \in \mathbb{N}_{1:n_x}$, such that the following conditions are satisfied:  
\begin{enumerate}
    \item It follows that 
    \begin{align}\label{terminalsetdef}
        \mathcal{X}_f \subseteq \bigcap_{i = 1}^{n_x}\mathsf{Int}_{d_i}(\mathcal{X}_S; \gamma_i). 
    \end{align}
    \item For every $x \in \mathcal{X}_S$, there exists ${u} \in C({x})$ such that $\mathcal{F}(x, u; \mathcal{D}_N) \oplus \mathcal{W} \subseteq \mathcal{X}_f$, where $\mathcal{F}(x, u; \mathcal{D}_N) = \mathcal{F}_1(x, u; \mathcal{D}_N) \times \cdots \times \mathcal{F}_{n_x}(x, u; \mathcal{D}_N)$. 
\qedwhite
\end{enumerate}
\end{assum}

\ras{terminalsetas} implies the existence of $\mathcal{X}_S$ and the controller $C$ such that every state in $\mathcal{X}_S$ can be steered to $\mathcal{X}_f$ by applying the controller $C$. 
 A detailed way of how to construct $\mathcal{X}_f$ such that \ras{terminalsetas} is satisfied is given in \rsec{terminalsetsec}. 

\subsection{Event-triggered condition} \label{evtrigmpcsec}
Let us now provide an overview of the event-triggered condition, which determines the time instants $t_k$, $k \in \mathbb{N}_{\ge 0}$ for solving the OCP \req{ocp}. 
Let $\mathcal{D}_N$ be a training dataset as defined in the previous section and suppose that the OCP is solved at some $t_k$ to obtain the optimal control sequence $\bfmath{u}^*(t_k)$ and the corresponding predictive state sequence $\hat{\bfmath{x}}^*(t_k)$. 
Then, the next time instant $t_{k + 1}$ for solving the OCP is determined as follows:
\begin{align}\label{trigger}
t_{k + 1} = &\min_{j \in \mathbb{N}_{0: H_k}}\{t_k + j: \notag \\
&d_i({x}(t_k + j), \hat{{x}}^*(j | t_k)) > \xi_i^* (j ; x(t_k)), \exists i \in \mathbb{N}_{1: n_x}\} . 
\end{align}
In \req{trigger}, the thresholds $\xi_i^* (j ; x(t_k))$ for all $i \in \mathbb{N}_{1: n_x}$ and $j \in \mathbb{N}_{0: H_k}$ are defined as follows. First, the terminal ones $\xi_i^*(H_k ; x(t_k)) \in \mathbb{R}_{<0}$, $i \in \mathbb{N}_{1: n_x}$ are \textit{any} negative constants so that an event is guaranteed to be triggered within the prediction horizon $H_k$. 
For $j \in \mathbb{N}_{0: H_k - 1}$, $\xi_i^*(j ; x(t_k)) \in \mathbb{R}_{\ge 0}$ are defined as
\begin{align}\label{xidef}
    \xi_i^*(j ; x(t_k)) = \psi_i^*(j | t_k) - \frac{\hat{\delta}_{N, i}(j | t_k)}{b_i},
\end{align}
for all $i \in \mathbb{N}_{1: n_x}$, where $\hat{\delta}_{N, i}(j | t_k) = \delta_{N, i}(\hat{x}^*(j | t_k), {u}^*(j | t_k))$
and $\psi_i^*(j | t_k) \in \mathbb{N}_{> 0}$ are obtained by iteratively solving the following problem \textit{in backwards} from $j = H_k - 1$ to $0$:
\begin{align}\label{recopt}
&\psi_1^*(j | t_k), \ldots, \psi_{n_x}^*(j | t_k) =  \argmax_{\psi_i(j | t_k)} \prod_{i = 1}^{n_x} \psi_i(j | t_k) \notag \\
{\rm s.t.}\ \ & \|{b}\odot \psi(j | t_k) \|^2_{\Lambda_{x, i}^{-1}} \le c^2_{N, i}(j + 1 | t_k), \\
& \frac{\hat{\delta}_{N,i}(j | t_k)}{b_i} \le \psi_i(j | t_k) \le \sqrt{2} \alpha_i - \varsigma, \notag \\
&\ \ \ \ \ \ \ \ \ \ \ \ \ \ \ i = 1, \ldots, n_x, \notag 
\end{align}
where $b = [b_1, \ldots, b_{n_x}]^\top$, ${\psi} = [\psi_1, \ldots, \psi_{n_x}]^\top$, and $\varsigma > 0$ is an any given positive constant. Moreover, $c_{N, i}(j + 1 | t_k)$, $j \in \mathbb{N}_{0: H_k - 1}$ 
are defined as follows: 
\begin{align}\label{cdef1}
    c_{N, i}(H_k | t_k) = \sqrt{2 \log\left(\frac{2 \alpha_i^2}{2 \alpha_i^2 - \gamma_i^2}\right)}
\end{align}
 and 
\begin{align}\label{cdef2}
c_{N, i}(j + 1 | t_k) = \sqrt{2 \log\left(\frac{2 \alpha_i^2}{2 \alpha_i^2 - \psi_i^{*2}(j + 1|t_k) }\right)},
\end{align}
for all $j = H_k - 2, \ldots, 0$. 

Note that the optimization problem \req{recopt} is a \textit{geometric programming problem} \cite{boyd2007tutorial}, and thus it can be translated into a convex form and solved in polynomial time. As we will see in the theoretical analysis provided in \rsec{analysissec}, the event-triggered condition \req{trigger} designed above is given to fulfill \textit{recursive feasibility}, which means that the OCP is guaranteed to be feasible at $t_{k + 1}$ (for details, see \rsec{analysissec}). 
The cost function in \req{recopt} has been here defined as a purpose to obtain $\psi _i(j | t_k)$ as large as possible, and it aims at enlarging the thresholds for the event-triggered condition $\xi^*_i(j ; x(t_k))$ (see \req{trigger} and \req{xidef}), so that it will enhance reducing the frequency of solving the OCPs. 

\begin{rem}
\normalfont
Another option for the objective function to be maximized as a purpose to obtain $\psi_i$, $i \in \mathbb{N}_{1: n_x}$ as large as possible is to take the total sum of $\psi_i$ for all $i \in \mathbb{N}_{1: n_x}$, i.e, $\sum_{i = 1}^{n_x} \psi_i$. In such a case, the optimization problem \req{recopt} is defined as a second-order cone programming (SOCP) \cite{boyd2004convex} which is also a convex optimization problem and can be solved in polynomial time. 
\qedwhite
\end{rem}

\subsection{Proposed event-triggered MPC algorithm}\label{etmpcalgsec}

\begin{algorithm}
\caption{Learning-based event-triggered MPC via iterative task}\label{itertask}
\begin{varwidth}[t]{0.95\linewidth} \textbf{Input:} $\mathcal{X}_{\rm init}$ (initial state set), $\bar{H}$ (initial prediction horizon), $\mathcal{D}_{N, i} = \{Z_N, Y_{N, i}\}$ for $i \in \mathbb{N}_{1: n_x}$ (initial training dataset), $\gamma_i$ for $i \in \mathbb{N}_{1: n_x}$ (a set of scalars); \end{varwidth}
\begin{algorithmic}[1]

\Repeat
    \smallskip
    \Statex \ \ \ \ \ \textbf{[Update phase]}
    \smallskip
    
    \State \begin{varwidth}[t]{0.9\linewidth} 
    Using the dataset $\mathcal{D}_{N} = \{\mathcal{D}_{N,i}\}^{n_x}_{i=1}$, update the prediction model \req{gpmodel} and construct the terminal set $\mathcal{X}_f$ (together with $\mathcal{X}_S$ and $C$) such that \ras{terminalsetas} is satisfied; for details, see \rsec{terminalsetsec};\label{safetygameline} \end{varwidth}

    \smallskip
    \Statex \ \ \ \ \ \textbf{[Execution phase]}
    \smallskip
    
    \State \begin{varwidth}[t]{0.9\linewidth} Select the initial state $x(0)$ (possibly randomly) from $\mathcal{X}_{\rm init}$;\label{initialstateline}\end{varwidth}
    
    \State $k \leftarrow 0$, $t_k \leftarrow 0$, $H_k\leftarrow \bar{H}$; \label{initialtimeline}
    \State $\widetilde{\mathcal{D}}_i \leftarrow \varnothing$, $\forall i \in \mathbb{N}_{n_x}$; \label{initialstateline2}

    \Repeat \label{mpcapplyline}
        \State \begin{varwidth}[t]{0.85\linewidth} Given $\mathcal{X}_f$, $H_k$, $x(t_k)$, and using the prediction model \req{gpmodel}, solve the OCP \req{ocp} to obtain $\bfmath{u}^*(t_k)$, $\hat{\bfmath{x}}^*(t_k)$;\end{varwidth}\label{solveopcline}
        \If{Solve \req{recopt} and it is feasible}\label{triggersolline}
            
            \State $j \leftarrow 0$;
            
            \Repeat\label{applyinputlinestart}
                \State \begin{varwidth}[t]{0.7\linewidth} Apply ${{u}}^*(j | t_k)$ to the system \req{realdynamics} and observe the next state ${x}(t_k + j + 1)$; \label{applyinputline}\end{varwidth}
            
                \State \begin{varwidth}[t]{0.75\linewidth}
                $\widetilde{\mathcal{D}}_i \leftarrow \widetilde{\mathcal{D}}_i \cup \{{z}(t_k + j), x_i(t_k + j + 1)\}$ for all $i \in \mathbb{N}_{1: n_x}$
                (Add data points);\label{plantdataaddline}\end{varwidth}
        
                \State $j \leftarrow j + 1$;\label{datanumaddline}
            
            \Until{$t_{k + 1}$ is triggered from \req{trigger};}\label{triggerline}
            
            \State $m_k \leftarrow t_{k + 1} - t_k$;\label{intereventtimeline}
        
            \State \begin{varwidth}[t]{0.8\linewidth}$H_{k + 1} \leftarrow H_k - (m_k - 1)$;\label{decreasehorizon}\end{varwidth}
    
            \State $k \leftarrow k + 1$ ;\label{proceedtimeline}
            
        \Else{ (i.e., \req{recopt} is not feasible)}\label{notfeasibleline}
        
        \State \begin{varwidth}[t]{0.85\linewidth} Apply all inputs ${\bfmath{u}}^*(t_k)$ to the system \req{realdynamics} and observe the states $x_i(t_k + j)$, $j \in \mathbb{N}_{1:H_k}$. 
        Then, add all data points $\{z(t_k + j), x_i(t_k + j + 1)\}$ to $\widetilde{\mathcal{D}}_i$, i.e., $\widetilde{\mathcal{D}}_i \leftarrow \widetilde{\mathcal{D}}_i \cup \{z(t_k + j), x_i(t_k + j + 1)\}$
        for all $i \in \mathbb{N}_{1: n_x}$ and $j \in \mathbb{N}_{0: H_k - 1}$; \end{varwidth} \label{allinputapplyline}
        
        \State \begin{varwidth}[t]{0.8\linewidth} ${\mathcal{D}}_{N', i} \leftarrow {\mathcal{D}}_{N, i} \cup \widetilde{\mathcal{D}}_i$ for all $i \in \mathbb{N}_{1: n_x}$ ($N'$ denotes the total number of training data points after the update); \end{varwidth}\label{addalldatasline}
        
        \smallskip
        \smallskip
        \State \textbf{break} (Proceed to the next iteration);\label{triggersollineend}
            
        \EndIf
    
    \Until{$x(t_k) \in \mathcal{X}_S$}\label{notdecreasehorizonline}
    
\Until{${x}(t_k) \in \mathcal{X}_S$;}\label{compitertask}
\State \begin{varwidth}[t]{0.9\linewidth} Switch to the controller $C$ to drive the state to $\mathcal{X}_f$ and stay therein for all future times;\end{varwidth}\label{switchconline}
\end{algorithmic}
\end{algorithm}

In this section, we propose a detailed algorithm of the event-triggered MPC strategy. 
Before providing the algorithm, we would like to note first that it may be possible that the problem for finding $\psi^* _i (j|t_k)$, $i\in\mathbb{N}_{1:n_x}$ \req{recopt} becomes infeasible for some $j$ and $t_k$, so that we cannot set the thresholds for the event-triggered strategy and thus have to stop solving the OCP afterwards. As will be clearer later, such infeasibility implies that the feasibility of the OCP after $t_{k}$ is not guaranteed. Intuitively, this could happen especially at the beginning of execution, where the variance computed by the GP model is large in a wide region of the state space. 
Nevertheless, we will show later in this paper that \req{recopt} becomes feasible when the variance becomes small enough (for details, see \rsec{analysissec}).   

In view of the above, in this paper we provide an event-triggered MPC algorithm based on an \textit{iterative task}, in which we {reset} a state trajectory once the problem \req{recopt} becomes infeasible. 
The overall event-triggered MPC algorithm is illustrated in \ralg{itertask}, and the details are described as follows. 
The inputs for the algorithm are $\mathcal{X}_{\rm init}$ (initial state set), $\bar{H}$ (initial prediction horizon), $\mathcal{D}_{N, i} = \{Z_N, Y_{N, i}\}$ for $i \in \mathbb{N}_{1: n_x}$ (initial training dataset), and $\gamma_i$ for $i \in \mathbb{N}_{1: n_x}$ (a set of scalars). Here, it is noted that the initial training dataset $\mathcal{D}_{N, i}$ may be obtained in several ways, such as by employing an expert (i.e., we let an expert control the system for some time to collect the initial training data), or by applying a certain model-free controller (e.g., PID controller).  

In the main phase, the algorithm repeatedly executes the update task (\rline{safetygameline}) and execution task (\rline{initialstateline}-\rline{notdecreasehorizonline}) until the state trajectory reaches the set $\mathcal{X}_S$. 
In Update phase, the prediction model \req{gpmodel} is updated and the terminal set $\mathcal{X}_f$ is constructed such that \ras{terminalsetas} is satisfied (together with the set $\mathcal{X}_S$ and the controller $C$). For a detailed procedure of how to construct the terminal set, see \rsec{terminalsetsec}. 
The algorithm proceeds by Execution phase. 
First, we initialize the state, time, prediction horizon, and the training dataset obtained for the current execution phase 
(\rline{initialstateline}--\rline{initialstateline2}). 
Then, for each $t_k$, $k\in\mathbb{N}_{\geq 0}$, the OCP \req{ocp} is solved to obtain the optimal control sequence $\bfmath{u}^*(t_k)$ and the corresponding state sequence $\hat{\bfmath{x}}^*(t_k)$ (\rline{solveopcline}). 
Then, the optimization problem \req{recopt} is solved, and if it is feasible, the optimal control inputs $u^*(j | t_k)$ are applied to the system \req{realdynamics} and collect the training data $\{{z}(t_k + j), x_i(t_k + j + 1)\}$ for $i \in \mathbb{N}_{1: n_x}$ until the next OCP time $t_{k+1}$ that is determined from \req{trigger}  (\rline{applyinputlinestart}--\rline{triggerline}). As shown in \rline{intereventtimeline} and \rline{decreasehorizon}, the OCP at the next time $t_{k+1}$ is solved with the new prediction horizon $H_{k + 1} = H_{k} - (m_k - 1)$, where $m_k$ is the inter-event time $m_k = t_{k + 1} - t_k$. 
Note in particular that if the inter-even time is greater than 1 ($m_k >1$), the prediction horizon decreases by $m_k-1$. 
Note also that the controller adds the training data $\widetilde{D}_i$, $i\in\mathbb{N}_{1:n_x}$ in the Execution phase, but the prediction model and the terminal set are not updated with the new training data until the algorithm proceeds to the Update phase in the next iteration. 
If the optimization problem \req{recopt} is not feasible (\rline{notfeasibleline}), the optimal control inputs are applied until $t_k + H_k$ and add the training data in $\widetilde{D}_i$ (\rline{addalldatasline}). Then, all the training data collected so far at the current iteration $\widetilde{D}_i$, is added to $\mathcal{D}_{N, i}$ for all $i \in \mathbb{N}_{1: n_x}$. 
Then, it moves on to the next iteration and goes back to the Update phase (\rline{triggersollineend}). 

Finally, once the state trajectory enters $\mathcal{X}_S$, the algorithm terminates the iterative task and switches to the controller $C$ to drive the state to $\mathcal{X}_f$, and stay therein for all future times (\rline{compitertask}--\rline{switchconline}). 

\begin{rem}\label{rem2}
\normalfont
While our approach requires solving an additional optimization problem \req{recopt} (in comparison to the periodic MPC schemes) for determining the thresholds to characterize the event-triggered strategy, we argue that the proposed approach still allows us to reduce the computational burden in comparison to the periodic MPC scheme in the following sense. 
As the iteration progresses and the number of the training dataset increases, the computational burden to compute predictions according to the GP model \req{gpmodel} increases. 
Hence, if we employ the periodic MPC, the computational burden to solve the OCPs would significantly increase as the iteration progresses. 
On the other hand, our approach allows us to reduce the number of solving the OCPs as the iteration progresses; the more we get training data, the more we can decrease the number of solving the OCPs. Note that the geometric programming \req{recopt} can be solved in polynomial time and its computational complexity does \textit{not} depend on the size of the training dataset (i.e., the computational burden of \req{recopt} remain the same for all iterations.) 
Hence, even though our approach requires solving the geometric programming problem, the computational benefit of our approach is still more significant than the periodic MPC especially as the iteration progresses. For this clarification, see our numerical experiment of \rsec{simsec}. 
\qedwhite
\end{rem}

\begin{rem}
\normalfont
Our event-triggered MPC algorithm applies a \textit{dual-mode strategy}, where once the state trajectory enters $\mathcal{X}_S$, the controller is switched to $C$ in \ras{terminalsetas} to drive the state into the terminal set $\mathcal{X}_f$ and stays therein for all future times. 
Hence, it is noted that our proposed approach does not control the system with the event-triggered strategy once the state trajectory enters $\mathcal{X}_S$ (i.e., the control actions in the set $\mathcal{X}_S$ need to be updated every time step). Note that the controller $C$ satisfying \ras{terminalsetas} has been constructed in an \textit{offline} fashion, which means that there is no need to solve an optimization problem to determine the control inputs. Therefore, it is argued that the state of the system can be controlled in the terminal region without heavy computation using the controller $C$. 
\qedwhite
\end{rem}

\section{Analysis}\label{analysissec}
In this section, we analyze the recursive feasibility of the OCP \req{ocp} (\rsec{feasibilitysec}) and the convergence of the closed-loop system under \ralg{itertask} (\rsec{stabilitysec}). 

\subsection{Recursive feasibility}\label{feasibilitysec}
Let us analyze the recursive feasibility of the OCP \req{ocp}. 
First, note that each OCP time $t_k$, $k \in \mathbb{N}_{\geq 0}$ is determined according to the event-triggered condition \req{trigger} whose parameters are given by solving the optimization problem \req{recopt}. 
As mentioned previously in \rsec{etmpcalgsec}, even though the OCP at $t_k$ is feasible, it is possible that the event-triggered condition \req{trigger} becomes \textit{ill}-defined due to an infeasibility of \req{recopt}, and this could happen when the variance computed by the GP model is large in a wide region of the state space. 
In order to guarantee the feasibility of \req{recopt}, we need to provide an assumption that the training dataset $\mathcal{D}_N = \{\mathcal{D}_{N, i}\}_{i = 1}^{n_x}$ is \textit{sufficiently collected} in the sense that the variances (uncertainties) $\delta_{N, i}(x, u)$ for all $i \in \mathbb{N}_{1: n_x}$ become small enough, as well as that the upper bound of the RKHS is selected small enough. 

\begin{assum}\label{datasetas}
\normalfont
Given the training dataset $\mathcal{D}_N = \{\mathcal{D}_{N, i}\}_{i = 1}^{n_x}$, we have $\delta_{N, i}(x, u) < \sqrt{2} \alpha_i b_i$ for all $x \in \mathbb{R}^{n_x}$, $u \in \mathcal{U}$, and $i \in \mathbb{N}_{1: n_x}$. In addition, {$b_i \le \min_{\ell \in \mathbb{N}_{1: n_x}}(\sqrt{\lambda_{\ell, i}} / \alpha_\ell)$}.
\end{assum}
Under the above assumptions, we can show the feasibility of \req{recopt}:
\begin{lem}\label{recoptfeasiblelem}
\normalfont
Let Assumptions~1, 2 and 3 hold and suppose that the OCP \req{ocp} at $t_k$ is feasible. 
Then, the optimization problem \req{recopt} at $t_k$ is also feasible. 
\qedwhite
\end{lem}
For the proof of \rlem{recoptfeasiblelem}, see \rapp{recoptfeasiblelemproof}. \rlem{recoptfeasiblelem} shows that the feasibility of the OCP implies the feasibility of \req{recopt}. 
Using \rlem{recoptfeasiblelem}, we can now show that the event-triggered MPC under the event-triggered condition \req{trigger} achieves recursive feasibility: 

\begin{thm}\label{feasiblethm}
\normalfont
Let \ras{rkhsas}, 2 and 3 hold.
Suppose that the OCP \req{ocp} at time $t_k$ ($k \in \mathbb{N}_{\ge 0}$) is feasible.
Moreover, suppose that the next OCP time $t_{k + 1}$ 
is determined according to the event-triggered condition \req{trigger}.
Then, the OCP at $t_{k + 1}$ is feasible with the prediction horizon $H_{k + 1} = H_k - m_k + 1$, where $m_k = t_{k + 1} - t_k$. 
\qedwhite
\end{thm}
For the proof of \rthm{feasiblethm}, see \rapp{feasiblethmproof}. 

\subsection{Convergence analysis}\label{stabilitysec}
Let us now show the result of the convergence analysis of the proposed approach in \ralg{itertask}. 
As a preliminary to provide the main result, we first give the following lemma regarding the length of the prediction horizon and the state of the system \req{realdynamics}.
\begin{lem}\label{horizonlem}
\normalfont
Given the training dataset $\mathcal{D}_N = \{\mathcal{D}_{N, i}\}_{i = 1}^{n_x}$, and let \ras{rkhsas}, 2 and 3 hold.
Suppose that the OCP at time $t_k$ ($k \in \mathbb{N}_{\ge 0}$) are feasible.
Moreover, suppose that the next OCP time $t_{k + 1}$ is determined according to the event-triggered condition \req{trigger} and the prediction horizon is computed as $H_{k + 1} = H_k - m_k + 1 = 1$. 
Then, $x(t_{k + 1}) \in \mathcal{X}_S$. 
\qedwhite
\end{lem}
For the proof of \rlem{horizonlem}, see \rapp{horizonlemproof}. \rlem{horizonlem} means that if the prediction horizon is set to $1$ in \ralg{itertask}, it follows that the state is in $\mathcal{X}_S$. 

Given a set of training dataset $\mathcal{D}_N = \{\mathcal{D}_{N, i}\}_{i = 1}^{n_x}$, let $\bar{\mathcal{X}} \subseteq \mathcal{X}_{\rm init}$ denote the set of all the states in $\mathcal{X}_{\rm init}$ which can be controlled towards the terminal set $\mathcal{X}_f$ under the prediction model \req{gpmodel} and the initial prediction horizon $\bar{H}$, i.e., $\bar{\mathcal{X}} = \{x(t_0) \in \mathcal{X}_{\rm init}: \exists \{u(j | t_0) \in \mathcal{U}\}_{j = 0}^{\bar{H} - 1}, \hat{x}({\bar{H}}| t_0) \in \mathcal{X}_f\}$, where $\hat{x}(j + 1 | t_0) = \hat{{f}}(\hat{x}(j | t_0), u(j | t_0); \mathcal{D}_{N})$ for all $j \in \mathbb{N}_{0 : \bar{H} - 1}$ with $\hat{x}(0 | t_0) = x(t_0)$. That is, $\bar{\mathcal{X}}$ is the set of all the initial states in $\mathcal{X}_{\rm init}$, for which the OCP is feasible with the initial prediction horizon $\bar{H}$. 

Using \rlem{horizonlem}, we now show that any state starting from $\bar{\mathcal{X}}$ converges to $\mathcal{X}_f$ in finite time, if the iteration progresses and the uncertainties on the optimal predictive states are small enough. 

\begin{thm}\label{stabilitythm}
\normalfont
Given the training data set $\mathcal{D}_N = \{\mathcal{D}_{N, i}\}_{i = 1}^{n_x}$, and let Assumptions~1, 2, and 3 hold. Let the Execution phase (line~3--line~23) in \ralg{itertask} be implemented.
Moreover, suppose that the uncertainty $\delta_{N}(x, u)$ fulfills the following condition:
\begin{align}\label{stabilityeq}
    d_i(0_{n_x}, \delta_{N}(x, u)) \le \xi_i^*(1 ; x),
\end{align}
for all $x \in \mathbb{R}^{n_x}$, $u \in \mathcal{U}$, and $i \in \mathbb{N}_{1: n_x}$.
Then, the state of the system \req{realdynamics} starting from any $x(t_0) \in \bar{\mathcal{X}}$ enters the terminal set $\mathcal{X}_f$ within the time steps $2\bar{H} - 1$. 
\qedwhite
\end{thm}

\begin{proof}
Suppose that we have $x(t_0) \in \bar{\mathcal{X}}$. 
Then, from \rlem{recoptfeasiblelem} and \rthm{feasiblethm}, the optimal control sequence $\bfmath{u}^*(t_k)$, the corresponding (predictive) states $\bfmath{\hat{x}}^*(t_k)$, and the set of parameters $\{\xi_i^*(j ; x(t_k))\}_{i = 1}^{n_x}$ are obtained for all OCP time $t_k \in \mathbb{N}_{\ge 0}$.
Moreover, for some given $T \in \mathbb{N}_{> 0}$, the prediction horizon of the OCP at $t_T$ can be recursively computed by 
\begin{align}\label{horizondevelop}
H_T &= H_{T - 1} - (m_{T - 1} - 1) \notag \\
&= H_{T - 2} - (m_{T - 2} - 1) - (m_{T - 1} - 1)\notag \\
&= H_{0} - \sum_{k = 0}^{T - 1} (m_k - 1).
\end{align}
Moreover, suppose that the uncertainty $\delta_{N, i}(x, u)$ fulfills the condition \req{stabilityeq} for all $x \in \mathbb{R}^{n_x}$, $u \in \mathcal{U}$, and $i \in \mathbb{N}_{1: n_x}$. Then, from \rlem{metricbound}, \req{stabilityeq}, and $d_i(x(t_k), \hat{x}^*(0 | t_k)) = 0$, the kernel metric \req{kernelmetric} between the actual state ${x}(t_k + 1) = {f}({x}(t_k), {{u}}^*(0|t_k)) + {w}(t_k)$ and the nominal one $\hat{{x}}^*(1 | t_k)$ can be computed by
\begin{align}\label{kernelmetricboundcomputezero}
&d_i({x}(t_k + 1), \hat{{x}}^*(1 | t_k)) \notag \\
&\le \zeta_{N, i}(x(t_k), \hat{x}^*(0 | t_k), {u}^*(0 | t_k)) \notag \\
&= \sqrt{2\alpha_i^2\left\{1 - \exp \left(- \frac{1}{2} \|\hat{\delta}_N(0 | t_k) \|^2_{\Lambda_{x, i}^{-1}}\right) \right\}} \notag \\
&= d_i(0_{n_x}, \hat{\delta}_{N}(0 | t_k)) \le \xi_i^*(1 ; x(t_k)),
\end{align}
for all $i \in \mathbb{N}_{1: n_x}$. It follows that the time instant $t_k + 1$ is not triggered in \req{trigger} and the inter-event time becomes $m_{k} \in \mathbb{N}_{2: H_k}$ for all $k \in \mathbb{N}_{0 : T - 1}$. Therefore, from \req{horizondevelop} and $m_{k} \in \mathbb{N}_{2: H_k}$ for all $k \in \mathbb{N}_{0: T - 1}$, the prediction horizon $H_T$ is obtained as
\begin{align}\label{horizonmax}
    H_T \le H_0 - \sum_{k = 0}^{T - 1}(2 - 1) = H_0 - T.
\end{align}
Using \req{horizonmax}, let us show that  the state of the system \req{realdynamics} starting from any $\bar{\mathcal{X}}$ is guaranteed to enter the terminal set $\mathcal{X}_f$ within finite time by contradiction as follows.

Assume that there exists time $t_T \ge t_0 + 2H_0 - 2$ such that we have $x(t_T) \notin \mathcal{X}_S$. From the contrapositive of \rlem{horizonlem}, since $x(t_T) \notin \mathcal{X}_S$, the prediction horizon at $t_T$ is obtained as $H_T \ge 2$. Moreover, \req{horizondevelop} can be further computed by
\begin{align}
H_T &= H_{0} + T - \sum_{k = 0}^{T - 1} m_k \notag \\
&= H_0 + T - (t_{T} - t_{T - 1} + t_{T - 1} - t_{T - 2} + \ldots + t_1 - t_0) \notag \\
&= H_0 + T - (t_T - t_0).
\end{align}
Hence, by the assumption $t_T \ge t_0 + 2H_0 - 2$, we have $H_0 \le T$.
However, from \req{horizonmax}, if $H_0 \le T$ holds, the prediction horizon becomes $H_T \le 0$. This clearly contradicts that the fact that we have $H_T > 0$. Thus, it is proved that the state of the system enters the set $\mathcal{X}_S$ within the time interval $2\bar{H} - 2$ by contradiction. 
Furthermore, from \ras{terminalsetas}, there exists $u = C(x(t_{T}))$, such that $\hat{x}(t_{T} + 1) = f(x(t_{T}), u) + w \in \mathcal{F}(x(t_{T}), u; \mathcal{D}_N) \oplus \mathcal{W} \subseteq \mathcal{X}_f$, satisfying the terminal constraint. 
Therefore, the state of the system starting from the initial state set $\bar{\mathcal{X}}$ is guaranteed to enter the terminal set $\mathcal{X}_f$ within the time interval $2\bar{H} - 1$.

\end{proof}

\section{Terminal set construction via symbolic models}\label{terminalsetsec}
In this section we formulate a terminal set construction via a symbolic model. First, in \rsec{simrelsetinvsec}, we review some concepts to define the symbolic model. Then, in \rsec{symbolicmodelsec}, several definitions of the symbolic model are formulated based on \rsec{simrelsetinvsec}. Finally, we provide how to construct the terminal set via the symbolic model in \rsec{constterminalsetsec}.

\subsection{Simulation relation and set-invariant theory}\label{simrelsetinvsec}
The construction of the terminal set $\mathcal{X}_f$ satisfying \ras{terminalsetas} relies on the notion of a symbolic model, which, as detailed below, abstracts the behaviors of the control system \req{realdynamics}. 
First, we define the notation of an approximate alternating simulation relation, which captures behavioral relationships between the two systems.
\begin{defn}\label{asrdef}
\normalfont
Let $\Sigma_a = (\mathcal{X}_a, \mathcal{U}_a, G_a)$ and $\Sigma_b = (\mathcal{X}_b, \mathcal{U}_b, G_b)$ be two transition systems. Given 
$\epsilon_i \in \mathbb{R}_{> 0}$ and metrics $d_i  : \mathbb{R}^{n_x} \times \mathbb{R}^{n_x} \rightarrow \mathbb{R}_{\geq 0}$ for all $i \in \mathbb{N}_{1: n_x}$, 
a relation $R(\varepsilon) \subseteq \mathcal{X}_a \times \mathcal{X}_b$ with $\varepsilon = \{\epsilon_i\}_{i = 1}^{n_x}$ is called an $\varepsilon$-approximate alternating simulation relation (or $\varepsilon$-ASR for short) from $\Sigma_a$ to $\Sigma_b$ if the following conditions are satisfied: 
\begin{itemize}
\setlength{\leftskip}{0.5cm}
    \item[(C.1)] For every ${x}_{a} \in \mathcal{X}_a$, there exists ${x}_{b} \in \mathcal{X}_b$ with $({x}_{a}, {x}_{b}) \in R(\varepsilon)$, 
    \item[(C.2)] For every $({x}_{a}, {x}_{b}) \in R(\varepsilon)$, we have $d_i ({x}_a, {x}_b) \le \epsilon_i$ for all $i \in \mathbb{N}_{1: n_x}$, 
    \item[(C.3)] For every $({x}_{a}, {x}_{b}) \in R(\varepsilon)$ and for every ${u}_a \in \mathcal{U}_a({x}_a)$, there exist ${u}_b \in \mathcal{U}_b({x}_b)$ such that for every ${x}_b' \in G_b({x}_b, {u}_b)$ there exists ${x}_a' \in G_a({x}_a, {u}_a)$ satisfying $({x}_a', {x}_b') \in R(\varepsilon)$. \qedwhite 
\end{itemize}
\end{defn}
Roughly speaking, the existence of an $\varepsilon$-ASR from $\Sigma_a$ to $\Sigma_b$, i.e., $R(\varepsilon) \in \mathcal{X}_a \times \mathcal{X}_b$, means that the state evolution of a concrete system $\Sigma_b$ is simulated by the state evolution of the abstract system $\Sigma_a$. 
More specifically, if $({x}_a, {x}_b) \in R(\varepsilon)$, every transition of $\Sigma_a$ is approximately simulated by those of $\Sigma_b$ according to (C.1) - (C.3) in \rdef{asrdef}. 
Note that $\varepsilon$ as for ($\varepsilon$-)ASR is composed of the set of $\epsilon_i$, i.e., $\varepsilon = \{\epsilon_i\}_{i = 1}^{n_x}$, corresponding to a metric $d_i (\cdot, \cdot)$ in order to allocate the metric $d_i (\cdot, \cdot)$ to the kernel metric $d_i(\cdot, \cdot)$ for all $i \in \mathbb{N}_{1: n_x}$ for constructing an abstract system of the original system \req{realdynamics} in \rsec{symbolicmodelsec}.
The concept of an $\varepsilon$-ASR is useful to synthesize a controller for a concrete system by refining a controller synthesized for the abstract system with the alternating simulation relation by algorithmic techniques from discrete event systems, e.g., \cite{tabuada2009verification}.

Next, we define the concepts of a controlled invariant set and a {contractive} set \cite{blanchini1999set}, which are characterized here by the interval set \req{intervalset} and the kernel metrics \req{kernelmetric}, respectively. 
\begin{defn}\label{invariantdef}
\normalfont
Let $\mathcal{D}_N = \{\mathcal{D}_{N, i}\}_{i = 1}^{n_x}$ be the training dataset. A set $\mathcal{P} \subset \mathbb{R}^{n_x}$ is called a controlled invariant set if there exists a controller $C: \mathbb{R}^{n_x} \to 2^\mathcal{U}$ such that the following condition holds: for every ${x} \in \mathcal{P}$, there exists ${u} \in C({x})$ such that $\mathcal{F}(x, u; \mathcal{D}_N) \oplus \mathcal{W} \subseteq \mathcal{P}$.
\qedwhite
\end{defn}
\begin{defn}\label{contractivedef}
\normalfont
Let $\mathcal{D}_N = \{\mathcal{D}_{N, i}\}_{i = 1}^{n_x}$ be the training dataset, and let $d_i : \mathbb{R}^{n_x} \times \mathbb{R}^{n_x} \rightarrow \mathbb{R}_{\geq 0}$ be the kernel metric in \req{kernelmetric}. 
Given $\widetilde{\gamma} = \{\widetilde{\gamma}_i\}_{i = 1}^{n_x}$ with $\widetilde{\gamma}_i \in (0, \sqrt{2}\alpha_i)$ for $i \in \mathbb{N}_{1: n_x}$, a set $\mathcal{P} \subset \mathbb{R}^{n_x}$ is called a $\widetilde{\gamma}$-contractive set if there exists a controller $C: \mathbb{R}^{n_x} \to 2^\mathcal{U}$ such that the following condition holds: for every ${x} \in \mathcal{P}$, there exists ${u} \in C({x})$ such that $\mathcal{F}(x, u; \mathcal{D}_N) \oplus \mathcal{W} \subseteq \bigcap_{i=1}^{n_x}\mathsf{Int}_{d_i}(\mathcal{P}; \widetilde{\gamma}_i)$.
\qedwhite
\end{defn}
\rdef{contractivedef} means that a set $\mathcal{P} \subset \mathbb{R}^{n_x}$ is called $\widetilde{\gamma}$-contractive if all the states in $\mathcal{P}$ can be driven into a smaller (or, contractive) region $\bigcap_{i=1}^{n_x}\mathsf{Int}_{d_i}(\mathcal{P}; \widetilde{\gamma}_i)$ by applying the controller $C$. 
Note that setting $\widetilde{\gamma}_i \rightarrow 0$ for all $i \in \mathbb{N}_{1: n_x}$ corresponds to a controlled invariant set. Moreover, if $\widetilde{\gamma}_i \rightarrow \sqrt{2}\alpha_i$ for some $i \in \mathbb{N}_{1: n_x}$, then  $\bigcap_{i=1}^{n_x}\mathsf{Int}_{d_i}(\mathcal{P}; \widetilde{\gamma}_i) \rightarrow \varnothing$, since
\begin{align}
\mathsf{Int}_{d_i}(\mathcal{P}; \sqrt{2}\alpha_i) &= \{{x} \in \mathcal{P}\ |\ \mathcal{B}_{d_i}({x}; \sqrt{2}\alpha_i) \subseteq \mathcal{P}\} \notag \\
&= \{{x}_1 \in \mathcal{P} \ |\ \forall x_2 \in \mathbb{R}^{n_x}, \ d_i({x}_1, {x}_2) \le \sqrt{2}\alpha_i \notag \\ 
& \ \ \ \ \ \ \implies {x}_2 \in \mathcal{P} \} \notag \\
&= \{{x}_1 \in \mathcal{P} \ | \ \forall x_2 \in \mathbb{R}^{n_x}, \|{x}_1 - {x}_2\|_{\Lambda_{x, i}} \le \infty \notag  \\ 
& \ \ \ \ \ \ \implies {x}_2 \in \mathcal{P} \} \notag \\
&= \varnothing, 
\end{align}
where we used $d_i({x}_1, {x}_2) \le \sqrt{2}\alpha_i \Longleftrightarrow \|{x}_1 - {x}_2\|_{\Lambda_{x, i}} \le \infty$ from the definition of the kernel metric \req{kernelmetric}. Moreover, the last equality holds since for every $x_1 \in \mathcal{P} \subset \mathbb{R}^{n_x}$ there exists $x_2 \in \mathbb{R}^{n_x}$ such that $\|{x}_1 - {x}_2\|_{\Lambda_{x, i}} \leq \infty$ and $x_2 \notin \mathcal{P}$. 

Based on the contractive set in Definition~4, we can then construct the terminal set $\mathcal{X}_f$ satisfying \ras{terminalsetas} as follows. 
Suppose that we can construct $\widetilde{\gamma}$-contractive set $\mathcal{X}_S \subset \mathbb{R}^{n_x}$ for given $\widetilde{\gamma} = \{\widetilde{\gamma}_i\}_{i = 1}^{n_x}$, $\widetilde{\gamma}_i \in (0, \sqrt{2}\alpha_i)$, $i \in \mathbb{N}_{1: n_x}$. Then, let ${\gamma} = \{{\gamma}_i\}_{i = 1}^{n_x}$ with $0 < \gamma_i \le \widetilde{\gamma}_i$ for all $i \in \mathbb{N}_{1: n_x}$, and let $\mathcal{X}_f$ be the terminal set such that 
\begin{align}\label{ininterminal}
    \mathcal{X}_f \subseteq \left[
    \bigcap_{i=1}^{n_x}\mathsf{Int}_{d_i}(\mathcal{X}_S; \widetilde{\gamma}_i), \bigcap_{i=1}^{n_x}\mathsf{Int}_{d_i}(\mathcal{X}_S; {\gamma}_i)\right]. 
\end{align}
Then, it follows that $\mathcal{X}_f$ satisfies \ras{terminalsetas} with $\gamma$, since $\mathcal{X}_f \subseteq  \bigcap_{i=1}^{n_x}\mathsf{Int}_{d_i}(\mathcal{X}_S; {\gamma}_i)$, and for every $x \in \mathcal{X}_S$ there exists $u \in C(x)$ such that $\mathcal{F}(x, u; \mathcal{D}_N) \oplus \mathcal{W} \subseteq \bigcap_{i=1}^{n_x}\mathsf{Int}_{d_i}(\mathcal{X}_S; \widetilde{\gamma}_i) \subseteq \mathcal{X}_f$. 

\subsection{Symbolic models}\label{symbolicmodelsec}

Let us now construct a symbolic model for deriving the terminal set $\mathcal{X}_f$ satisfying \ras{terminalsetas}. Before constructing the symbolic model, from  \ras{terminalsetas}, the terminal set $\mathcal{X}_f$ is defined based on the interval set $\mathcal{F}$ in \rlem{errorbound} that over-approximates the unknown transition function $f$ such that $f(x, u) \in \mathcal{F}(x, u; \mathcal{D}_N)$ for all $x \in \mathbb{R}^{n_x}$ and $u \in \mathbb{R}^{n_u}$.
Therefore, we start by showing that the over-approximation $\mathcal{F}$ of the unknown system \req{realdynamics} can be described by a notation of a transition system in \rdef{transdef}:
\begin{defn}\label{realtrans}
\normalfont
Let $\mathcal{D}_{N} = \{\mathcal{D}_{N, i}\}_{i = 1}^{n_x}$ be the training dataset. A transition system induced by the system \req{realdynamics} whose transitions are learned by the GPR with the training dataset $\mathcal{D}_{N}$ is a triple $\Sigma = (\mathbb{R}^{n_x}, \mathcal{U}, G)$ consisting of:
\begin{itemize}
    \item $\mathbb{R}^{n_x}$ is a set of states;
    \item $\mathcal{U}$ is a set of inputs; 
    \item $G: \mathbb{R}^{n_x} \times \mathcal{U} \to 2^{\mathbb{R}^{n_x}}$ is a transition map, where ${x}^+ \in G({x}, {u})$ iff $x^+ \in \mathcal{F}(x, u; \mathcal{D}_N) \oplus \mathcal{W}$.
    \qedwhite
\end{itemize}
\end{defn} 
Based on the transition system $\Sigma$ in \rdef{realtrans}, a symbolic model for $\Sigma$ is constructed by discretizing the state and the input spaces. More specifically, we consider constructing the symbolic model by a triple $\mathsf{q} = ({\eta}_x, {\eta}_u, \varepsilon)$, where ${\eta}_x = [\eta_{x, 1}, \ldots, \eta_{x, n_x}]^\top \in \mathbb{R}^{n_x}_{>0}$ is the discretization vector of the state space, ${\eta}_u = [\eta_{u, 1}, \ldots, \eta_{u, n_u}]^\top \in \mathbb{R}^{n_u}_{> 0}$ is the discretization vector of the input space, and $\varepsilon = \{\epsilon_i\}_{i = 1}^{n_x}$ with $\epsilon_i \in \mathbb{R}_{ > 0}$ for $i \in \mathbb{N}_{1: n_x}$ is the set of parameters for the precision.
Using the triple $\mathsf{q}$, the corresponding symbolic model for the transition system $\Sigma$ denoted by $\Sigma_\mathsf{q}$ is formally defined as follows:
\begin{defn}\label{symbolictrans}
\normalfont
Let $\mathcal{D}_{N} = \{\mathcal{D}_{N, i}\}_{i = 1}^{n_x}$ be the training dataset and $\Sigma$ be the transition system in \rdef{realtrans}. Given $\mathsf{q} =  ({\eta}_x, { \eta}_u, \varepsilon)$, a symbolic model for $\Sigma$ is defined as a triple $\Sigma_\mathsf{q} = (\mathcal{X}_\mathsf{q}, \mathcal{U}_\mathsf{q}, G_\mathsf{q})$ consisting of:
\begin{itemize}
    \item $\mathcal{X}_\mathsf{q} = [\mathbb{R}^{n_x}]_{{ \eta}_x}$ is a set of states;
    \item $\mathcal{U}_\mathsf{q} = [\mathcal{U}]_{{ \eta}_u}$ is a set of inputs;
    \item $G_\mathsf{q}: \mathcal{X}_\mathsf{q} \times \mathcal{U}_\mathsf{q} \to 2^{\mathcal{X}_\mathsf{q}}$ is a transition map, where ${x}_{\mathsf{q}}^+ \in G_\mathsf{q}({x}_\mathsf{q}, {u}_\mathsf{q})$ iff $x_{\mathsf{q}, i}^+ \in \mathcal{F}_i({x}_\mathsf{q}, {u}_\mathsf{q}; \mathcal{D}_{N, i}) \oplus \mathcal{W}_i \oplus \mathcal{E}_i$, where $x_{\mathsf{q}, i}^+$ is the $i$-th element of ${x}_{\mathsf{q}}$ and $\mathcal{E}_i = \left[\pm (b_i \cdot \epsilon_i + \eta_{x, i})\right]$ for all $i \in \mathbb{N}_{1: n_x}$.
    \qedwhite
\end{itemize}
\end{defn}
The symbolic model $\Sigma_\mathsf{q}$ in \rdef{symbolictrans} is well defined in the sense that there exists an $\varepsilon$-ASR from $\Sigma_\mathsf{q}$ to $\Sigma$ in \rdef{asrdef} as follows:
\begin{prop}\label{varepsilonasrprop}
\normalfont
Given the training dataset $\mathcal{D}_N = \{\mathcal{D}_{N, i}\}_{i = 1}^{n_x}$, and suppose that  \ras{rkhsas} holds.
Let $\Sigma$ be the transition system in \rdef{realtrans}.
Given $\mathsf{q} = ({\eta}_x, { \eta}_u, \varepsilon)$, where
\begin{align}\label{epsilonparam}
\epsilon_i \in \left. \left[\sqrt{2\alpha_i^2\left\{1 - \exp(-\frac{1}{2}\|{ \eta}_x\|^2_{\Lambda_{x, i}^{-1}})\right\}}, \sqrt{2} \alpha_i\right. \right),
\end{align}
for all $i \in \mathbb{N}_{1: n_x}$, let $\Sigma_\mathsf{q}$ be the symbolic model for $\Sigma$ in \rdef{symbolictrans}. Then, the relation
\begin{align}\label{varepsilonasr}
R(\varepsilon) = \{({x}_\mathsf{q}, {x}) \in \mathcal{X}_\mathsf{q} \times \mathbb{R}^{n_x}: d_i({x}_\mathsf{q}, {x}) \le \epsilon_i, \ \forall i \in \mathbb{N}_{1: n_x}\}
\end{align}
is an $\varepsilon$-ASR from $\Sigma_\mathsf{q}$ to $\Sigma$.
\qedwhite
\end{prop}
For the proof, see \rapp{varepsilonasrpropproof}. Now, we augment the symbolic model $\Sigma_\mathsf{q}$ to a model that can derive the $\widetilde{\gamma}$-contractive set $\mathcal{X}_S$ satisfying \rdef{contractivedef} and thus derive the terminal set $\mathcal{X}_f$ by \req{ininterminal}. Using the discretization vector of the state space $\eta_x$, we define the set of the contractive parameters $\widetilde{\gamma} = \{\widetilde{\gamma}_i\}_{i = 1}^{n_x}$ by
\begin{align}\label{contractiveparam}
\widetilde{\gamma}_i = \sqrt{2\alpha_i^2\left\{1 - \exp\left(- \frac{2\mathsf{z}^2_i}{n}\|{ \eta}_x\|^2_{\Lambda_{x, i}^{-1}}\right)\right\}} \ge \gamma_i
\end{align}
for all $i \in \mathbb{N}_{1: n_x}$ with some given positive integer $\mathsf{z}_i \in \mathbb{N}_{> 0}$. The set of the contractive parameters $\widetilde{\gamma}$ are designed to have the following property in the discrete state space:
\begin{lem}\label{intoutlemdis}
\normalfont
Given $\widetilde{\gamma}_i$ by \req{contractiveparam} for all $i \in \mathbb{N}_{1: n_x}$. Then, for any two sets $\mathcal{X}_{\mathsf{q}1}, \mathcal{X}_{\mathsf{q}2} \subseteq [\mathbb{R}^{n_x}]_{\eta_x}$, it follows that
\begin{align}\label{intoutdis}
\mathcal{X}_{\mathsf{q}1} \subseteq \bigcap_{i = 1}^{n_x} \widetilde{\mathsf{Int}}_{d_i}(\mathcal{X}_{\mathsf{q}2}; \widetilde{\gamma}_i), \ \Longleftrightarrow \ \bigcup_{i = 1}^{n_x}\widetilde{\mathsf{Out}}_{d_i}(\mathcal{X}_{\mathsf{q}1}; \widetilde{\gamma}_i) \subseteq \mathcal{X}_{\mathsf{q}2}.
\end{align}
\qedwhite
\end{lem}
For the proof of \rlem{intoutlemdis}, see \rapp{intoutlemdisproof}. \rlem{intoutlemdis} is useful in a theoretical analysis of the set invariance in the later section. Given contractive parameters $\widetilde{\gamma}$ by \req{contractiveparam} in addition to the quadruple $\mathsf{q} = ({\eta}_x, {\eta}_u, \varepsilon)$, 
the augmented symbolic model denoted by ${\Sigma}_{\mathsf{q}\tilde{\gamma}}$ is formally defined as follows:
\begin{defn}\label{symbolictransaug}
\normalfont
Let $\mathcal{D}_{N} = \{\mathcal{D}_{N, i}\}_{i = 1}^{n_x}$ be the training dataset and $\Sigma$ be the transition system in \rdef{realtrans}.
Given $\mathsf{q} = ({ \eta}_x, { \eta}_u, \varepsilon)$, let $\Sigma_\mathsf{q}$ be the symbolic model in \rdef{symbolictrans}. Moreover, given $\widetilde{\gamma}$ by \req{contractiveparam}, an augmented symbolic model for $\Sigma$ is defined as a triple ${\Sigma}_{\mathsf{q}\tilde{\gamma}} = (\mathcal{X}_{\mathsf{q}\tilde{\gamma}}, \mathcal{U}_{\mathsf{q}\tilde{\gamma}}, G_{\mathsf{q}\tilde{\gamma}})$ consisting of:
\begin{itemize}
    \item $\mathcal{X}_{\mathsf{q}\tilde{\gamma}} = [\mathbb{R}^{n_x}]_{{ \eta}_x}$ is a set of states;
    \item $\mathcal{U}_{\mathsf{q}\tilde{\gamma}} = [\mathcal{U}]_{{ \eta}_u}$ is a set of inputs;
    \item $G_{\mathsf{q}\tilde{\gamma}}: \mathcal{X}_{\mathsf{q}\tilde{\gamma}} \times \mathcal{U}_{\mathsf{q}\tilde{\gamma}} \to 2^{\mathcal{X}_{\mathsf{q}\tilde{\gamma}}}$ is a transition map, where $\widetilde{{x}}_{\mathsf{q}\tilde{\gamma}}^+ \in G_{\mathsf{q}\tilde{\gamma}}({x}_{\mathsf{q}\tilde{\gamma}}, {u}_{\mathsf{q}\tilde{\gamma}})$ iff ${x}_{\mathsf{q}\tilde{\gamma}}^+ \in  \bigcup_{i = 1}^{n_x} \widetilde{\mathsf{Out}}_{d_i}(G_\mathsf{q}({x}_{\mathsf{q}\tilde{\gamma}}, {u}_{\mathsf{q}\tilde{\gamma}}); \widetilde{\gamma}_i)$.
    \qedwhite
\end{itemize}
\end{defn}

As for the symbolic model ${\Sigma}_{\mathsf{q}\tilde{\gamma}}$ and $\Sigma_\mathsf{q}$, we have the result that there exits an $0$-ASR from ${\Sigma}_{\mathsf{q}\tilde{\gamma}}$ to $\Sigma_\mathsf{q}$ as follows:
\begin{lem}\label{0asrlem}
\normalfont
Given the training dataset $\mathcal{D}_N = \{\mathcal{D}_{N, i}\}_{i = 1}^{n_x}$, and suppose that \ras{rkhsas} holds. Let $\Sigma$ be the transition system in \rdef{realtrans}.
Given $\mathsf{q} = ({\eta}_x, { \eta}_u, \varepsilon)$ and $\widetilde{\gamma}$ in \req{contractiveparam}, let $\Sigma_\mathsf{q}$ and $\Sigma_{\mathsf{q}\tilde{\gamma}}$ be the symbolic model for $\Sigma$ in \rdef{symbolictrans} and \rdef{symbolictransaug}, respectively.
Then, the relation
\begin{align}\label{0asr}
R(0) = \{({x}_{\mathsf{q}\tilde{\gamma}}, {x}_\mathsf{q}) \in \mathcal{X}_{\mathsf{q}\tilde{\gamma}}  \times \mathcal{X}_\mathsf{q} | \ d_i({x}_{\mathsf{q}\tilde{\gamma}}, {x}_\mathsf{q}) = 0, \ \forall i \in \mathbb{N}_{1: n_x} \}
\end{align}
is a $0$-ASR from ${\Sigma}_{\mathsf{q}\tilde{\gamma}}$ to $\Sigma_\mathsf{q}$.
\qedwhite
\end{lem}
For the proof of \rlem{0asrlem}, see \rapp{0asrlemproof}. Using \rprop{varepsilonasrprop} and \rlem{0asrlem}, and combining $0$-ASR from $\Sigma_{\mathsf{q}\tilde{\gamma}}$ to $\Sigma_{\mathsf{q}}$ with $\varepsilon$-ASR from $\Sigma_{\mathsf{q}}$ to $\Sigma$, it is trivial that there exists an $\varepsilon$-ASR from $\Sigma_{\mathsf{q}\tilde{\gamma}}$ to $\Sigma$ as in \cite{zamani2011symbolic}. 

\subsection{Construction of terminal set}\label{constterminalsetsec}
\begin{algorithm}
\caption{Terminal set construction via symbolic model}\label{safetygame}
\textbf{Input:} $\mathcal{X}$ (specification set), $\mathcal{D}_{N}$ (training dataset), $\eta_x, \eta_u, \varepsilon, \widetilde{\gamma}$ (some parameters for symbolic model), $\gamma_i$ for $i \in \mathbb{N}_{1: n_x}$ (a set of scalars);

\textbf{Output:} $\mathcal{X}_f$ (terminal set), $\mathcal{X}_S$ ($\tilde{\gamma}$-contractive set), $C$ (refined controller);
\begin{algorithmic}[1]

\State $\mathsf{q} \leftarrow (\eta_x, \eta_u, \varepsilon)$;\label{tripleline}

\State ${\Sigma}_{\mathsf{q}\tilde{\gamma}} \leftarrow (\mathcal{X}_{\mathsf{q}\tilde{\gamma}}, \mathcal{U}_{\mathsf{q}\tilde{\gamma}}, G_{\mathsf{q}\tilde{\gamma}})$;\label{symbolicline}

\State $\ell \leftarrow 0$;

\State $\mathcal{Q}_{\ell} \leftarrow \left[\bigcap_{i=1}^{n_x}\mathsf{Int}_{d_i}(\mathcal{X}; \epsilon_i)\right]_{{\eta}_x}$\label{finitestatesetline};

\Repeat\label{solvesafetygamestart}
    \State $\mathcal{Q}_{\ell + 1} \leftarrow {\rm Pre}_{{\Sigma}_{\mathsf{q}\tilde{\gamma}}}(\mathcal{Q}_{\ell})$; \label{predecessoralg}
    \State $\ell \leftarrow \ell + 1$;
\Until{$\mathcal{Q}_{\ell - 1} = \mathcal{Q}_{\ell}$;}\label{solvesafetygameend}

\State $\mathcal{X}_{S, {\mathsf{q}\tilde{\gamma}}} \leftarrow \mathcal{Q}_{\ell};$\label{compinvariantdisline}

\State $\mathcal{X}_S \leftarrow \{{x} \in \mathcal{X}: \ ({x}_{\mathsf{q}\tilde{\gamma}}, {x}) \in R(\varepsilon), \exists{x}_{\mathsf{q}\tilde{\gamma}} \in \mathcal{X}_{S, {\mathsf{q}\tilde{\gamma}}}\};$ \label{compinvariantline}

\State $C_{\mathsf{q}\tilde{\gamma}} \leftarrow \{{u}_{\mathsf{q}\tilde{\gamma}} \in \mathcal{U}_{\mathsf{q}\tilde{\gamma}}: \widetilde{G}({x}_{\mathsf{q}\tilde{\gamma}}, {u}_{\mathsf{q}\tilde{\gamma}}) \subseteq \mathcal{X}_{S, {\mathsf{q}\tilde{\gamma}}}, \ \forall {x}_{\mathsf{q}\tilde{\gamma}} \in \mathcal{X}_{S, {\mathsf{q}\tilde{\gamma}}}\};$\label{descontroller}

\State $C \leftarrow \{C_{\mathsf{q}\tilde{\gamma}}({x}_{\mathsf{q}\tilde{\gamma}}): ({x}_{\mathsf{q}\tilde{\gamma}}, {x}) \in R(\varepsilon), \ \forall {x} \in \mathcal{X}_S \}$; \label{refinecontroller}

\State $\mathcal{X}_f \leftarrow [\bigcap_{i=1}^{n_x}\mathsf{Int}_{\tilde{\gamma}_i}(\mathcal{X}_S), \bigcap_{i=1}^{n_x}\mathsf{Int}_{\gamma_i}(\mathcal{X}_S)]$;\label{terminalsetline}

\State \textbf{return} $\mathcal{X}_f, \mathcal{X}_S, C$;
\end{algorithmic}
\end{algorithm}

Let us now construct the terminal set $\mathcal{X}_f$.  
Given the training dataset $\mathcal{D}_{N}$, the triple
$\mathsf{q} = ({\eta}_x, { \eta}_u, \varepsilon)$, and $\widetilde{\gamma}$ in \req{contractiveparam}, let $\Sigma_{\mathsf{q}\tilde{\gamma}}$ be the symbolic model for $\Sigma$. Moreover, let $\mathcal{X} \subset \mathbb{R}^{n_x}$ be a given specification set, where it is assumed that $\mathcal{X}$ is compact and can be either convex or non-convex. Then, using the symbolic model ${\Sigma}_{\mathsf{q}\tilde{\gamma}}$, we can find the terminal set $\mathcal{X}_f$ in the specification set $\mathcal{X}$ by investigating a $\tilde{\gamma}$-contractive set $\mathcal{X}_S$ and the corresponding controller $C$ by a safety game (see, e.g., \cite{tabuada2009verification}). 

The algorithm of the safety game is illustrated in \ralg{safetygame} and can be incorporated into the Update phase in \ralg{itertask}. The inputs for \ralg{safetygame} are $\mathcal{X}$ (specification set), $\mathcal{D}_{N}$ (training dataset), $\eta_x, \eta_u, \varepsilon, \widetilde{\gamma}$ (some parameters for symbolic model), and $\gamma_i$ for $i \in \mathbb{N}_{1: n_x}$ (a set of scalars). The algorighm starts by constructing the symbolic model $\Sigma_{\mathsf{q}\tilde{\gamma}}$ with $\mathsf{q} = (\eta_x, \eta_u, \varepsilon)$ and $\tilde{\gamma}$ in \req{contractiveparam}, and computing a finite set $\mathcal{Q}_0$ of the states from $\mathcal{X}$ as $\mathcal{Q}_0 = \left[\bigcap_{i=1}^{n_x}\mathsf{Int}_{d_i}(\mathcal{X}; \epsilon_i)\right]_{{\eta}_x}$ (\rline{tripleline} - \rline{finitestatesetline}). The algorithm proceeds to solving a safety game by applying the \textit{predecessor operator} ${\rm Pre}_{{\Sigma}_{\mathsf{q}\tilde{\gamma}}}: 2^{\mathcal{X}_{\mathsf{q}\tilde{\gamma}}} \to 2^{\mathcal{X}_{\mathsf{q}\tilde{\gamma}}}$ defined as
\begin{align}\label{predecessordef}
{\rm Pre}_{{\Sigma}_{\mathsf{q}\tilde{\gamma}}}(\mathcal{Q}) = \{{x}_{\mathsf{q}\tilde{\gamma}} \in \mathcal{Q}: G_{\mathsf{q}\tilde{\gamma}}({x}_{\mathsf{q}\tilde{\gamma}}, {u}_{\mathsf{q}\tilde{\gamma}}) \subseteq Q, \exists {u}_{\mathsf{q}\tilde{\gamma}} \in \mathcal{U}_{\mathsf{q}\tilde{\gamma}}\},
\end{align}
for a given $\mathcal{Q} \subseteq \mathcal{X}_{\mathsf{q}\tilde{\gamma}}$ (\rline{solvesafetygamestart} - \rline{solvesafetygameend}). The operator ${\rm Pre }_{{\Sigma}_{\mathsf{q}\tilde{\gamma}}}(\mathcal{Q})$ represents all the state ${x}_{\mathsf{q}\tilde{\gamma}} \in \mathcal{Q}$ for which there exist control inputs $u_{\mathsf{q}\tilde{\gamma}} \in \mathcal{U}_{\mathsf{q}\tilde{\gamma}}$ such that the $u_{\mathsf{q}\tilde{\gamma}}$-successors are in $\mathcal{Q}$. Therefore, the operator ${\rm Pre }_{{\Sigma}_{\mathsf{q}\tilde{\gamma}}}(\mathcal{Q})$ provides a collection of the states, where the states of the symbolic model ${\Sigma}_{\mathsf{q}\tilde{\gamma}}$ can remain in $\mathcal{Q}$. 
Hence, using the $\varepsilon$-ASR from ${\Sigma}_{\mathsf{q}\tilde{\gamma}}$ to $\Sigma$, if $\mathcal{X}_{S, {\mathsf{q}\tilde{\gamma}}} \neq \varnothing$, the set $\mathcal{X}_S$ is obtained from the fixed point set of $\mathcal{X}_{S, {\mathsf{q}\tilde{\gamma}}}$ as a controlled invariant set according to \rdef{invariantdef} (\rline{compinvariantdisline} - \rline{compinvariantline}). Moreover, the controller $C$, by which the states of $\Sigma$ can remain in $\mathcal{X}_S$, is refined from the fixed point set of $C_{{\mathsf{q}\tilde{\gamma}}}$ (\rline{descontroller} - \rline{refinecontroller}). Finally, the terminal set $\mathcal{X}_f$ is obtained by extracting a set between  $\bigcap_{i=1}^{n_x}\mathsf{Int}_{d_i}(\mathcal{X}_S; \widetilde{\gamma}_i)$ and $ \bigcap_{i=1}^{n_x}\mathsf{Int}_{d_i}(\mathcal{X}_S; \widetilde{\gamma}_i)$ (\rline{terminalsetline}). Note that \ralg{safetygame} is guaranteed to terminate within a finite number of iterations, since both $\left[\bigcap_{i=1}^{n_x}\mathsf{Int}_{d_i}(\mathcal{X}; \epsilon_i)\right]_{{\eta}_x}$ and $[\mathcal{U}]_{{ \eta}_u}$ are finite. The following result follows that $R(\varepsilon)$ is the $\varepsilon$-ASR from ${\Sigma}_{\mathsf{q}\tilde{\gamma}}$ to $\Sigma$, and thus the proof is omitted (see, e.g., \cite{tabuada2009verification}).
\begin{lem}\label{invariantlem}
\normalfont
Let $\mathcal{X}$, $\mathcal{D}_N$, $\mathsf{q} = ({\eta}_x, { \eta}_u, \varepsilon)$, and $\widetilde{\gamma}$ in \req{contractiveparam} be the inputs for \ralg{safetygame}.
Suppose that \ralg{safetygame} is implemented and we obtain $\mathcal{X}_S \neq \varnothing$. 
Then, $\mathcal{X}_S$ is a controlled invariant set in $\mathcal{X}$ according to \rdef{invariantdef} and $C$ is the corresponding controller.
\qedwhite
\end{lem}

Let us now show that the main result of constructing the terminal set $\mathcal{X}_f$ satisfying \ras{terminalsetas} by using $\varepsilon$-ASR from ${\Sigma}_{\mathsf{q}\tilde{\gamma}}$ to $\Sigma$. The following proposition states that the controlled invariant set $\mathcal{X}_S$ obtained by \ralg{safetygame} is also the $\widetilde{\gamma}$-contractive set according to \rdef{contractivedef}, and all the states in $\mathcal{X}_S$ can drive into the terminal set $\mathcal{X}_f$ by the controller $C$:
\begin{prop}\label{contractiveprop}
\normalfont
Let $\mathcal{X}$, $\mathcal{D}_N$, $\mathsf{q} = ({\eta}_x, { \eta}_u, \varepsilon)$, and $\widetilde{\gamma}$ in \req{contractiveparam} be the inputs for \ralg{safetygame}.
Suppose that \ralg{safetygame} is implemented and we obtain $\mathcal{X}_S \neq \varnothing$. 
Then, $\mathcal{X}_S$ is a $\widetilde{\gamma}$-contractive set in $\mathcal{X}$ according to \rdef{contractivedef}, and all the states in $\mathcal{X}_S$ can drive into the terminal set $\mathcal{X}_f$ by the corresponding controller $C$.
\qedwhite
\end{prop}

\begin{proof}
Suppose that \ralg{safetygame} is implemented with the inputs $\mathcal{X}$, $\mathcal{D}_N$, $\mathsf{q} = ({\eta}_x, { \eta}_u, \varepsilon)$, and $\widetilde{\gamma}$ in \req{contractiveparam} and the solution of the safety game is obtained as $\mathcal{Q}_{\ell - 1} = \mathcal{Q}_{\ell} = \mathcal{X}_{S, \mathsf{q}\tilde{\gamma}} \neq \varnothing$ for some $\ell \in \mathbb{N}_{>0}$. Then, from the notation of the transition map $G_{\mathsf{q}\tilde{\gamma}}$ in \rdef{symbolictransaug} and the operator ${\rm Pre}_{{\Sigma}_{\mathsf{q}\tilde{\gamma}}}$ in \req{predecessordef}, we have 
\begin{align}\label{discreatesafesettrans}
\mathcal{X}_{S, {\mathsf{q}\tilde{\gamma}}} &= {\rm Pre }_{{\Sigma}_{\mathsf{q}\tilde{\gamma}}}(\mathcal{X}_{S, {\mathsf{q}\tilde{\gamma}}})\notag\\
&=\{{x}_{\mathsf{q}\tilde{\gamma}} \in \mathcal{X}_{S, {\mathsf{q}\tilde{\gamma}}}: G_{\mathsf{q}\tilde{\gamma}}({x}_{\mathsf{q}\tilde{\gamma}}, {u}_{\mathsf{q}\tilde{\gamma}}) \subseteq \mathcal{X}_{S, {\mathsf{q}\tilde{\gamma}}}, \notag \\
& \ \ \ \ \ \ \ \ \ \ \ \ \ \ \ \ \ \ \ \ \ \ \ \ \ \ \ \ \ \ \ \ \ \ \ \ \ \ \ \ \exists {u}_{\mathsf{q}\tilde{\gamma}} \in \mathcal{U}_{\mathsf{q}\tilde{\gamma}}\} \notag \\
&= \{{x}_{\mathsf{q}\tilde{\gamma}} \in \mathcal{X}_{S, {\mathsf{q}\tilde{\gamma}}}: \bigcup_{i = 1}^{n_x}\widetilde{\mathsf{Out}}_{d_i}(G_\mathsf{q}({x}_{\mathsf{q}\tilde{\gamma}}, {u}_{\mathsf{q}\tilde{\gamma}}); \widetilde{\gamma}_i) \subseteq \mathcal{X}_{S, {\mathsf{q}\tilde{\gamma}}}, \notag \\
& \ \ \ \ \ \ \ \ \ \ \ \ \ \ \ \ \ \ \ \ \ \ \ \ \ \ \ \ \ \ \ \ \ \ \ \ \ \ \ \ \exists {u}_{\mathsf{q}\tilde{\gamma}} \in \mathcal{U}_{\mathsf{q}\tilde{\gamma}}\}.
\end{align}
Furthermore, using \rlem{intoutlemdis}, \req{discreatesafesettrans} is then computed as
\begin{align}\label{discreatesafeset}
\mathcal{X}_{S, {\mathsf{q}\tilde{\gamma}}}&= \{{x}_{\mathsf{q}\tilde{\gamma}} \in \mathcal{X}_{S, {\mathsf{q}\tilde{\gamma}}}: G_\mathsf{q}({x}_{\mathsf{q}\tilde{\gamma}}, {u}_{\mathsf{q}\tilde{\gamma}}) \subseteq \bigcap_{i = 1}^{n_x}\widetilde{\mathsf{Int}}_{d_i}(\mathcal{X}_{S, {\mathsf{q}\tilde{\gamma}}}; \widetilde{\gamma}_i), \notag \\
& \ \ \ \ \ \ \ \ \ \ \ \ \ \ \ \ \ \ \ \ \ \ \ \ \ \ \ \ \ \ \ \ \ \ \ \ \ \ \ \ \exists {u}_{\mathsf{q}\tilde{\gamma}} \in \mathcal{U}_{\mathsf{q}\tilde{\gamma}}\}.
\end{align}
Therefore, from \req{discreatesafeset}, we obtain the result that there exist control inputs $u_{\mathsf{q}\tilde{\gamma}} \in \mathcal{U}_{\mathsf{q}\tilde{\gamma}}$ such that all the $u_{\mathsf{q}\tilde{\gamma}}$-successors in the set $\mathcal{X}_{S, {\mathsf{q}\tilde{\gamma}}}$ are inside $\bigcap_{i = 1}^{n_x}\widetilde{\mathsf{Int}}_{d_i}(\mathcal{X}_{S, {\mathsf{q}\tilde{\gamma}}}; \widetilde{\gamma}_i)$. That is, for every $x_{\mathsf{q}\tilde{\gamma}} \in \mathcal{X}_{S, \mathsf{q}\tilde{\gamma}}$, there exists a controller $C_{{\mathsf{q}\tilde{\gamma}}}$ such that ${x}_{\mathsf{q}\tilde{\gamma}} \xrightarrow{{u}_{\mathsf{q}\tilde{\gamma}}} {x}_{\mathsf{q}\tilde{\gamma}}^+ \in \bigcap_{i = 1}^{n_x}\widetilde{\mathsf{Int}}_{d_i}(\mathcal{X}_{S, {\mathsf{q}\tilde{\gamma}}}; \widetilde{\gamma}_i)$, where $u_{\mathsf{q}\tilde{\gamma}} \in C_{{\mathsf{q}\tilde{\gamma}}}({x}_{\mathsf{q}\tilde{\gamma}})$. Moreover, using the $\varepsilon$-ASR from ${\Sigma}_{\mathsf{q}\tilde{\gamma}}$ to $\Sigma$, both $\mathcal{X}_{S}$ and $\bigcap_{i = 1}^{n_x} \mathsf{Int}_{d_i}(\mathcal{X}_{S}; \widetilde{\gamma}_i)$ can be computed from the fixed point set of $\mathcal{X}_{S, {\mathsf{q}\tilde{\gamma}}}$ and $\bigcap_{i = 1}^{n_x}\widetilde{\mathsf{Int}}_{d_i}(\mathcal{X}_{S, {\mathsf{q}\tilde{\gamma}}}; \widetilde{\gamma}_i)$. At the same time, the controller $C$ can be refined from the fixed point set of $C_{{\mathsf{q}\tilde{\gamma}}}$. Hence, all the states in $\mathcal{X}_S$ can drive into the set $\bigcap_{i=1}^{n_x}\mathsf{Int}_{d_i}(\mathcal{X}_S; \tilde{\gamma}_i)$ by applying the controller $C$. Therefore, $\mathcal{X}_S$ is a ${\widetilde{\gamma}}$-contractive set in $\mathcal{X}$. Furthermore, the terminal set $\mathcal{X}_f$ is defined as $\bigcap_{i=1}^{n_x}\mathsf{Int}_{d_i}(\mathcal{X}_S; \tilde{\gamma}_i) \subseteq \mathcal{X}_f$ in \req{terminalsetdef}. Thus, for every $x \in \mathcal{X}_S$, there exists a controller $C$ such that ${x} \xrightarrow{{u}} {x}^+ \in \mathcal{X}_f$, where $u \in C({x})$.
\end{proof}
From \rprop{contractiveprop}, the terminal set $\mathcal{X}_f$ in the OCP \req{ocp} satisfying \ras{terminalsetas} is constructed by \ralg{safetygame} via the symbolic model $\Sigma_{\mathsf{q}\tilde{\gamma}}$.

\section{Numerical Simulation}\label{simsec}

In this section, we demonstrate the effectiveness of the proposed event-triggered MPC in \ralg{itertask} through a numerical simulation. The simulation was conducted on MacOS Big Sur, 8-core Intel Core i9 2.4GHz, 32GB RAM.

\begin{figure*}[t]
    \centering
	\subfigure[Iteration $1$.]{
	    \centering
		\includegraphics[width = 0.3\linewidth]{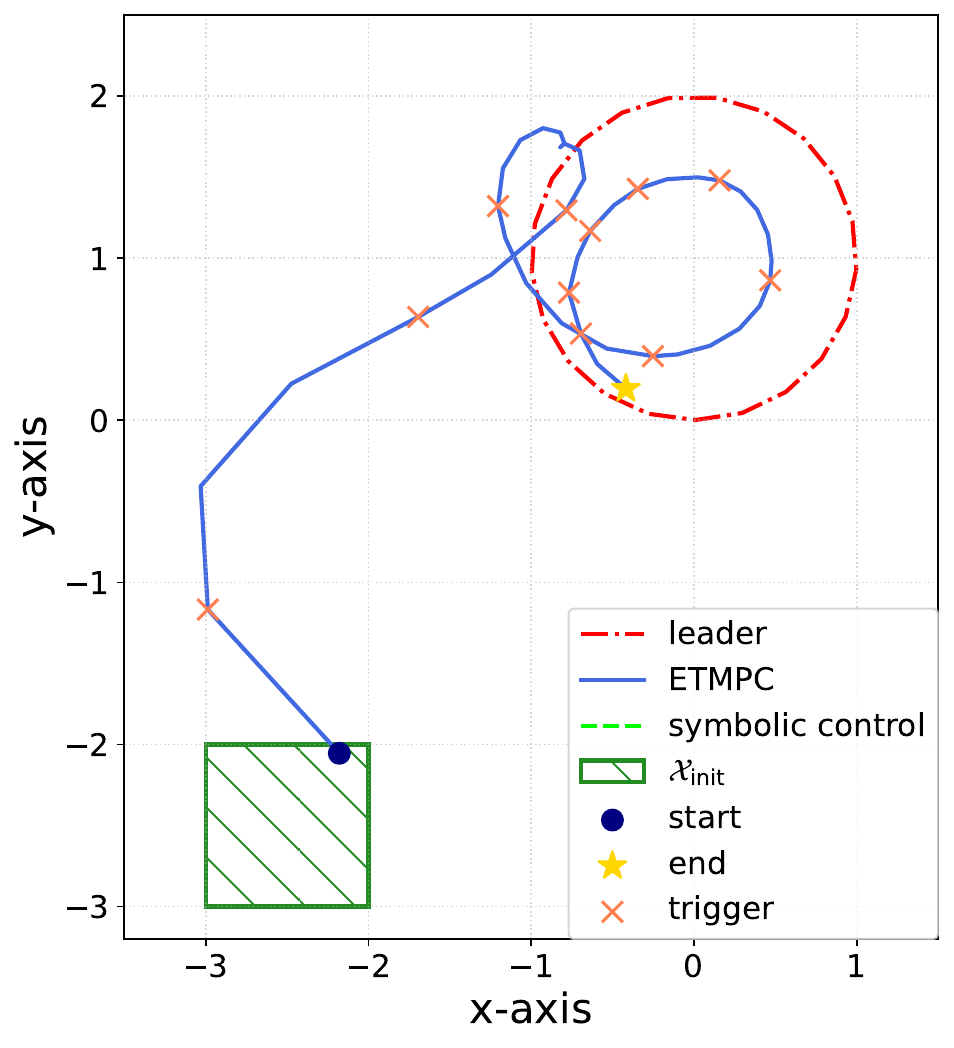}
	}
	\subfigure[Iteration $5$.]{
	    \centering
		\includegraphics[width = 0.3\linewidth]{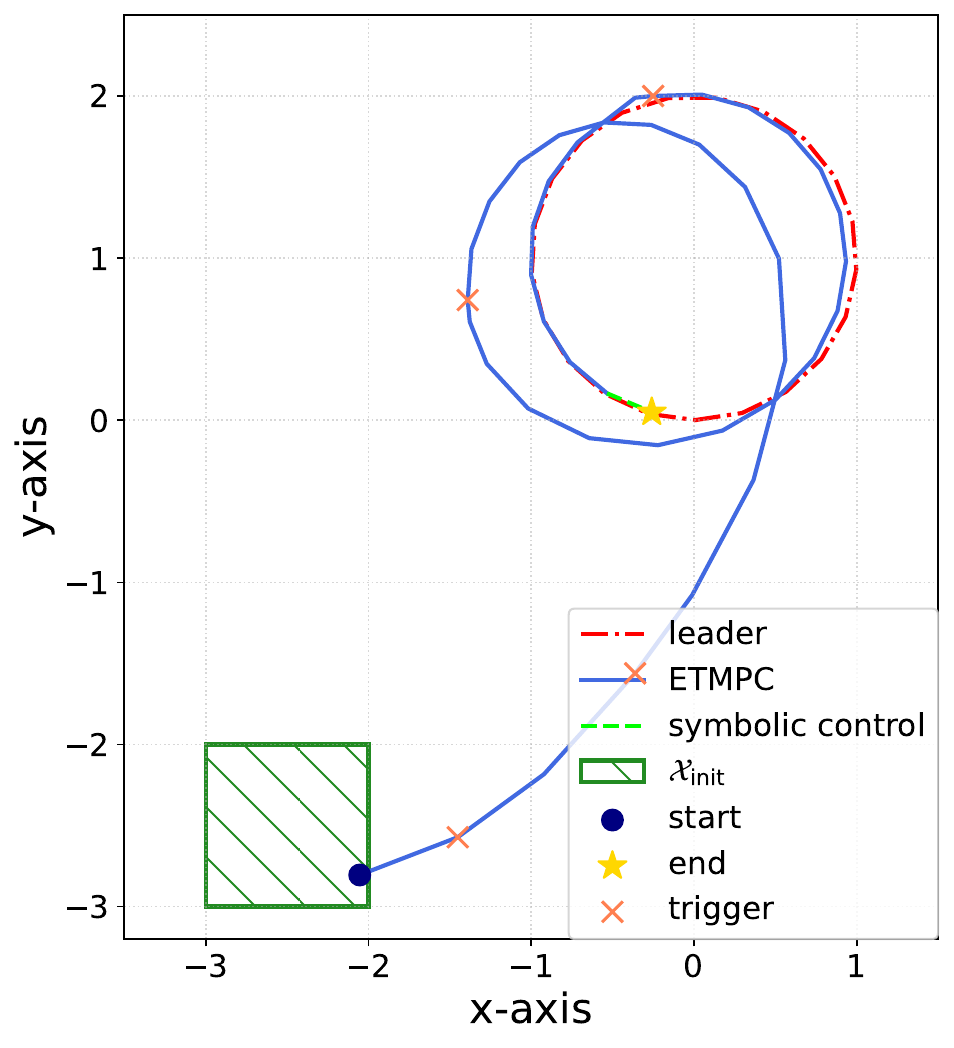}
	}
	\subfigure[Iteration $10$.]{
	    \centering
		\includegraphics[width = 0.3\linewidth]{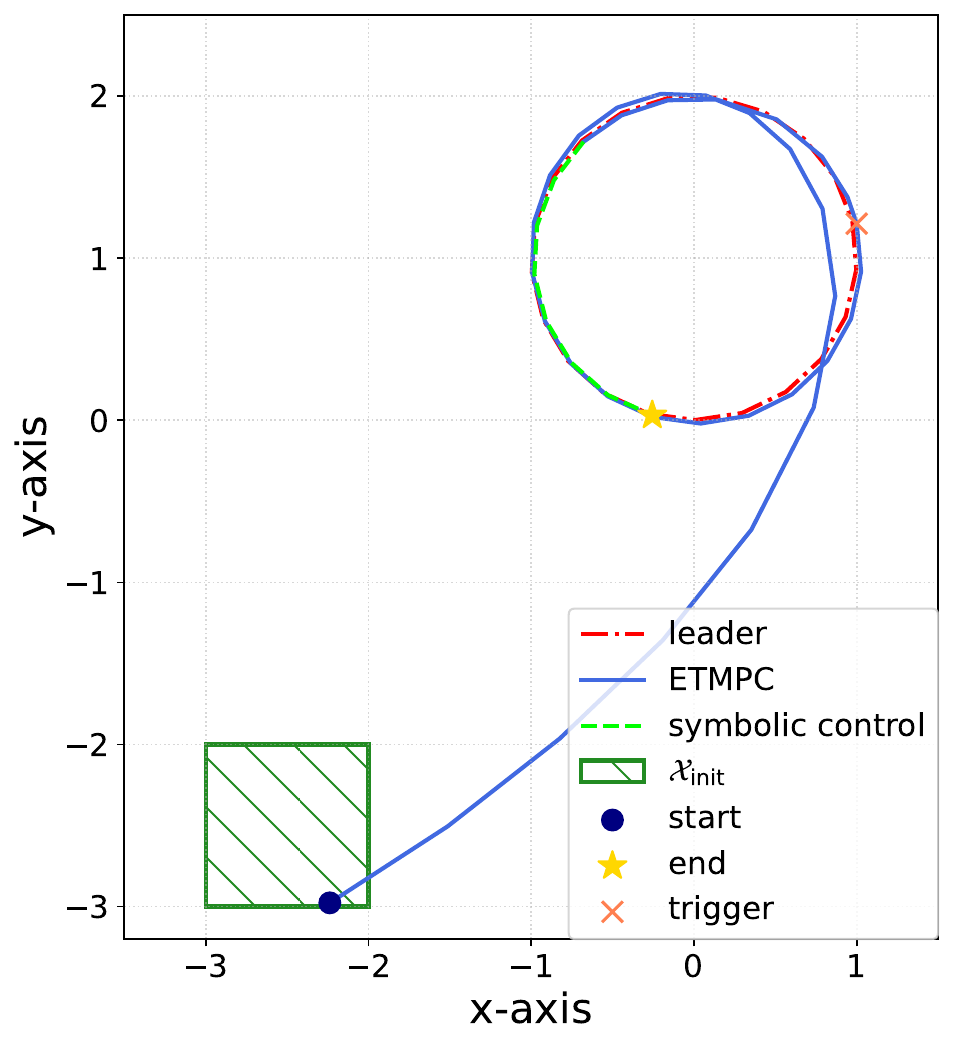}
	}
	\caption{Tracking trajectories of the robot $({\rm x}(t), {\rm y}(t))$ under the event-triggered MPC (blue solid line) and the symbolic controller (green dashed line), the positions where the event-triggered conditions \req{trigger} are violated (orange cross mark), and the reference trajectories $({\rm x}_r(t), {\rm y}_r(t))$ (red dash-dotted line). At each iteration, the initial state of the robot (blue circle mark) is given randomly from $\mathcal{X}_{\rm init}$ (green rectangle with slashes), and the algorithm proceeds to the next iteration after 40 steps have passed (yellow star mark).}
	\label{trajtrigger4}
\end{figure*}

\begin{figure*}[t]
    \centering
	\subfigure[Trajectories of the error distance.]{
	    \centering
		\includegraphics[width = 0.3\linewidth]{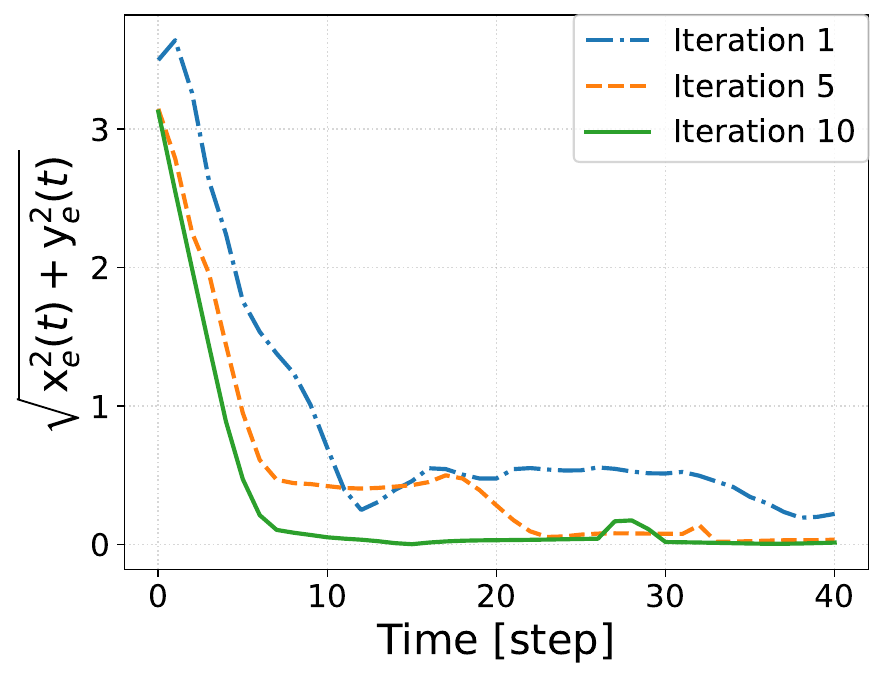}
	}
	\subfigure[Trajectories of the error attitude.]{
	    \centering
		\includegraphics[width = 0.3\linewidth]{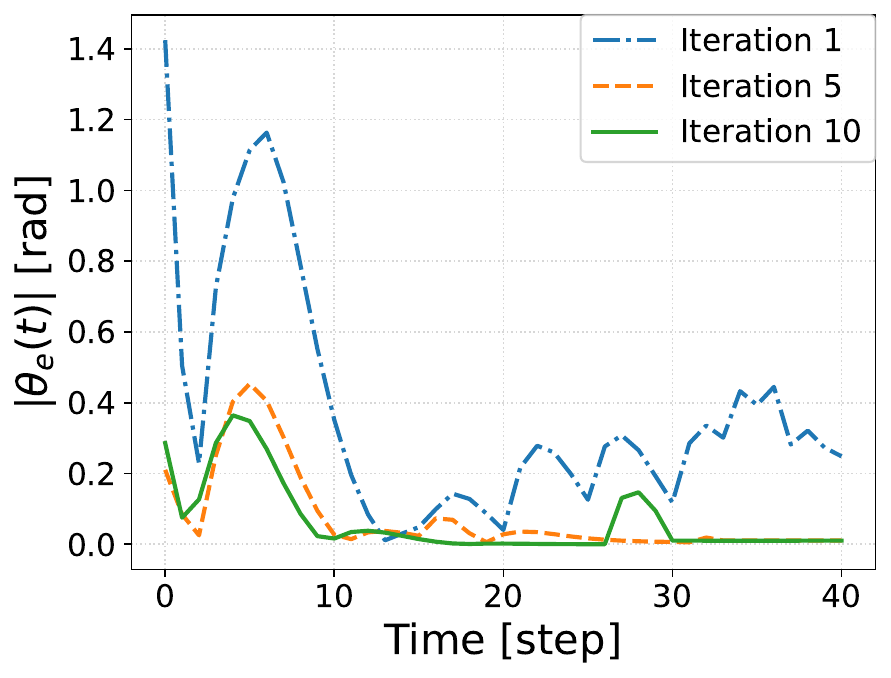}
	}
	\subfigure[Trajectories of the control input.]{
	    \centering
		\includegraphics[width = 0.3\linewidth]{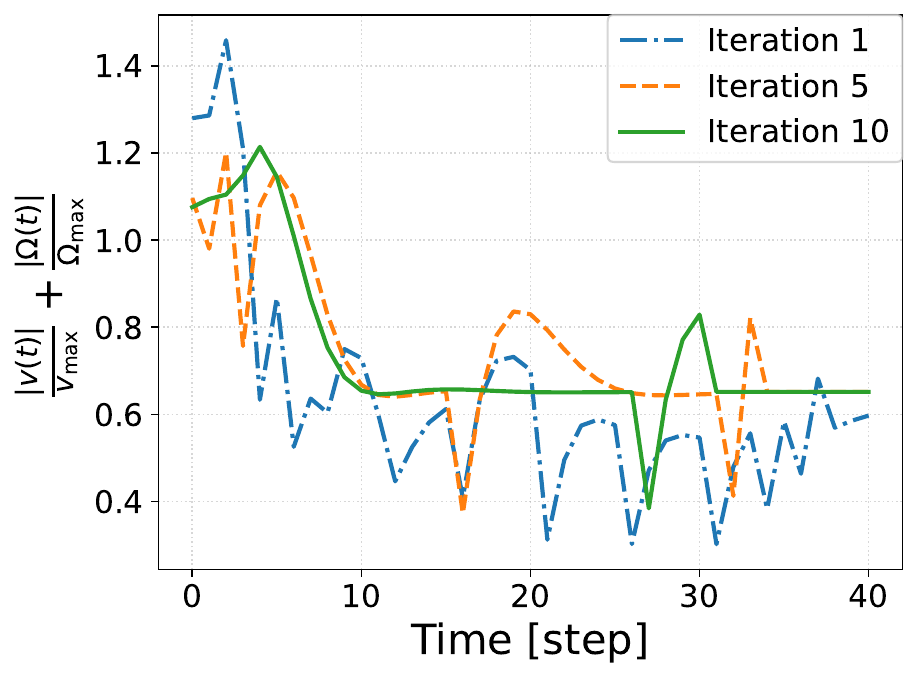}
	}
	\caption{(a) and (b) illustrate, at each iteration, the trajectories of the error distance and attitude between the robot \req{foldyn} and the leader \req{refdyn}, respectively. (c) shows the trajectories of the control input evaluated by the control input constraint at each iteration.}
	\label{trajtrigger5}
\end{figure*}

We consider a leader-follower tracking control problem of the nonholonomic mobile robot formulated by the following unicycle dynamics:
\begin{align}\label{foldyn}
\frac{{\rm d}x(t)}{{\rm d}t}= \left[
\begin{array}{c}
\dot{\rm x}(t) \\
\dot{\rm y}(t) \\
\dot{\theta}(t) \\
\end{array}
\right] = \left[
\begin{array}{cc}
\cos{\theta(t)} & 0\\
\sin{\theta(t)} & 0\\
0 & 1\\
\end{array}
\right]\left[
\begin{array}{c}
v(t) \\
\Omega(t) \\
\end{array}
\right],
\end{align}
for all $t \in \mathbb{R}_{\ge 0}$, where $x(t) = [{\rm x}(t), {\rm y}(t), \theta(t)]^\top \in \mathbb{R}^{3}$ is the state that represents the two-dimensional position at $t$, $({\rm x}(t), {\rm y}(t))$ and the angle $\theta(t)$, and $u(t) = [v(t), \Omega(t)]^\top \in \mathbb{R}^2$ is the control input that represents linear velocity $v(t)$ and angular velocity $\Omega(t)$. 
Moreover, the control input set is given by $\mathcal{U} = \{[v, \Omega]^\top \in \mathbb{R}^2: |v| \le v_{\rm max} = 3, \ |\Omega| \le \Omega_{\rm max} = \pi \}$.
The control objective in the simulation is to make the trajectory of the robot \req{foldyn} track the following reference trajectory of the leader $x_r(t) = [{\rm x}_r(t), {\rm y}_r(t), \theta_r(t)]^\top \in \mathbb{R}^3$:
\begin{align}\label{refdyn}
\frac{{\rm d}x_r(t)}{{\rm d}t}  &=\left[
\begin{array}{c}
\dot{\rm x}_r(t) \\
\dot{\rm y}_r(t) \\
\dot{\theta}_r(t) \\
\end{array}
\right] = \left[
\begin{array}{cc}
\cos{\theta_r(t)} & 0\\
\sin{\theta_r(t)} & 0\\
0 & 1\\
\end{array}
\right]\left[
\begin{array}{c}
v_r(t) \\
\Omega_r(t) \\
\end{array}
\right],
\end{align}
for all $t \in \mathbb{R}_{\ge 0}$, where the reference linear velocity $v_r(t)$, angular velocity $\Omega_r(t)$, and initial state configuration $x_r(0)$ are given by $v_r(t) = 1$, $\Omega_r(t) = 1$ for all $t \geq 0$, and $x_r(0) = [0, 0, 0]^\top$, respectively. To investigate that the robot \req{foldyn} will follow the reference trajectory \req{refdyn}, we define the error posture  $x_e(t) = [{\rm x}_e(t), {\rm y}_e(t), \theta_e(t)]^\top \in \mathbb{R}^3$ between the $x_r(t)$ and $x(t)$ as
\begin{align}\label{errorpos}
x_e(t) = \left[
\begin{array}{ccc}
\cos{\theta(t)} & \sin{\theta(t)} & 0\\
-\sin{\theta(t)} & \cos{\theta(t)} & 0\\
0 & 0 & 1\\
\end{array}
\right] (x_r(t) - x(t)),
\end{align}
for all $t \in \mathbb{R}_{\ge 0}$ (see, e.g., \cite{kanayama1990stable}). The dynamics of the error posture $x_e(t)$ derived by \req{foldyn} - \req{errorpos} is then given by
\begin{align}\label{errdyn}
\frac{{\rm d}x_e(t)}{{\rm d}t} &= \left[
\begin{array}{c}
\dot{\rm x}_e(t) \\
\dot{\rm y}_e(t) \\
\dot{\theta}_e(t) \\
\end{array}
\right]
=\left[
\begin{array}{c}
\Omega(t) {\rm y}_e(t) - v(t) + v_r(t) \cos{\theta_e(t)} \\
-\Omega(t) {\rm x}_e(t) + v_r(t) \sin{\theta_e(t)}\\
\Omega_r(t) - \Omega(t)\\
\end{array}
\right],
\end{align}
for all $t \in \mathbb{R}_{\ge 0}$. We assume that the error posture in \req{errdyn} is sampled at time interval $\Delta t \in \mathbb{R}_{> 0}$ and makes the transitions under the influence of additive noise $w(t) \in \mathbb{R}^3$. Therefore, we consider the perturbed system of the form:
\begin{align}\label{simdyn}
&x_e(t + 1) = f(x_e(t), u(t)) + w(t), \notag \\
&= \left[
\begin{array}{c}
{\rm x}_e(t) \\
{\rm y}_e(t) \\
{\theta}_e(t) \\
\end{array}
\right] \notag \\
&\ \ \ \ \ + \Delta t \left[
\begin{array}{c}
\Omega(t) {\rm y}_e(t) - v(t) + v_r(t) \cos{\theta_e(t)} \\
-\Omega(t) {\rm x}_e(t) + v_r(t) \sin{\theta_e(t)}\\
\Omega_r(t) - \Omega(t)\\
\end{array}
\right] + w(t)
\end{align}
for all $t \in \mathbb{N}_{\ge 0}$. Here, it is assumed that $\Delta t = 0.3$ and the additive noise is bounded as $w(t) \in \mathcal{W} = \{w \in \mathbb{R}^3: |w_i| \le \sigma_{w_i} = 0.01, \forall i \in \mathbb{N}_{1: 3}\}$. Note that the function $f$ is unknown apriori. 

\begin{table*}
\centering
  \begin{tabular}{|c||c|c|c|c|c|c|c|c|c|c||c|}  \hline
  Iteration & 1 & 2 & 3 & 4 & 5 & 6 & 7 & 8 & 9 & 10 & Periodic \\ \hline \hline
  Number of triggers & 11 & 7 & 8 & 5 & 4 & 4 & 3 & 1 & 2 & 1 &40 \\ \hline
  Number of data points & 25 & 43 & 44 & 49 & 53 & 62 & 76 & 78 & 81 & 82 & 82\\ \hline
  Total computation time [ms] & 582 & 376 & 433 & 291 & 254 & 283 & 179 & 87 & 114 & 74 & 1073 \\ \hline
  \end{tabular}
  \caption{Number of triggers, number of data points of the training dataset, and total computation time of \req{ocp} and \req{recopt} at each iteration by the proposed approach and the conventional periodic MPC (right end).}
  \label{triggerdatapoint}
\end{table*}

Based on the above system configuration, we consider tracking the robot \req{foldyn} to the reference trajectory \req{refdyn} by learning the system of the error posture \req{simdyn} with the GP regression and applying the event-triggered MPC in \ralg{itertask}. In the algorithm, the initial training dataset $\mathcal{D}_{N}$ with the number of the data points $N = 30$ are collected by either applying control inputs randomly from $\mathcal{U}$ or applying a PID controller whose control gains are heuristically chosen to reduce the tracking error. 
The initial prediction horizon $H_0$ and the stage cost function $h_s$ of the OCP \req{ocp} are set to be $H_0 = 30$ and $h_s(\hat{x}_e, u) = \|\hat{x}_e\| + \|u\|$, respectively. The set of the parameters for constructing the symbolic model in \rdef{symbolictransaug} are given by $\mathcal{X} = \{x \in \mathbb{R}^3: |x_i| \le 0.3, \forall i \in \mathbb{N}_{1: 3} \}$, $\eta_x = [0.01, 0.01, 0.01]^\top$, $\eta_u = [0.25, 0.25]^\top$, and $\mathsf{z} = \{3, 3, 3\}$. Furthermore, in order to suppress the number of the training dataset for learning the unknown function \req{simdyn}, we consider adding the data points of the input $z_e = [x_e^\top, u^\top]^\top$ and the output $f(x_e, u) + w$ to the training dataset $\mathcal{D}_N$ only when the variance of the GP model in \req{gpvar} exceeds a certain threshold. Specifically, the input-output data is added to the training dataset only when the condition to guarantee the recursive feasibility is violated, e.g., we set the training dataset as $\widetilde{\mathcal{D}}_i \leftarrow \widetilde{\mathcal{D}}_i \cup \{{z}(t_k + j), x_i(t_k + j + 1)\}$ only when $d_i(x(t_k+j), \hat{x}^*(j|t_k)) > \xi^*_i (j ; x(t_k))$. 

\rfig{trajtrigger4} (a)-(c) show the trajectories of the robot \req{foldyn} obtained by applying \ralg{itertask} for the cases of iteration $1, 5$, and $10$. The figures illustrate the reference trajectories $({\rm x}_r(t), {\rm y}_r(t))$ (red dash-dotted line), tracking trajectories of the robot $({\rm x}(t), {\rm y}(t))$ under the event-triggered MPC (blue solid line) and the symbolic controller (green dashed line), and the positions where the event-triggered conditions \req{trigger} are violated (orange cross mark). At each iteration, the initial state of the robot (blue circle mark) is given randomly from $\mathcal{X}_{\rm init} = [-2, -3]^\top \times [-2, -3]^\top \times [0, \pi]^\top$ (green rectangle with slashes), and the algorithm proceeds to the next iteration after 40 steps have passed (yellow star mark). In correspondence with \rfig{trajtrigger4}, \rfig{trajtrigger5} (a)-(b) show the trajectories of the error distance $\sqrt{{\rm x}_e^2(t) + {\rm y}_e^2(t)}$ and the error attitude $|\theta_e(t)|$ between the robot \req{foldyn} and the leader \req{refdyn} for each iteration. The figures show that although the robot \req{foldyn} can not follow the reference trajectory \req{refdyn} within 40 steps in the initial iteration, the tracking performance is improved and the control objective can be achieved as the number of iterations increases. Furthermore, \rfig{trajtrigger5} (c) provides the trajectories of the control input, in which $\frac{|v(t)|}{v_{\rm max}} + \frac{|\Omega(t)|}{\Omega_{\rm max}}$ is taken as an index to evaluate the control input constraint. 
\rtab{triggerdatapoint} reports the number of triggers where the event-triggered conditions \req{trigger} are violated, the number of data points of the training dataset $\mathcal{D}_N$, and the total computation time of the OCP \req{ocp} and the optimization problem \req{recopt} at each iteration until the number of iterations reaches 10. 
We observe that heavy computational load is required for the initial iteration, i.e., the OCP \req{ocp} is solved 12 times in 40 steps and the total computation time takes 582 ms. However, as the number of data points increases, the OCP \req{ocp} is solved less than twice in 40 steps and the total computation time only takes 74 ms, indicating that the frequency of solving the OCPs is significantly reduced. 
In the right end of \rtab{triggerdatapoint}, we also illustrate the result by the periodic MPC, where the OCP is solved every time based on the training data obtained at iteration~10. The table shows that the total computation time is 1073 ms, which is 14 times larger than the proposed approach.  
In summary of the simulation results, we can confirm that the number of solving the OCPs can be reduced by applying the proposed event-triggered strategy (especially as the number of iterations increases) and computational advantage over the periodic case. 

\section{Conclusion}\label{conclusionsec}
In this paper, we proposed a novel learning-based event-triggered MPC framework for nonlinear control systems whose dynamics are unknown apriori. The proposed framework formulated the OCP by a nominal model obtained from GP regression and a terminal set constructed by a symbolic abstraction. Moreover, the event-triggered condition was designed to guarantee recursive feasibility. In order to derive the terminal set and the event-triggered condition, we employed the deterministic error bound with the assumption that the underlying dynamics lie in the RKHS corresponding to the kernels utilized in GP regression. Furthermore, we provided convergence analysis with a concept of decreasing horizon strategy and showed that the finite-time convergence to the terminal set is achieved as the uncertainty of the GP model becomes smaller. Then, we proposed an iterative task algorithm for the learning-based event-triggered MPC to reduce the uncertainty of the GP model and update the control performance. Finally, we confirmed the validity of our proposed approach through a numerical simulation of a tracking control problem.


\appendix
\subsection{Proof of \rlem{continuity}}\label{continuityproof}
\begin{proof}
From \ras{rkhsas}, we have
\begin{align}\label{continuitybefore}
    &|f_i(x_1, u) - f_i(x_2, u)| \notag \\
    &\le b_i \cdot \sqrt{\mathsf{k}_i(z_1, z_1) - 2 \mathsf{k}_i(z_1, z_2) + \mathsf{k}_i(z_2, z_2)},
\end{align}
for all $i\in \mathbb{N}_{1: n_x}$, $x_1, x_2 \in \mathbb{R}^{n_x}$, and $u \in \mathcal{U}$, where $z_1 = [x_1^\top, u^\top]^\top, z_2 = [x_2^\top, u^\top]^\top$, see Lemma 4.28 in \cite{hearst1998support}. Moreover, from the definition of the SE kernel in \req{sekernel}, we can compute $\mathsf{k}_i(z_1, z_2)$ as
\begin{align}
    \mathsf{k}_i(z_1, z_2) &=  \alpha_i^2 \exp(-\frac{1}{2} \|z_1 - z_2\|^2_{\Lambda_i^{-1}}) \notag \\
    &= \alpha_i^2 \exp (-\frac{1}{2}\|x_1 - x_2\|^2_{\Lambda_{x, i}^{-1}} - \frac{1}{2}\|u - u\|^2_{\Lambda_{u, i}^{-1}}) \notag \\
    &= \mathsf{k}_{x, i}(x_1, x_2),
\end{align}
for all $i\in \mathbb{N}_{1: n_x}$, where $\Lambda_{u, i} = \diag{\lambda^u_{i, 1}, \ldots, \lambda^u_{i, n_u}}$. 
Therefore, \req{continuitybefore} follows that
\begin{align}
    &|f_i(x_1, u) - f_i(x_2, u)| \notag \\
    &\le b_i \cdot \sqrt{\mathsf{k}_{x, i}(x_1, x_1) - 2 \mathsf{k}_{x, i}(x_1, x_2) + \mathsf{k}_{x, i}(x_2, x_2)} \notag \\
    &= b_i \cdot d_i(x_1, x_2),
\end{align}
for all $i\in \mathbb{N}_{1: n_x}$, where $d_i(\cdot, \cdot)$ is the kernel metric defined in \req{kernelmetric}.
\end{proof}

\subsection{Proof of \rlem{inoutlem}}\label{inoutlemproof}
\begin{proof}
Given $\gamma_i \in [0, \sqrt{2}\alpha_i]$ ($i \in \mathbb{N}_{1: n_x}$), let $\mathcal{X}_1, \mathcal{X}_2$ be the any two sets satisfying \req{inoutlemeq}. Then, from the definition of $\mathsf{Int}_{d_i}(\cdot; \gamma_i)$ ($i \in \mathbb{N}_{1: n_x}$), the relation between $\mathcal{X}_1$ and $\mathcal{X}_2$ is characterized by the ball sets $\mathcal{B}_{d_i}(\cdot; \gamma_i)$ ($i \in \mathbb{N}_{1: n_x}$) as
\begin{align}\label{intdefdeploy}
\mathcal{X}_1 &\subseteq \bigcap_{i=1}^{n_x}\mathsf{Int}_{d_i}(\mathcal{X}_2; \gamma_i) = \bigcap_{i=1}^{n_x}\{{x}_1 \in \mathcal{X}_2: \mathcal{B}_{d_i}({x}_1; \gamma_i) \subseteq \mathcal{X}_2\}.
\end{align}
Hence, from \req{intdefdeploy}, given $x_1 \in \mathcal{X}_1$, we have
\begin{align}\label{balldefdeploy}
    \mathcal{B}_{d_i}(x_1; \gamma_i) \subseteq \mathcal{X}_2,
\end{align}
for all $i \in \mathbb{N}_{1: n_x}$.
Moreover, recall that the ball set of radius $\gamma_i$ centered at $x_1$ is defined as $\mathcal{B}_{d_i}({x}_1; \gamma_i) = \{{x}_2 \in \mathbb{R}^{n_x} : \ d_i({x}_1, {x}_2) \le \gamma_i\}$. 
Hence, if $x_1, x_2 \in \mathbb{R}^{n_x}$ fulfill $d_i(x_1, x_2) \le \gamma_i$, we then have
\begin{align}\label{balldefdeploy2}
    x_2 \in \mathcal{B}_{d_i}({x}_1; \gamma_i),
\end{align}
for all $i \in \mathbb{N}_{1: n_x}$. Therefore, given $x_1 \in \mathcal{X}_1$ and $x_2 \in \mathbb{R}^{n_x}$ satisfying $d_i(x_1, x_2) \le \gamma_i$ for all $i \in \mathbb{N}_{1: n_x}$, it follows from \req{balldefdeploy} and \req{balldefdeploy2} that $x_2 \in \mathcal{B}_{d_i}(x_1; \gamma_i) \subseteq \mathcal{X}_2$. 
\end{proof}

\subsection{Proof of \rlem{recoptfeasiblelem}}\label{recoptfeasiblelemproof}
\begin{proof}
Suppose that the OCP \req{ocp} at time $t_k$ ($k \in \mathbb{N}_{\ge 0}$) are feasible, and we obtain the optimal control sequence $\bfmath{u}^*(t_k)$ and the corresponding (predictive) states $\hat{\bfmath{x}}^*(t_k)$. Then, from \ras{datasetas}, the uncertainties of the optimal predictive states (i.e., ${\hat{\delta}_{N, i}(j| t_k)}$, $i \in \mathbb{N}_{1: n_x}, j \in \mathbb{N}_{0: H_k - 1}$) follows that $\hat{\delta}_{N, i}(j| t_k) \le \sqrt{2}\alpha_i b_i - \varsigma$ for all $i \in \mathbb{N}_{1: n_x}$ and $j \in \mathbb{N}_{0: H_k - 1}$.
Hence, we immediately obtain the result that the second constraints in \req{recopt} is not empty for all $i \in \mathbb{N}_{1: n_x}$ and $j \in \mathbb{N}_{0: H_k - 1}$, and there exist $\psi_i(j | t_k) \in \mathbb{R}_{> 0}$ that satisfies the second constraints in \req{recopt} for all $i \in \mathbb{N}_{1: n_x}$ and $j \in \mathbb{N}_{0: H_k - 1}$. Thus, given $\psi_i(j | t_k) \in \mathbb{R}_{>0}$ that satisfies the second constraints in \req{recopt} for all $i \in \mathbb{N}_{1: n_x}$ and $j \in \mathbb{N}_{0: H_k - 1}$, we can compute the infimum of the right hand side of the first constraints in \req{recopt} as $ \|{b}\odot \psi(j | t_k) \|^2_{\Lambda_{x, i}^{-1}} \ge \|\hat{\delta}_N(j | t_k)\|^2_{\Lambda^{-1}_{x, i}}$ for all $i \in \mathbb{N}_{1: n_x}$ and $j \in \mathbb{N}_{0: H_k - 1}$. Therefore, if we have
\begin{align}\label{mathincluprob}
  \|\hat{\delta}_N(j | t_k)\|^2_{\Lambda^{-1}_{x, i}} \le c_{N, i}^2(j + 1 | t_k), 
\end{align}
for all $i \in \mathbb{N}_{1: n_x}$ and $j \in \mathbb{N}_{0: H_k - 1}$, it is sufficient that the optimization problem \req{recopt} is feasible at every $t_k$ ($k \in \mathbb{N}_{\ge 0}$). Let us now prove that \req{mathincluprob} holds for all $i \in \mathbb{N}_{1: n_x}$ and $j \in \mathbb{N}_{0: H_k - 1}$ by mathematical induction as follows:

First, consider the case $j = H_k - 1$. From the definition of a set of scalars $\gamma_i$ in \ras{terminalsetas}, it follows that $d_i(0_{n_x}, \delta_N(x, u)) \le \gamma_i$ for all $x \in \mathcal{X}$, $u \in \mathcal{U}$, and $i \in \mathbb{N}_{1: n_x}$. Hence, the upper bound of $\|\hat{\delta}_N(H_k - 1 | t_k)\|^2_{\Lambda^{-1}_{x, i}}$ can be computed by
\begin{align}
    \|\hat{\delta}_N(H_k - 1 | t_k)\|^2_{\Lambda^{-1}_{x, i}} \le {2 \log\left(\frac{2 \alpha_i^2}{2 \alpha_i^2 - \gamma_i^2}\right)} = c^2_{N, i}(H_k | t_k)
\end{align}
for all $i \in \mathbb{N}_{1: n_x}$, which means that \req{mathincluprob} holds for the case $j = H_k - 1$. 

Next, consider the case $j = \ell$ with some $\ell \in \mathbb{N}_{1: H_k - 1}$, and suppose that \req{mathincluprob} holds for all $j \in \mathbb{N}_{\ell: H_k - 1}$, i.e., $\|\hat{\delta}_N(\ell | t_k)\|^2_{\Lambda^{-1}_{x, i}} \le c_{N, i}^2(\ell + 1 | t_k)$ for all $i \in \mathbb{N}_{1: n_x}$ and $j \in \mathbb{N}_{\ell: H_k - 1}$. Then, the optimization problem \req{recopt} for the case $j = \ell$ is feasible, and we obtain the solutions $\psi_i^*(\ell | t_k)$ for all $i \in \mathbb{N}_{1: n_x}$, i.e., $\|b \odot \psi^*(\ell | t_k)\|^2_{\Lambda^{-1}_{x, i}} \le c_{N, i}^2(\ell + 1 | t_k)$ for all $i \in \mathbb{N}_{1: n_x}$, where $\psi^* = [\psi_1^*, \ldots, \psi_{n_x}^*]^\top$. Furthermore, using $\psi_i^*(\ell | t_k)$ and the notations of $c_{N, i}$ in \req{cdef2}, the difference between $c^2_{N, i}(\ell | t_k)$ and $c^2_{N, i}(\ell + 1 | t_k)$ can be computed by
\begin{align}\label{cdiff}
    &c^2_{N, i}(\ell | t_k) - c^2_{N, i}(\ell + 1 | t_k)  \notag \\
    &\ge 2 \log\left(\frac{2 \alpha_i^2}{2 \alpha_i^2 - \psi_i^{*2}(\ell |t_k) }\right) - \|b \odot \psi^*(\ell | t_k)\|^2_{\Lambda^{-1}_{x, i}} \notag \\
    & = 2 \log \left\{\frac{2 \alpha_i^2 \exp \left(- \frac{1}{2} \| b \odot \psi^{*2}(\ell |t_k) \|^2_{\Lambda_{x, i}^{-1}}\right)}{2 \alpha_i^2 - \psi_i^{*2}(\ell |t_k) } \right\},
\end{align}
for all $i \in \mathbb{N}_{1: n_x}$. We now analyze the lower bound of \req{cdiff}, and define the following function $h_i: \mathbb{R}^{n_x}_{> 0} \rightarrow \mathbb{R}$:
\begin{align}\label{cdifffuc}
h_i(e) = e_i^2 + 2\alpha_i^2 \exp \left(- \frac{1}{2} \| b \odot {e} \|^2_{\Lambda_{x, i}^{-1}}\right) - 2\alpha_i^2,
\end{align}
for all $i \in \mathbb{N}_{1: n_x}$ with $e = [e_1, \ldots, e_{n_x}]^\top \in \mathbb{R}^{n_x}_{>0}$. The partial derivatives of $h_i$ are then given by
\begin{align}
    \frac{\partial h_i(e)}{\partial e_p} =& - \frac{2 \alpha_i^2 \cdot b_p^2}{\lambda_{i, p}} \cdot e_p \cdot \exp \left(- \frac{1}{2} \| b \odot {e} \|^2_{\Lambda_{x, i}^{-1}}\right) \notag \\
    &+ \left\{
    \begin{array}{cc}
    2 e_p \  & \ (p = i) \\
    0 & (p \neq i)
    \end{array}
    \right.,
\end{align}
for all $i, p \in \mathbb{N}_{1: n_x}$. Moreover, the second-order partial derivatives of $h_i$ are obtained by
\begin{align}
    \frac{\partial^2 h_i(e)}{\partial e_p\partial e_q} =& \frac{2 \alpha_i^2 b_p^2 b_q^2}{\lambda_{i, p}\lambda_{i, q}} \cdot \exp \left(- \frac{1}{2} \| b \odot {e} \|^2_{\Lambda_{x, i}^{-1}}\right) \cdot e_p \cdot e_q \notag \\
    &+
    \left\{
    \begin{array}{cc}
    2(1 - \frac{\alpha_i^2 \cdot b_p^2}{\lambda_{i, p}} \cdot \exp (- \frac{1}{2} \| b \odot {e} \|^2_{\Lambda_{x, i}^{-1}}) )  \ & (q = p) \\
    0 \ &  (q \neq p)
    \end{array}
    \right.
\end{align}
for all $i, p, q \in \mathbb{N}_{1: n_x}$. Furthermore, from \ras{datasetas}, it follows that $ \alpha_i b_p \le \sqrt{\lambda_{i, p}}$ for all $i, p \in \mathbb{N}_{1: n_x}$ and we have $\frac{\partial^2 h_i(e)}{\partial e_p\partial e_q} > 0$ for all $e \in \mathbb{R}_{>0}^{n_x}$ and $i, p, q \in \mathbb{N}_{1: n_x}$. Hence, $h_i(e)$, $e \in \mathbb{R}_{>0}^{n_x}$ ($i \in \mathbb{N}_{1: n_x}$) is convex. Moreover, we have $\lim_{e \rightarrow +0_{n_x}} h_i(e) = 0$ for all $i \in \mathbb{N}_{1: n_x}$. Thus, for all $e \in \mathbb{R}_{> 0}$ and $i \in \mathbb{N}_{1: n_x}$, it follows that $h_i(e) > 0$. Therefore, using $h_i(\psi^*(\ell | t_k)) > 0$ for any positive solutions $\psi_i^*(\ell | t_k)$ ($i \in \mathbb{N}_{1: n_x}$), we have
\begin{align}\label{cdifflog}
\frac{2 \alpha_i^2 \exp \left(- \frac{1}{2} \| b \odot \psi^{*2}(\ell |t_k) \|^2_{\Lambda_{x, i}^{-1}}\right)}{2 \alpha_i^2 - \psi_i^{*2}(\ell |t_k) } > 1,
\end{align}
for all $i \in \mathbb{N}_{1: n_x}$. From \req{cdiff} and \req{cdifflog}, we have $c_{N, i}(\ell + 1 | t_k) \le c_{N, i}(\ell | t_k)$ hold for all $i \in \mathbb{N}_{1: n_x}$. Hence, for $j = \ell - 1$, we have
\begin{align}
     \|\hat{\delta}_{N}(\ell - 1| t_k)\|^2_{\Lambda_{x, i}^{-1}} &\le c_{N, i}(H_k | t_k) \le \ldots c_{N, i}(\ell + 1 | t_k) \notag \\
     &\le c_{N, i}((\ell - 1) + 1| t_k),
\end{align}
for all $i \in \mathbb{N}_{1: n_x}$. That is, \req{mathincluprob} holds for the case $j = \ell - 1$. Thus, it follows from mathematical induction that \req{mathincluprob} holds for all $i \in \mathbb{N}_{1: n_x}$ and $j \in \mathbb{N}_{0: H_k - 1}$. Hence, the optimization problem \req{recopt} is feasible at every $t_k$ ($k \in \mathbb{N}_{\ge 0}$). 
\end{proof}

\subsection{Proof of Theorem 1}\label{feasiblethmproof}
In order to prove \rlem{feasiblethm}, we need to resort the following lemma, in which upper bounds of kernel metrics between two nominal states and between an actual state and the nominal one are derived:

\begin{lem}\label{metricbound}
\normalfont
Given the training dataset $\mathcal{D}_N = \{\mathcal{D}_{N, i}\}^{n_x} _{i=1}$, and let \ras{rkhsas} hold. 
Then, the kernel metrics between two nominal states are given by
\begin{align}
& d_i(\hat{{f}}({x}_1, {u}; \mathcal{D}_{N}), \hat{{f}}({x}_2, {u}; \mathcal{D}_{N})) \le \hat{\zeta}_{N, i}({x}_1, {x}_2),
\end{align}
for all $x_1, x_2 \in \mathbb{R}^{n_x}$, $u \in \mathbb{R}^{n_u}$, and $i \in \mathbb{N}_{1: n_x}$, where
\begin{align}\label{zetahatdef}
&\hat{\zeta}_{N, i}({x}_1, {x}_2) = \sqrt{2\alpha_i^2\left\{1 - \exp \left(- \frac{1}{2} \| {b}\odot {d}({x}_1, {x}_2) \|^2_{\Lambda_{x, i}^{-1}}\right) \right\}}
\end{align}
with $d = [d_1, \ldots, d_{n_x}]^\top$. Moreover, the kernel metrics between an actual state and the nominal one are given by
\begin{align}
& d_i({f}({x}_1, {u}) + {w}, \hat{{f}}({x}_2, {u}; \mathcal{D}_{N})) \le \zeta_{N, i}({x}_1, {x}_2, {u}),
\end{align}
for all $x_1, x_2 \in \mathbb{R}^{n_x}$, $u \in \mathbb{R}^{n_u}$, and $i \in \mathbb{N}_{1: n_x}$, where 
\begin{align}\label{zetadef}
&\zeta_{N, i}({x}_1, {x}_2, {u}) \notag \\
&= \sqrt{2\alpha_i^2\left\{1 - \exp \left(- \frac{1}{2} \| {b}\odot {d}({x}_1, {x}_2) + \delta_N(x_2, u) \|^2_{\Lambda_{x, i}^{-1}}   \right) \right\}}
\end{align}
with $d = [d_1, \ldots, d_{n_x}]^\top$ and $\delta_N = [\delta_{N, 1}, \ldots, \delta_{N, n_x}]^\top$.
\qedwhite
\end{lem}

(\textit{proof} of \rlem{metricbound}):
Given the training dataset $\mathcal{D}_{N} = \{\mathcal{D}_{N, i}\}_{i = 1}^{n_x}$, from the definition of \req{kernelmetric}, the kernel metric between the two nominal states $\hat{f}({x}_1, {u}; \mathcal{D}_{N})$ and $\hat{f}({x}_2, {u}; \mathcal{D}_{N})$ can be computed by
\begin{align}\label{kernelmetricdeploynominal}
&d_i(\hat{{f}}({x}_1, {u}; \mathcal{D}_{N, i}), \hat{{f}}({x}_2, {u}; \mathcal{D}_{N})) \notag \\
&= \sqrt{2\alpha_i^2 - 2 \mathsf{k}_{x, i}(\hat{{f}}({x}_1, {u}; \mathcal{D}_{N, i}), \hat{{f}}({x}_2, {u}; \mathcal{D}_{N}))} \notag \\
&= \sqrt{2\alpha_i^2\{1 - \exp (- \frac{1}{2}\|\hat{{f}}({x}_1, {u}; \mathcal{D}_{N, i}) - \hat{{f}}({x}_2, {u}; \mathcal{D}_{N})\|^2_{\Lambda_{x, i}^{-1}}) \}},
\end{align}
for all $i \in \mathbb{N}_{1: n_x}$. Moreover, from \rlem{continuity}, we then have
\begin{align}\label{appcontiunitynominal}
|\hat{f}_i({x}_1, {u}; \mathcal{D}_{N, i}) - \hat{f}_i({x}_2, {u}; \mathcal{D}_{N, i}) | \le b_i \cdot d_i({x}_1, {x}_2)
\end{align}
for all $i \in \mathbb{N}_{1: n_x}$. Hence, it follows from \req{appcontiunitynominal} and \req{kernelmetricdeploynominal} that
\begin{align}
&d_i(\hat{f}({x}_1, {u}; \mathcal{D}_N), \hat{{f}}({x}_2, {u}; \mathcal{D}_{N})) \notag \\
&\le \sqrt{2\alpha_i^2\left\{1 - \exp \left(- \frac{1}{2} \sum_{j = 1}^{n_x}\frac{1}{\lambda^x_{i, j}} |b_i \cdot d_i({x}_1, {x}_2)|^2  \right) \right\}} \notag \\
&= \hat{\zeta}_{N, i}({x}_1, {x}_2, {u}),
\end{align}
for all $i \in \mathbb{N}_{1: n_x}$. In the similar way as the above, the upper bound of the kernel metric between the actual state ${f}({x}_1, {u}) + {w}$ and the nominal one $\hat{{f}}({x}_2, {u}; \mathcal{D}_{N})$ can be computed by
\begin{align}\label{kernelmetricdeploy}
&d_i({f}({x}_1, {u}) + {w}, \hat{{f}}({x}_2, {u}; \mathcal{D}_{N})) \notag \\
&= \sqrt{2\alpha_i^2 - 2 \mathsf{k}_{x, i}({f}({x}_1, {u}) + {w}, \hat{{f}}({x}_2, {u}; \mathcal{D}_N))} \notag \\
&= \sqrt{2\alpha_i^2\left\{1 - \exp \left(- \frac{1}{2}\|{f}({x}_1, {u}) + {w} - \hat{{f}}({x}_2, {u}; \mathcal{D}_{N})\|^2_{\Lambda_{x, i}^{-1}} \right) \right\}}
\end{align}
for all $i \in \mathbb{N}_{1: n_x}$. Moreover, from \rlem{errorbound} and \rlem{continuity}, we then have
\begin{align}\label{appcontiunity}
&|f_i({x}_1, {u}) + w_i - \hat{f}_i({x}_2, {u}; \mathcal{D}_{N, i})| \notag \\
&\le |f_i({x}_1, {u}) - f_i({x}_2, {u})| + |f_i({x}_2, {u}) - \hat{f}_i({x}_2, {u}; \mathcal{D}_{N, i})| + \sigma_{w_i} \notag \\
&\le b_i \cdot d_i({x}_1, {x}_2) + \beta_{N, i} \cdot \sigma_{f_i}(x_2, u; \mathcal{D}_{N, i}) + \sigma_{w_i}, \notag \\
&= b_i \cdot d_i({x}_1, {x}_2) + \delta_{N, i}(x_2, u),
\end{align}
for all $i \in \mathbb{N}_{1: n_x}$. Hence, from \req{kernelmetricdeploy} and \req{appcontiunity}, it follows that
\begin{align}
&d_i({f}({x}_1, {u}) + {w}, \hat{{f}}({x}_2, {u}; \mathcal{D}_{N})) \le \zeta_{N, i}({x}_1, {x}_2, {u}),
\end{align}
for all $i \in \mathbb{N}_{1: n_x}$.
\qedwhite\\

Based on \rlem{metricbound}, let us now prove \rthm{feasiblethm} as follows:

\begin{proof}
Suppose that the OCP at $t_k$ is feasible, and we have the optimal control sequence $\bfmath{u}^*(t_k)$, the corresponding (predictive) states $\hat{\bfmath{x}}^*(t_k)$. Then, from \rlem{recoptfeasiblelem}, the optimization problem \req{recopt} is feasible and the set of parameters $\{\xi_i^*(j ; x(t_k))\}_{i = 1}^{n_x}$ of the event-triggered conditions \req{trigger} are obtained. Moreover, suppose that the next OCP time $t_{k + 1}$ is determined according to \req{trigger} and let the inter-event time given by $m_k = t_{k + 1} - t_k$. In addition, let $x(t_k), \ldots, x(t_k + m_k)$ be the actual state sequence of the system \req{realdynamics} by applying ${u}^*(0| t_k), \ldots, {u}^* (t_{k} + m_k - 1| t_k)$. 
Then, for every choice of $m_k \in \mathbb{N}_{1: H_k}$, it follows that 
\begin{align*}
d_i({x}(t_k + m_k - 1), \hat{{x}}^*(m_k - 1 | t_k)) \le \xi_i^* (m_k - 1 ; x(t_k)), 
\end{align*}
for all $i \in \mathbb{N}_{1: n_x}$, which is due that, for each $m_k \in \mathbb{N}_{1: H_k}$, the event-triggered conditions \req{trigger} are satisfied until $t_{k} + m_k - 1$, i.e., $d_i({x}(t_k + j), \hat{{x}}^*(j | t_k)) \leq \xi_i^* (j ; x(t_k))$ for all $i \in \mathbb{N}_{1: n_x}$ and $j \in \mathbb{N}_{0: m_k - 1}$. Thus, from \rlem{metricbound}, the kernel metric between the actual state $x(t_{k + 1})$ and the nominal one $\hat{{x}}^*(m_k| t_k)$ is given by 
\begin{align}\label{kernelmetricboundcompute}
&d_i({x}(t_{k + 1}), \hat{{x}}^*(m_k| t_k)) = d_i({x}(t_k + m_k), \hat{{x}}^*(m_k| t_k)) \notag \\
&\le \zeta_{N, i}({x}(t_k + m_k - 1), \hat{{x}}^*(m_k - 1 | t_k), \hat{{u}}^*(m_k - 1 | t_k)), \notag \\
\end{align}
for all $i \in \mathbb{N}_{1: n_x}$.
Moreover, since $\xi_i^* (m_k - 1 ; x(t_k))$, $i \in \mathbb{N}_{1: n_x}$ are defined in \req{xidef} and $\psi_i^* (m_k - 1 | t_k)$, $i \in \mathbb{N}_{1:n_x}$ are the optimal solution of \req{recopt}, we then have
\begin{align}\label{rhobound}
&\| {b}\odot \xi^*(m_k - 1 ; x(t_k)) + \hat{\delta}_N(m_k - 1 | t_k) \|^2_{\Lambda_{x, i}^{-1}} \notag \\
&= \| {b}\odot \psi^*(m_k - 1 | t_k) \|^2_{\Lambda_{x, i}^{-1}} \le c^2_{N, i}(m_k | t_k),
\end{align}
for all $i \in \mathbb{N}_{1: n_x}$. Using \req{kernelmetricboundcompute}, \req{rhobound}, the notation of $c_{N, i}$ in \req{cdef1} and \req{cdef2}, the upper bound of $d_i({x}(t_{k + 1}), \hat{{x}}^*(m_k| t_k))$ is obtained as 
\begin{align}\label{metricupperthmdeploy}
&d_i({x}(t_{k + 1}), \hat{{x}}^*(m_k| t_k)) \notag \\
&\le \sqrt{2\alpha_i^2\left\{1 - \exp \left(- \frac{1}{2} c_{N, i}^2(m_k | t_k)\right) \right\}} \notag \\
&\le \begin{cases}
\gamma_i, & \ \mathrm{if}\ m_k = H_k\\
\psi_i^*(m_k| t_k), & \ \mathrm{if}\ m_k \in \mathbb{N}_{1: H_k - 1}
\end{cases},
\end{align}
for all $i \in \mathbb{N}_{1: n_x}$. Using \req{metricupperthmdeploy}, let us now prove the feasibility of the OCP at $t_{k + 1}$ holds. In particular, it is shown that there exists a feasible control sequence for the OCP at time $t_{k + 1}$ that satisfies the terminal constraint, i.e., the state $x(t_{k+1})$ is driven into $\mathcal{X}_f$ within the prediction horizon $H_{k + 1} = H_k -m_k + 1$. 

First, consider the case $m_k = H_k$. From \req{metricupperthmdeploy}, we then obtain $d_i({x}(t_{k + 1}), \hat{{x}}^*(H_k| t_k)) \leq \gamma_i$ for all $i \in \mathbb{N}_{1: n_x}$. 
In addition, since the OCP is feasible at $t_k$, it fulfills the terminal constraint, i.e, $\hat{x}^*(H_k | t_k) \in \mathcal{X}_f$. Hence, if $m_k = H_k$,  it follows from \req{metricupperthmdeploy}, \rlem{inoutlem}, \ras{terminalsetas}, and $\hat{x}^*(H_k | t_k) \in \mathcal{X}_f$ that $x(t_{k + 1}) \in \mathcal{X}_S$. Thus, from \ras{terminalsetas}, there exists $u = C(x(t_{k + 1}))$, such that $\hat{x}(t_{k + 1} + 1) = \hat{f}(x(t_{k + 1}), u) \in \mathcal{F}(x(t_{k + 1}), u; \mathcal{D}_N) \oplus \mathcal{W} \subseteq \mathcal{X}_f$, satisfying the terminal constraint. Therefore, if $m_k = H_k$, the OCP at $t_{k + 1}$ is feasible with the prediction horizon $H_{k+1} = H_k - m_k + 1 = 1$.

Next, consider the case $m_k \in \mathbb{N}_{1: H_k - 1}$. From \req{metricupperthmdeploy}, it follows that $d_i({x}(t_{k + 1}), \hat{{x}}^*(m_k| t_k)) \leq \psi_i^*(m_k | t_k)$ for all $i \in \mathbb{N}_{1: n_x}$. 
Consider a candidate feasible control sequence for the OCP at $t_{k+1}$ by $u^*(m_k| t_k), \ldots ,u^*(H_k - 1| t_k)$, and let $\hat{x}(j | t_{k + 1})$, $j \in \mathbb{N}_{0: H_k - m_k - 1}$ be the corresponding predictive states from $x(t_{k+1})$, i.e., $\hat{x}(j + 1 | t_{k + 1}) = \hat{f}(\hat{x}(j | t_{k + 1}), u^*(m_k + j|t_k))$ for all $j \in \mathbb{N}_{0: H_k - m_k - 1}$ with $\hat{x}(0 | t_{k + 1}) = x(t_{k+1})$. Then, applying the first control element $u^*(m_k| t_k)$, the kernel metric between the two nominal state $\hat{x}(1 |t_{k + 1})$ and $\hat{x}^*(m_k + 1|t_k)$ can be computed as
\begin{align}\label{kernelmetricnominalboundcompute}
&d_i(\hat{x}(1 |t_{k + 1}), \hat{x}^*(m_k + 1|t_k)) \notag \\
&\le \hat{\zeta}_{N, i}(\hat{x}(0 |t_{k + 1}), \hat{x}^*(m_k |t_k)) \notag \\
&\le \sqrt{2\alpha_i^2\left\{1 - \exp \left(- \frac{1}{2} \| {b}\odot \psi^*(m_k | t_k) \|^2_{\Lambda_{x, i}^{-1}}\right) \right\}},
\end{align}
for all $i \in \mathbb{N}_{1: n_x}$, where we used \rlem{metricbound} and $d_i({x}(t_{k + 1})$,
$\hat{{x}}^*(m_k| t_k)) \leq \psi_i^*(m_k | t_k)$ for all $i \in \mathbb{N}_{1: n_x}$. 
Using \req{kernelmetricnominalboundcompute}, the notation of $c_{N, i}$ in \req{cdef1} and \req{cdef2}, and $\| {b}\odot \psi^*(m_k | t_k) \|^2_{\Lambda_{x, i}^{-1}}  \le c^2_{N, i}(m_k + 1 | t_k)$, the upper bound of $d_i(\hat{x}(1 |t_{k + 1}), \hat{x}^*(m_k + 1|t_k)) $ is given by
\begin{align}\label{metricupperthmdeploy2}
&d_i(\hat{x}(1 |t_{k + 1}), \hat{x}^*(m_k + 1|t_k)) \notag \\
&\le \sqrt{2\alpha_i^2\left\{1 - \exp \left(- \frac{1}{2} c_{N, i}^2(m_k + 1 | t_k)\right) \right\}} \notag \\ 
&\le \begin{cases}
\gamma_i, & \ \mathrm{if}\ m_k = H_k - 1\\
\psi_i^*(m_k + 1| t_k), & \ \mathrm{if}\  m_k \in \mathbb{N}_{1: H_k - 2}
\end{cases},
\end{align}
for all $i \in \mathbb{N}_{1: n_x}$. Hence, if $m_k = H_k - 1$, it follows from \req{metricupperthmdeploy2}, \rlem{inoutlem}, \ras{terminalsetas}, and $\hat{x}^*(H_k | t_k) \in \mathcal{X}_f$ that $\hat{x}(1 |t_{k + 1}) \in \mathcal{X}_S$. 
Thus, from \ras{terminalsetas}, there exists $u = C(\hat{x}(1 |t_{k + 1}))$, such that $\hat{x}(2 |t_{k + 1}) = \hat{f}(\hat{x}(1 |t_{k + 1}), u) \in \mathcal{X}_f$, satisfying the terminal constraint. Therefore, if $m_k = H_k - 1$, the OCP at $t_{k + 1}$ is feasible with the prediction horizon $H_{k+1} = H_k - m_k + 1 = 2$. If $m_k \in \mathbb{N}_{1: H_k - 2}$, it follows that $d_i(\hat{x}(1 |t_{k + 1}), \hat{x}^*(m_k + 1|t_k)) \le \psi_i^*(m_k + 1| t_k)$ for all $i \in \mathbb{N}_{1: n_x}$. 
Hence, we can repeat the same procedure as above and the upper bound of the kernel metric between the two nominal state $\hat{x}(H_k - m_k| t_{k + 1})$ and $\hat{x}^*(H_k | t_{k})$ is given by
\begin{align}\label{metricupperthmdeploy3}
d_i(\hat{x}(H_k - m_k| t_{k + 1}), \hat{x}^*(H_k | t_k)) \le \gamma_i,
\end{align}
for all $i \in \mathbb{N}_{1: n_x}$, and therefore, $\hat{x}(H_k - m_k |t_{k + 1}) \in \mathcal{X}_S$. 
Thus, from \ras{terminalsetas}, there exists $u = C(\hat{x}(H_k - m_k |t_{k + 1}))$, such that $\hat{x}(H_k - m_k + 1 |t_{k + 1}) = \hat{f}(\hat{x}(H_k - m_k |t_{k + 1}), u) \in \mathcal{X}_f$, satisfying the terminal constraint. In other words, the OCP at $t_{k + 1}$ is shown to be feasible with the prediction horizon $H_{k+1} = H_k - m_k + 1$. \\
\end{proof}

\subsection{Proof of \rlem{horizonlem}}\label{horizonlemproof}
\begin{proof}
Suppose that the OCP and \req{recopt} are both feasible at $t_k$, and we obtain the optimal control sequence $\bfmath{u}^*(t_k)$, the corresponding (predictive) states $\hat{\bfmath{x}}^*(t_k)$, and ${\xi^*(j ; x(t_k))}_{i = 1}^{n_x}$. Moreover, suppose that the next OCP time $t_{k + 1}$ is determined according to \req{trigger} and let the prediction horizon updated by $H_{k + 1} = H_{k} - m_k + 1 = 1$. Then, from $m_k = t_{k + 1} - t_k$, we have $t_{k + 1} = t_k + H_k$. 
Thus, from \rlem{metricbound}, the kernel metric between the actual state $x(t_{k + 1})$ and the nominal one $\hat{{x}}^*(H_{k}|t_{k})$ is given by 
\begin{align}\label{kernelmetricboundcomputethm}
&d_i({x}(t_{k + 1}), \hat{{x}}^*(H_{k}| t_{k})) = d_i({x}(t_{k} + H_{k}), \hat{{x}}^*(H_{k}| t_{k})) \notag \\
&\le \zeta_{N, i}({x}(t_k + H_k - 1), \hat{{x}}^*(H_k - 1| t_k), \hat{{u}}^*(H_k - 1 | t_k)),
\end{align}
for all $i \in \mathbb{N}_{1: n_x}$.
Moreover, since $\xi_i^* (H_k - 1 ; x(t_k))$, $i \in \mathbb{N}_{1: n_x}$ are defined in \req{xidef} and $\psi_i^* (H_k - 1 | t_k)$, $i \in \mathbb{N}_{1:n_x}$ are the optimal solution of \req{recopt}, we then have
\begin{align}\label{rhoboundthm}
&\| {b}\odot \xi^*(H_{k} - 1 ; x(t_k)) + \hat{\delta}_N(m_{k} - 1 | t_{k}) \|^2_{\Lambda_{x, i}^{-1}} \notag \\
&= \| {b}\odot \psi^*(m_k - 1 | t_k) \|^2_{\Lambda_{x, i}^{-1}} \le c^2_{N, i}(H_{k} | t_{k}),
\end{align}
for all $i \in \mathbb{N}_{1: n_x}$. Using \req{kernelmetricboundcomputethm}, \req{rhoboundthm} and the notation of $c_{N, i}$ defined in \req{cdef1}, the upper bound of $d_i({x}(t_{k + 1}), \hat{{x}}^*(H_{k}| t_k))$ is obtained as 
\begin{align}\label{metricupperthmdeploythm}
&d_i({x}(t_{k + 1}), \hat{{x}}^*(H_{k}| t_{k})) \notag \\
&\le \sqrt{2\alpha_i^2\left\{1 - \exp \left(- \frac{1}{2} c_{N, i}^2(H_{k} | t_{k})\right) \right\}} \le \gamma_i,
\end{align}
for all $i \in \mathbb{N}_{1: n_x}$. In addition, since the OCP is feasible at $t_{k}$, it fulfills the terminal constraint, i.e., $\hat{x}^*(H_{k} | t_{k}) \in \mathcal{X}_f$.  Hence, it follows from \req{metricupperthmdeploythm}, \rlem{inoutlem}, \ras{terminalsetas}, and $\hat{x}^*(H_{k} | t_{k}) \in \mathcal{X}_f$ that $x(t_{k + 1}) \in \mathcal{X}_S$. 
\end{proof}

\subsection{Proof of \rprop{varepsilonasrprop}}\label{varepsilonasrpropproof}
\begin{proof}
The proof is provided by showing that the relation $R(\varepsilon)$ satisfies all the conditions (C.1) - (C.3) of the $\varepsilon$-ASR from $\Sigma_\mathsf{q}$ to $\Sigma$ in \rdef{asrdef}. We start by proving the condition (C. 1) in \rdef{asrdef} holds. Consider the relation $R(\varepsilon)$ given by \req{epsilonparam}. If $x \in \mathbb{R}^{n_x}$ then, by the definition of $\mathcal{X}_\mathsf{q} = [\mathbb{R}^{n_x}]_{\eta_x}$, there exists $x_\mathsf{q} \in \mathcal{X}_\mathsf{q}$ such that $|x_{\mathsf{q}, i} - x_i| \le \eta_{x, i}$ for all $i \in \mathbb{N}_{1: n_x}$. It follows that
\begin{align}
&d_i({x}_\mathsf{q}, {x}) = \sqrt{2\alpha_i^2\left\{1 - \exp\left(-\frac{1}{2}\|{x}_\mathsf{q} - {x}\|^2_{\Lambda_{x, i}^{-1}}\right) \right\}} \notag \\
&\le \sqrt{2\alpha_i^2\left\{1 - \exp\left(-\frac{1}{2}\|{ \eta}_x\|^2_{\Lambda_{x, i}^{-1}} \right) \right\}} \le \epsilon_i,
\end{align}
for all $i \in \mathbb{N}_{1: n_x}$, and hence the condition (C. 1) is satisfied.
Moreover, the condition (C. 2) is satisfied from the definition of $R(\varepsilon)$ in \req{varepsilonasr}. 
Regarding the condition (C. 3), let $({x}_\mathsf{q}, {x}) \in R(\varepsilon)$ and $u = {u}_\mathsf{q} \in \mathcal{U}_\mathsf{q} \subset \mathcal{U}$ and consider the transition ${x}^+ \in G({x}, {u})$ in $\Sigma$. It follows from the definition of $\mathcal{X}_\mathsf{q} = [\mathbb{R}^{n_x}]_{\eta_x}$ that there exists $x_\mathsf{q}^+ \in \mathcal{X}_\mathsf{q}$ such that $|x_{\mathsf{q}, i}^+ - x_i^+| \le \eta_{x, i}$ for all $i \in \mathbb{N}_{1: n_x}$.
Now, let us show that ${x}_\mathsf{q}^+ \in G_\mathsf{q}({x}_\mathsf{q}, {u}_\mathsf{q})$. From \rlem{errorbound} and \rlem{continuity}, we obtain
\begin{align}
&|x_{\mathsf{q}, i}^+ - \hat{f}_i({x}_\mathsf{q}, {u}_\mathsf{q}; \mathcal{D}_{N, i})| \notag \\
&\le  |x_{i}^+ - \hat{f}_i({x}_\mathsf{q}, {u}_\mathsf{q}; \mathcal{D}_{N, i})| + \eta_{x, i} \notag \\
&\le  |f_i(x, u) + \sigma_{w_i} -\hat{f}_i({x}_\mathsf{q}, {u}_\mathsf{q}; \mathcal{D}_{N, i})| + \eta_{x, i} \notag \\
&\le |f_i({x}, {u}) - f_i({x}_\mathsf{q}, {u}_\mathsf{q})| + |f_i({x}_\mathsf{q}, {u}_\mathsf{q}) - \hat{f}_i({x}_\mathsf{q}, {u}_\mathsf{q}; \mathcal{D}_{N, i})| \notag \\
&\ \ \ \ \ \ \ \ \ \ \ \ \ \ \ \ \ \ \ \ \ \ \ \ \ \ \ \ \ + \sigma_{w_i} + \eta_{x, i} \notag \\
&\le b_i \cdot d_i ({x}, {x}_\mathsf{q}) +  \beta_{N, i} \cdot \sigma_{f_i}({x}_\mathsf{q}, u_\mathsf{q}; \mathcal{D}_{N, i}) + \sigma_{w_i} + \eta_{x, i} \notag \\
&\le b_i \cdot \epsilon_i + \beta_{N, i} \cdot \sigma_{f_i}({x}_\mathsf{q}, u_\mathsf{q}; \mathcal{D}_{N, i}) + \sigma_{w_i} + \eta_{x, i},
\end{align}
for all $i \in \mathbb{N}_{1: n_x}$, where ${z}_\mathsf{q} = [{x}^\top_\mathsf{q}, {u}^\top_\mathsf{q}]^\top$. From the above, it then follows that $x_{\mathsf{q}, i}^+ \in \mathcal{F}_i({x}_\mathsf{q}, {u}_\mathsf{q}; \mathcal{D}_{N, i}) \oplus \mathcal{W}_i \oplus \mathcal{E}_i$ for all $i \in \mathbb{N}_{1: n_x}$. This directly means that ${x}_\mathsf{q}^+ \in G_\mathsf{q}({x}_\mathsf{q}, {u}_\mathsf{q})$ with $({x}_\mathsf{q}^+, {x}^+) \in R(\varepsilon)$. Hence, $R(\varepsilon)$ in \req{varepsilonasr} is an $\varepsilon$-ASR from $\Sigma_\mathsf{q}$ to $\Sigma$.
\end{proof}

\subsection{Proof of \rlem{intoutlemdis}}\label{intoutlemdisproof}
\begin{proof}
Given $\widetilde{\gamma}_i$ by \req{contractiveparam} for all $i \in \mathbb{N}_{1: n_x}$, and $\mathcal{X}_{\mathsf{q}1}, \mathcal{X}_{\mathsf{q}2} \subseteq [\mathbb{R}^{n_x}]_{{ \eta}_x}$ be the any two sets satisfying $\mathcal{X}_{\mathsf{q}1} \subseteq \bigcap_{i=1}^{n_x}\widetilde{\mathsf{Int}}_{d_i}(\mathcal{X}_{\mathsf{q}2}; \widetilde{\gamma}_i)$. Then, from the definition of $\widetilde{\mathsf{Int}}_{d_i}(\cdot; \widetilde{\gamma}_i)$ ($i \in \mathbb{N}_{1: n_x}$), the relation between $\mathcal{X}_{\mathsf{q}1}$ and $\mathcal{X}_{\mathsf{q}2}$ is characterized by ball sets $\widetilde{\mathcal{B}}_{d_i}(\cdot; \widetilde{\gamma}_i)$ ($i \in \mathbb{N}_{1: n_x}$) as
\begin{align}\label{intdefdeploydis}
\mathcal{X}_{\mathsf{q}1} &\subseteq \bigcap_{i=1}^{n_x}\widetilde{\mathsf{Int}}_{d_i}(\mathcal{X}_{\mathsf{q}2}; \widetilde{\gamma}_i) = \bigcap_{i=1}^{n_x}\{{x}_{\mathsf{q}1} \in \mathcal{X}_{\mathsf{q}2}: \widetilde{\mathcal{B}}_{d_i}({x}_{\mathsf{q}1}; \widetilde{\gamma}_i) \subseteq \mathcal{X}_{\mathsf{q}2}\}.
\end{align}
Hence, from \req{intdefdeploydis}, given $x_{\mathsf{q}1} \in \mathcal{X}_{\mathsf{q}1}$, we then have
\begin{align}\label{balldefdeploydis}
\widetilde{\mathcal{B}}_{d_i}(x_{\mathsf{q}1}; \widetilde{\gamma}_i) \subseteq \mathcal{X}_{\mathsf{q}2},
\end{align}
for all $i \in \mathbb{N}_{1: n_x}$.
Moreover, recall that the ball set $\widetilde{\mathcal{B}}_{d_i}(x_{\mathsf{q}1}; \widetilde{\gamma}_i) \subseteq \mathcal{X}_\mathsf{q}$ centered at $x_{\mathsf{q}1} \in \mathcal{X}_\mathsf{q}$ is defined as $\widetilde{\mathcal{B}}_{d_i}(x_{\mathsf{q}1}; \widetilde{\gamma}_i) = \{{x}_{\mathsf{q}2} \in \mathcal{X}_\mathsf{q} : \ d_i({x}_{\mathsf{q}1}, {x}_{\mathsf{q}2}) \le \widetilde{\gamma}_i\}$. It follows, for all $x_{\mathsf{q}1}, x_{\mathsf{q}2} \in \mathcal{X}_\mathsf{q}$ with $d_i(x_{\mathsf{q}1}, x_{\mathsf{q}2}) \le \widetilde{\gamma}_i$, that 
\begin{align}\label{balldefdeploydis2}
    x_{\mathsf{q}2} \in \mathcal{B}_{d_i}({x}_{\mathsf{q}1}; \widetilde{\gamma}_i).
\end{align}
Therefore, combining \req{balldefdeploydis} and \req{balldefdeploydis2}, given $x_{\mathsf{q}1} \in \mathcal{X}_{\mathsf{q}1}$ and $x_{\mathsf{q}2} \in \mathcal{X}_\mathsf{q}$ satisfying $d_i(x_{\mathsf{q}1}, x_{\mathsf{q}2}) \le \widetilde{\gamma}_i$ for all $i \in \mathbb{N}_{1: n_x}$, follows that ${x}_{\mathsf{q}2} \in \mathcal{B}_{d_i}(x_{\mathsf{q}1}; \widetilde{\gamma}_i) \subseteq \mathcal{X}_{\mathsf{q}2}$. Furthermore, $d_i({x}_{\mathsf{q}1}, {x}_{\mathsf{q}2}) \le \widetilde{\gamma}_i$ are developed as
\begin{align}
d_i({x}_{\mathsf{q}1}, {x}_{\mathsf{q}2}) &= \sqrt{2\alpha_i^2\left\{1 - \exp\left(- \frac{1}{2}\|x_{\mathsf{q}1} - x_{\mathsf{q}2}\|^2_{\Lambda_{x, i}^{-1}}\right)\right\}} \notag \\ 
& \le \sqrt{2\alpha_i^2\left\{1 - \exp\left(- \frac{2\mathsf{z}^2_i}{n}\|{ \eta}_x\|^2_{\Lambda_{x, i}^{-1}}\right)\right\}},
\end{align}
for all $i \in \mathbb{N}_{1: n_x}$. Therefore, the fixed points $x_{\mathsf{q}1}, x_{\mathsf{q}2}$ satisfying $d_i({x}_{\mathsf{q}1}, {x}_{\mathsf{q}2}) \le \widetilde{\gamma}_i$ for all $i \in \mathbb{N}_{1: n_x}$ follows that
\begin{align}
|x_{\mathsf{q}1, i} - x_{\mathsf{q}2, i}| \le \mathsf{z}_i\frac{2}{\sqrt{n}} \eta_{x, i}.
\end{align}
for all $i \in \mathbb{N}_{1: n_x}$.
Hence, in the discrete state space $[\mathbb{R}^{n_x}]_{\eta_x}$, we then have
\begin{align}\label{outtransformdis}
&{x}_{\mathsf{q}1} \in \mathcal{X}_{\mathsf{q}1}, \ d_i({x}_{\mathsf{q}1}, {x}_{\mathsf{q}2}) \le \widetilde{\gamma}_i, \ \forall i \in \mathbb{N}_{1: n_x} \ \notag \\
&\Longleftrightarrow \ {x}_{\mathsf{q}2} \in \bigcup_{i = 1}^{n_x} \widetilde{\mathsf{Out}}_{d_i}(\mathcal{X}_{\mathsf{q}1}; \widetilde{\gamma}_i)
\end{align}
Therefore, combining \req{balldefdeploydis}, \req{balldefdeploydis2}, and \req{outtransformdis}, we obtain the result as
\begin{align}\label{outincludedis}
&\mathcal{X}_{\mathsf{q}1} \subseteq \bigcap_{i=1}^{n_x}\widetilde{\mathsf{Int}}_{d_i}(\mathcal{X}_{\mathsf{q}2}; \widetilde{\gamma}_i) \notag \\
&\Longleftrightarrow \ ( {x}_{\mathsf{q}1} \in \mathcal{X}_{\mathsf{q}1}, \ d_i({x}_{\mathsf{q}1}, {x}_{\mathsf{q}2}) \le \widetilde{\gamma}_i, \ \forall i \in \mathbb{N}_{1: n_x} \implies x_{\mathsf{q}2}\in \mathcal{X}_{\mathsf{q}2})  \notag \\
&\Longleftrightarrow \ ( {x}_{\mathsf{q}2} \in \bigcup_{i = 1}^{n_x} \widetilde{\mathsf{Out}}_{d_i}(\mathcal{X}_{\mathsf{q}1}; \widetilde{\gamma}_i) \implies \ {x}_{\mathsf{q}2} \in \mathcal{X}_{\mathsf{q}2}) \notag \\
&\Longleftrightarrow \ \bigcup_{i = 1}^{n_x} \widetilde{\mathsf{Out}}_{d_i}(\mathcal{X}_{\mathsf{q}1}; \widetilde{\gamma}_i) \subseteq \mathcal{X}_{\mathsf{q}2}.
\end{align}
This directly means that \rlem{intoutlemdis} holds.
\end{proof}

\subsection{Proof of \rlem{0asrlem}}\label{0asrlemproof}
\begin{proof}
Let $\Sigma_\mathsf{q} = (\mathcal{X}_\mathsf{q}, \mathcal{U}_\mathsf{q}, G_\mathsf{q})$ and ${\Sigma}_{\mathsf{q}\tilde{\gamma}} = (\mathcal{X}_{\mathsf{q}\tilde{\gamma}}, \mathcal{U}_{\mathsf{q}\tilde{\gamma}}, G_{\mathsf{q}\tilde{\gamma}})$ be the symbolic models in  \rdef{symbolictrans} and \rdef{symbolictransaug}, respectively. 
The sets of states and inputs for $\Sigma_\mathsf{q}$ and ${\Sigma}_{\mathsf{q}\tilde{\gamma}}$ are equal as $\mathcal{X}_\mathsf{q} = \mathcal{X}_{\mathsf{q}\tilde{\gamma}}$ and $\mathcal{U}_\mathsf{q} = \mathcal{U}_{\mathsf{q}\tilde{\gamma}}$, so for every $x_\mathsf{q} \in \mathcal{X}_\mathsf{q}$, there exits $x_{\mathsf{q}\tilde{\gamma}} \in \mathcal{X}_{\mathsf{q}\tilde{\gamma}}$ such that $x_{\mathsf{q}\tilde{\gamma}} = x_\mathsf{q}$, and it follows that $d_i(x_{\mathsf{q}\tilde{\gamma}}, x_{\mathsf{q}\tilde{\gamma}}) = 0$ for all $i \in \mathbb{N}_{1: n_x}$. Therefore, the condition (C.1) in \rdef{asrdef} is satisfied.
Moreover, the condition (C.2) in \rdef{asrdef} holds from the definition of $R(0)$ in \req{0asr}.
Let us now show that condition (C.3) holds. From the definition of the transition map $G_{\mathsf{q}\tilde{\gamma}}$ given by $G_{\mathsf{q}\tilde{\gamma}}({x}_{\mathsf{q}\tilde{\gamma}}, {u}_{\mathsf{q}\tilde{\gamma}}) = \bigcup_{i = 1}^{n_x} \mathsf{Out}_{\tilde{\gamma}_i}^{\mathcal{X}_{\mathsf{q}\tilde{\gamma}}}(G({x}_{\mathsf{q}\tilde{\gamma}}, {u}_{\mathsf{q}\tilde{\gamma}}))$, for every ${x}_{\mathsf{q}\tilde{\gamma}} = {x}_{\mathsf{q}} \in [\mathbb{R}^{n_x}]_{{ \eta}_x}$ and $u_{\mathsf{q}\tilde{\gamma}} = {u}_\mathsf{q} \in [\mathbb{R}^{n_u}]_{{ \eta}_u}$, we have $G_\mathsf{q}({x}_\mathsf{q}, {u}_\mathsf{q}) \subset G_{\mathsf{q}\tilde{\gamma}}({x}_{\mathsf{q}\tilde{\gamma}}, {u}_{\mathsf{q}\tilde{\gamma}})$. Hence, for every $x_\mathsf{q}^+ \in G_\mathsf{q}({x}_\mathsf{q}, {u}_\mathsf{q})$, there exits $x_{\mathsf{q}\tilde{\gamma}}^+ \in G_{\mathsf{q}\tilde{\gamma}}({x}_{\mathsf{q}\tilde{\gamma}}, {u}_{\mathsf{q}\tilde{\gamma}})$ satisfying $(x_{\mathsf{q}\tilde{\gamma}}^+, x_\mathsf{q}^+) \in R(0)$. This directly means that the condition (C. 3) in \rdef{asrdef} holds. 
Therefore, it is shown that the relation $R(0)$ in \req{0asr} is $0$-ASR from ${\Sigma}_{\mathsf{q}\tilde{\gamma}}$ to $\Sigma_\mathsf{q}$.
\end{proof}

\subsection{Proof of Proposition~2}
\begin{proof}
Suppose that \ralg{safetygame} is implemented with the inputs $\mathcal{X}$, $\mathcal{D}_N$, $\mathsf{q} = ({\eta}_x, { \eta}_u, \varepsilon)$, and $\widetilde{\gamma}$ in \req{contractiveparam} and the solution of the safety game is obtained as $\mathcal{Q}_{\ell - 1} = \mathcal{Q}_{\ell} = \mathcal{X}_{S, \mathsf{q}\tilde{\gamma}} \neq \varnothing$ for some $\ell \in \mathbb{N}_{>0}$. Then, from the notation of the transition map $G_{\mathsf{q}\tilde{\gamma}}$ in \rdef{symbolictransaug} and the operator ${\rm Pre}_{{\Sigma}_{\mathsf{q}\tilde{\gamma}}}$ in \req{predecessordef}, we have 
\begin{align}\label{discreatesafesettrans}
\mathcal{X}_{S, {\mathsf{q}\tilde{\gamma}}} &= {\rm Pre }_{{\Sigma}_{\mathsf{q}\tilde{\gamma}}}(\mathcal{X}_{S, {\mathsf{q}\tilde{\gamma}}})\notag\\
&=\{{x}_{\mathsf{q}\tilde{\gamma}} \in \mathcal{X}_{S, {\mathsf{q}\tilde{\gamma}}}: G_{\mathsf{q}\tilde{\gamma}}({x}_{\mathsf{q}\tilde{\gamma}}, {u}_{\mathsf{q}\tilde{\gamma}}) \subseteq \mathcal{X}_{S, {\mathsf{q}\tilde{\gamma}}}, \notag \\
& \ \ \ \ \ \ \ \ \ \ \ \ \ \ \ \ \ \ \ \ \ \ \ \ \ \ \ \ \ \ \ \ \ \ \ \ \ \ \ \ \exists {u}_{\mathsf{q}\tilde{\gamma}} \in \mathcal{U}_{\mathsf{q}\tilde{\gamma}}\} \notag \\
&= \{{x}_{\mathsf{q}\tilde{\gamma}} \in \mathcal{X}_{S, {\mathsf{q}\tilde{\gamma}}}: \bigcup_{i = 1}^{n_x}\widetilde{\mathsf{Out}}_{d_i}(G_\mathsf{q}({x}_{\mathsf{q}\tilde{\gamma}}, {u}_{\mathsf{q}\tilde{\gamma}}); \widetilde{\gamma}_i) \subseteq \mathcal{X}_{S, {\mathsf{q}\tilde{\gamma}}}, \notag \\
& \ \ \ \ \ \ \ \ \ \ \ \ \ \ \ \ \ \ \ \ \ \ \ \ \ \ \ \ \ \ \ \ \ \ \ \ \ \ \ \ \exists {u}_{\mathsf{q}\tilde{\gamma}} \in \mathcal{U}_{\mathsf{q}\tilde{\gamma}}\}.
\end{align}
Furthermore, using \rlem{intoutlemdis}, \req{discreatesafesettrans} is then computed as
\begin{align}\label{discreatesafeset}
\mathcal{X}_{S, {\mathsf{q}\tilde{\gamma}}}&= \{{x}_{\mathsf{q}\tilde{\gamma}} \in \mathcal{X}_{S, {\mathsf{q}\tilde{\gamma}}}: G_\mathsf{q}({x}_{\mathsf{q}\tilde{\gamma}}, {u}_{\mathsf{q}\tilde{\gamma}}) \subseteq \bigcap_{i = 1}^{n_x}\widetilde{\mathsf{Int}}_{d_i}(\mathcal{X}_{S, {\mathsf{q}\tilde{\gamma}}}; \widetilde{\gamma}_i), \notag \\
& \ \ \ \ \ \ \ \ \ \ \ \ \ \ \ \ \ \ \ \ \ \ \ \ \ \ \ \ \ \ \ \ \ \ \ \ \ \ \ \ \exists {u}_{\mathsf{q}\tilde{\gamma}} \in \mathcal{U}_{\mathsf{q}\tilde{\gamma}}\}.
\end{align}
Therefore, from \req{discreatesafeset}, we obtain the result that there exist control inputs $u_{\mathsf{q}\tilde{\gamma}} \in \mathcal{U}_{\mathsf{q}\tilde{\gamma}}$ such that all the $u_{\mathsf{q}\tilde{\gamma}}$-successors in the set $\mathcal{X}_{S, {\mathsf{q}\tilde{\gamma}}}$ are inside $\bigcap_{i = 1}^{n_x}\widetilde{\mathsf{Int}}_{d_i}(\mathcal{X}_{S, {\mathsf{q}\tilde{\gamma}}}; \widetilde{\gamma}_i)$. That is, for every $x_{\mathsf{q}\tilde{\gamma}} \in \mathcal{X}_{S, \mathsf{q}\tilde{\gamma}}$, there exists a controller $C_{{\mathsf{q}\tilde{\gamma}}}$ such that ${x}_{\mathsf{q}\tilde{\gamma}} \xrightarrow{{u}_{\mathsf{q}\tilde{\gamma}}} {x}_{\mathsf{q}\tilde{\gamma}}^+ \in \bigcap_{i = 1}^{n_x}\widetilde{\mathsf{Int}}_{d_i}(\mathcal{X}_{S, {\mathsf{q}\tilde{\gamma}}}; \widetilde{\gamma}_i)$, where $u_{\mathsf{q}\tilde{\gamma}} \in C_{{\mathsf{q}\tilde{\gamma}}}({x}_{\mathsf{q}\tilde{\gamma}})$. Moreover, using the $\varepsilon$-ASR from ${\Sigma}_{\mathsf{q}\tilde{\gamma}}$ to $\Sigma$, both $\mathcal{X}_{S}$ and $\bigcap_{i = 1}^{n_x} \mathsf{Int}_{d_i}(\mathcal{X}_{S}; \widetilde{\gamma}_i)$ can be computed from the fixed point set of $\mathcal{X}_{S, {\mathsf{q}\tilde{\gamma}}}$ and $\bigcap_{i = 1}^{n_x}\widetilde{\mathsf{Int}}_{d_i}(\mathcal{X}_{S, {\mathsf{q}\tilde{\gamma}}}; \widetilde{\gamma}_i)$. At the same time, the controller $C$ can be refined from the fixed point set of $C_{{\mathsf{q}\tilde{\gamma}}}$. Hence, all the states in $\mathcal{X}_S$ can drive into the set $\bigcap_{i=1}^{n_x}\mathsf{Int}_{d_i}(\mathcal{X}_S; \tilde{\gamma}_i)$ by applying the controller $C$. Therefore, $\mathcal{X}_S$ is a ${\widetilde{\gamma}}$-contractive set in $\mathcal{X}$. Furthermore, the terminal set $\mathcal{X}_f$ is defined as $\bigcap_{i=1}^{n_x}\mathsf{Int}_{d_i}(\mathcal{X}_S; \tilde{\gamma}_i) \subseteq \mathcal{X}_f$ in \req{terminalsetdef}. Thus, for every $x \in \mathcal{X}_S$, there exists a controller $C$ such that ${x} \xrightarrow{{u}} {x}^+ \in \mathcal{X}_f$, where $u \in C({x})$.
\end{proof}

\end{document}